\documentclass[11pt]{article}
\usepackage{amsmath}
\usepackage{amsthm}
\usepackage{amssymb,amsfonts}
\usepackage{hyperref}
\usepackage{latexsym}
\usepackage{todonotes}
\usepackage{nicefrac}
\usepackage{epsfig}
\usepackage{stmaryrd}
\usepackage{setspace}
\usepackage{tikz-cd,tikz,tikz-3dplot}
\usetikzlibrary{calc, matrix, arrows, decorations.pathmorphing}
\tikzset{%
	descr/.style={fill=white},
	baseline={([yshift=-\the\dimexpr\fontdimen22\textfont2\relax]
	                    current bounding box.center)},
}
\usepackage{enumerate}
\usepackage[all]{xypic}
\usepackage{bbm,ifpdf,tikz}
\ifpdf
\usepackage{pdfsync}
\fi

\newtheorem{theorem}{Theorem}[section]

\newtheorem{lemma}[theorem]{Lemma}
\newtheorem{proposition}[theorem]{Proposition}
\newtheorem{corollary}[theorem]{Corollary}

\newtheorem{conjecture}[theorem]{Conjecture}

{
\theoremstyle{definition}

\newtheorem{example}[theorem]{Example}

\newtheorem{remark}[theorem]{Remark}
}

\newcommand{\becircled}{\mathaccent "7017}
\newcommand{\excise}[1]{}

\newcommand{\id}{\operatorname{id}}

\renewcommand{\dim}{\operatorname{dim}}

\newcommand{\rk}{\operatorname{rk}}

\newcommand{\gr}{\operatorname{gr}}

\renewcommand{\and}{\qquad\text{and}\qquad}
\newcommand{\Ind}{\operatorname{Ind}}
\newcommand{\Res}{\operatorname{Res}}
\newcommand{\Hom}{\operatorname{Hom}}

\renewcommand{\Vec}{\operatorname{Vec}}
\newcommand{\uKL}{\operatorname{KL}}
\newcommand{\KL}{\Sigma}
\newcommand{\cusp}{\operatorname{cusp}}
\newcommand{\nbc}{\operatorname{nbc}}
\newcommand{\I}{\operatorname{I}}

\newcommand{\bV}{\mathbb{V}}
\newcommand{\bH}{\mathbb{H}}
\newcommand{\Pol}{\operatorname{Pol}}
\newcommand{\cPol}{\mathcal{P}}
\newcommand{\cPolid}{\mathcal{P}_{\operatorname{id}}}
\DeclareMathOperator{\Ch}{Ch}

\newcommand{\Z}{\mathbb{Z}}
\newcommand{\Q}{\mathbb{Q}}
\newcommand{\N}{\mathbb{N}}
\newcommand{\R}{\mathbb{R}}

\newcommand{\IH}{\operatorname{IH}}

\newcommand{\cK}{\mathcal{K}}
\renewcommand{\cR}{\mathcal{R}}
\newcommand{\cP}{\mathcal{P}}
\newcommand{\cG}{\mathcal{G}}
\newcommand{\cC}{\mathcal{C}}

\newcommand{\cI}{\mathcal{I}}

\renewcommand{\a}{\alpha}

\newcommand{\cA}{\mathcal{A}}
\newcommand{\cB}{\mathcal{B}}

\newcommand{\bk}{\textbf{k}}
\newcommand{\bS}{\textbf{S}}
\newcommand{\cF}{\mathcal{F}}
\newcommand{\tM}{\tilde{M}}
\newcommand{\GF}{\Gamma_{\!F}}
\newcommand{\GH}{\Gamma_{\!H}}

\newcommand{\OS}{\operatorname{OS}}

\renewcommand{\cL}{\mathcal{L}}
\newcommand{\fS}{\mathfrak{S}}
\newcommand{\bigmid}{\;\Big{|}\;}

\newcommand{\nicktodo}{\todo[inline,color=green!20]}
\def\BE#1{\textcolor[rgb]{.75,0.00,0.00}{[BE: #1]}}

\newcommand{\uH}{\underline{H}}
\newcommand{\Aut}{\operatorname{Aut}}

\newcommand{\co}{\colon}

\renewcommand{\deg}{\operatorname{deg}}
\newcommand{\Val}{\operatorname{Val}}

\newcommand{\cMat}{\mathcal{M}}
\newcommand{\cMatid}{\mathcal{M}_{\operatorname{id}}}

\newcommand{\grOS}{\gr\OS}

\DeclareMathOperator{\Nul}{Null}

\newcommand{\cS}{\mathcal{S}}
\newcommand{\cV}{\mathcal{V}}
\newcommand{\cQ}{\mathcal{Q}}

\newcommand{\cM}{\mathcal{M}}
\newcommand{\cN}{\mathcal{N}}

\newcommand{\Mat}{\operatorname{Mat}}

\newcommand{\CH}{\operatorname{CH}}
\newcommand{\uCH}{\underline{\CH}}
\renewcommand{\H}{H}
\newcommand{\Det}{\operatorname{Det}}

\newcommand{\namedto}[1]{\stackrel{#1}{\rightarrow}}

\newcommand{\pa}{\partial}

\DeclareMathOperator{\Cone}{Cone}
\DeclareMathOperator{\Cyl}{Cyl}
\renewcommand{\Bbbk}{\Q}

\newcommand{\cY}{\mathcal{Y}}

\newcommand{\revise}[1]{{#1}}

\begin{document}
\spacing{1.2}

\noindent{\Large\bf Categorical valuative invariants of polyhedra and matroids}\\

\noindent{\bf Ben Elias}\footnote{Supported by NSF grants DMS-2201387 and DMS-2039316.}\\
Department of Mathematics, University of Oregon, Eugene, OR 97403
\vspace{.1in}

\noindent{\bf Dane Miyata}\footnote{Supported by NSF grants DMS-2039316 and DMS-2053243.}\\
Department of Mathematics, University of Oregon, Eugene, OR 97403
\vspace{.1in}

\noindent{\bf Nicholas Proudfoot}\footnote{Supported by NSF grants DMS-1954050, DMS-2039316, and DMS-2053243.}\\
Department of Mathematics, University of Oregon, Eugene, OR 97403
\vspace{.1in}

\noindent{\bf Lorenzo Vecchi}\footnote{Partially supported by the National Group for Algebraic and Geometric Structures, and their Applications (GNSAGA - INdAM)}\\
KTH, Matematik, Lindstedsv\"{a}gen 25, 10044 Stockholm, Sweden\\

{\small
\begin{quote}
We introduce the notion of a categorical valuative invariant of polyhedra or matroids,
in which alternating sums of numerical invariants are replaced by split exact sequences in an additive category.
We provide categorical lifts of a number of valuative invariants of matroids, including the Poincar\'e polynomial,
the Chow and augmented Chow polynomials, and certain two-variable extensions of the Kazhdan--Lusztig polynomial
and $Z$-polynomial.  These lifts allow us to perform calculations equivariantly with respect to automorphism groups of matroids.
\end{quote} }

\section{Introduction}
Let $E$ be a finite set, and let $\Mat(E)$ be the free abelian group with basis given by matroids on $E$.  
An element of $\Mat(E)$ is said to be {\bf valuatively equivalent to zero} if it lies in the kernel of the
homomorphism that takes a matroid to the indicator function of its base polytope.
The following fundamental example will be revisited throughout the paper.

\begin{example}\label{octahedron 1}
Let $E = \{1,2,3,4\}$.  Let $M$ be the uniform matroid of rank 2 on $E$,
let $N$ be the matroid whose bases are all subsets of cardinality
2 except for $\{3,4\}$, let $N'$ be the matroid whose bases are all subsets of cardinality 2 except for $\{1,2\}$, and let
$N''$ be the matroid whose bases are all subsets of cardinality 2 except for $\{1,2\}$ and $\{3,4\}$.
In Figure \ref{octahedron figure}, the base polytope of $M$ is the octahedron, the base polytopes of $N$ and $N'$
are the two pyramids, and the base polytope of $N''$ is the square.  Thus $M - N - N' + N''$ is valuatively equivalent to
zero.
\end{example}

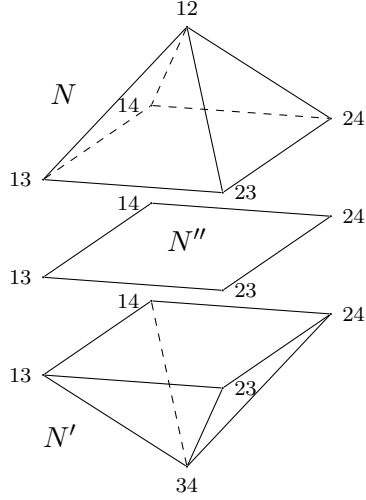
\begin{figure}[h]
\begin{center}
	\begin{tikzpicture}  
	[scale=0.65,auto=center,rotate around y=-10] 
	\tikzstyle{edges} = [thick];
    \node[style={circle,scale=0.8, fill=black, inner sep=0pt}, label=left:{\scriptsize 14}] (14) at (-2,0,-2) {};
    \node[style={circle,scale=0.8, fill=black, inner sep=0pt}, label=right:{\scriptsize 24}] (24) at (2,0,-2) {};
    \node[style={circle,scale=0.8, fill=black, inner sep=0pt}, label=right:{\scriptsize 23}] (23) at (2,0,2) {};
    \node[style={circle,scale=0.8, fill=black, inner sep=0pt}, label=left:{\scriptsize 13}] (13) at (-2,0,2) {};

    \node[style={circle,scale=0.8, fill=black, inner sep=0pt}, label=left:{\scriptsize 14}] (14N1) at (-2,2,-2) {};
    \node[style={circle,scale=0.8, fill=black, inner sep=0pt}, label=right:{\scriptsize 24}] (24N1) at (2,2,-2) {};
    \node[style={circle,scale=0.8, fill=black, inner sep=0pt}, label=right:{\scriptsize 23}] (23N1) at (2,2,2) {};
    \node[style={circle,scale=0.8, fill=black, inner sep=0pt}, label=left:{\scriptsize 13}] (13N1) at (-2,2,2) {};

    \node[style={circle,scale=0.8, fill=black, inner sep=0pt}, label=left:{\scriptsize 14}] (14N2) at (-2,-2,-2) {};
    \node[style={circle,scale=0.8, fill=black, inner sep=0pt}, label=right:{\scriptsize 24}] (24N2) at (2,-2,-2) {};
    \node[style={circle,scale=0.8, fill=black, inner sep=0pt}, label=right:{\scriptsize 23}] (23N2) at (2,-2,2) {};
    \node[style={circle,scale=0.8, fill=black, inner sep=0pt}, label=left:{\scriptsize 13}] (13N2) at (-2,-2,2) {};
    
    \node[style={circle,scale=0.8, fill=black, inner sep=0pt}, label=above:{\scriptsize 12}] (12) at (0,4.5,0) {};
    \node[style={circle,scale=0.8, fill=black, inner sep=0pt}, label=below:{\scriptsize 34}] (34) at (0,-4.5,0) {};

	 \node[label=left:$N$] at (-2,3,0) {};
	 \node[label=left:$N'$] at (-2,-4,0) {};
	 \node[label=$N''$] at (0,-0.5,0) {};

    \draw (13) -- (14);
    \draw (14) -- (24);
    \draw (24) -- (23);
    \draw (23) -- (13);

    \draw[dashed] (13N1) -- (14N1);
    \draw[dashed] (14N1) -- (24N1);
    \draw (24N1) -- (23N1);
    \draw (23N1) -- (13N1);

    \draw (13N2) -- (14N2);
    \draw (14N2) -- (24N2);
    \draw (24N2) -- (23N2);
    \draw (23N2) -- (13N2);
    
    \draw (12) -- (13N1);
    \draw[dashed] (12) -- (14N1);
    \draw (12) -- (24N1);
    \draw (12) -- (23N1);
    \draw (34) -- (13N2);
    \draw[dashed] (34) -- (14N2);
    \draw (34) -- (24N2);
    \draw (34) -- (23N2);
	
    \end{tikzpicture} 
\caption{A decomposition of the matroid $M = U_{2,4}$.
The label $ij$ refers to the point that takes the value 1 in the $i^\text{th}$ and $j^\text{th}$ coordinates,
such as $12 = (1,1,0,0)$.
}\label{octahedron figure}
\end{center}
\end{figure}

Let $A$ be an abelian group, and let $f:\Mat(E)\to A$ be any homomorphism.  This homomorphism is called {\bf valuative}
if it vanishes on elements that are valuatively equivalent to zero.  \revise{It is called an {\bf invariant}
if it is invariant under the action of the permutation group $\Aut(E)$ on $\Mat(E)$.}
Examples of valuative invariants of matroids include the following,
all of which take values in the group $A = \Z[t]$.
\begin{itemize}
\item The {\bf Poincar\'e polynomial} $\pi_M(t) = \sum t^i \dim\OS^i(M)$, 
where $\OS(M)$ is the Orlik--Solomon algebra of $M$.\footnote{The Poincar\'e polynomial
is closely related to the characteristic polynomial $\chi_M(t) = (-t)^{\rk M} \pi_M(-t^{-1})$.
We prefer the Poincar\'e polynomial because it has positive coefficients.}
This is a specialization of the Tutte polynomial $T_M(x,y)\in\Z[x,y]$, which is valuative by \cite[Lemma 6.4]{Speyer-trop}.
\item The {\bf Chow polynomial} $\uH_M(t) = \sum t^i \dim \uCH^i(M)$\revise{, which is valuative by} \cite[Theorem 8.14]{FSVal}.
Here $\uCH(M)$ is the Chow ring of $M$ \cite{FY}.
\item The {\bf augmented Chow polynomial} $\H_M(t) = \sum t^i \dim \CH^i(M)$\revise{, which is valuative by} \cite[Theorem 10]{fmsv-fpsac}. 
Here $\CH(M)$ is the augmented Chow ring of $M$ \cite{BHMPW1}.
\item The {\bf Kazhdan--Lusztig polynomial} $P_M(t)$ \revise{\cite{EPW}, which is valuative by \cite[Theorem 8.8]{ArdilaSanchez}}.
\item The {\bf \boldmath{$Z$}-polynomial} $Z_M(t)$ \revise{\cite{PXY}, which is valuative by \cite[Theorem 9.3]{FSVal}}.
\end{itemize}
\revise{We also consider one additional example, which takes values in the ring $\Z\langle x,y\rangle$ freely generated by two non-commuting variables $x$
and $y$.
\begin{itemize}
\item The {\bf \boldmath{$\cG$}-invariant} $\cG(M)$, 
introduced and shown to be valuative in \cite[Theorem 8.3]{Derksen}.
\end{itemize}}

Our goal in this paper is to promote the \revise{valuative relations among the polynomials appearing in the first five examples} to exact sequences
of graded vector spaces.  For the matroids in Example \ref{octahedron 1}, 
valuativity of the Poincar\'e polynomial tells us that we have the relation
$$\pi_M(t) - \pi_N(t) - \pi_{N'}(t) + \pi_{N''}(t) = 0.$$
We will prove that, after choosing orientations of the base polytopes of the four matroids, we obtain a canonical exact sequence (Theorem \ref{thm:OS})
\begin{equation}\label{OS oct}0\to \OS(M)\to \OS(N)\oplus \OS(N')\to \OS(N'')\to 0,\end{equation}
with similar exact sequences involving the Chow ring and augmented Chow ring
(Corollaries \ref{uCH val} and \ref{CH val}).  The story for the Kazhdan--Lusztig polynomial and $Z$-polynomial is similar but slightly more complicated:  we
introduce new bivariate polynomials $\tilde P_M(t,u)$ and $\tilde Z_M(t,u)$ with the property that $\tilde P_M(t,-1) = P_M(t)$ and $\tilde Z_M(t,-1) = Z_M(t)$,
and we interpret these polynomials as Poincar\'e polynomials of bigraded vector spaces that satisfy exact sequences analogous to that in Equation \eqref{OS oct} (Corollaries \ref{KL val} and \ref{Z val}).
\revise{Finally, we promote the corresponding valuative relation for $\cG$-invariants to a split-exact sequence in the free additive monoidal category on
two generators (Section \ref{sec:G}).}
Our results apply not only to the decomposition in Example \ref{octahedron 1}, but to arbitrary matroid decompositions \footnote{\revise{We point out that in the literature these are also called matroid subdivisions.}}, which are known to generate the group of all valuative equivalences in $\Mat(E)$ (Proposition \ref{mat-decomp}).\\

We have two motivations for this project, one philosophical and the other concrete.  The philosophical motivation
is that many of the valuativity results cited above are mysterious.  
Producing canonical exact sequences of graded vector spaces can be seen as a satisfying explanation.  
\revise{With the exception of the Poincar\'e polynomial, all of the above examples of valuative invariants
were proved to be valuative using a convolution operation of Ardila and Sanchez \cite[Theorem C]{ArdilaSanchez}.\footnote{\revise{The
reader is warned that, while the convolution of two valuative homomorphisms is again valuative,
the convolution of two invariants is typically no longer invariant.  The same warning applies at the categorical level.}}
Similarly, our categorical lifts of these examples are proved to be valuative by categorifying this convolution operation 
(Corollary \ref{convolution}), which we believe illuminates it.}
\excise{
One can prove that these various polynomials are valuative, but \revise{often the proof does not involve working on polytope decompositions directly, and therefore the proof itself does not shed light on why a function is expected to be valuative. One notable example, on which we will expand later in the paper, is \cite[Theorem C]{ArdilaSanchez}. This result allows to produce new valuative invariants by combining known valuations using a convolution product.}}

The concrete motivation is that it allows us to incorporate symmetries
of matroids into the theory of valuativity.  
\revise{While most matroids have no nontrivial symmetries, many very interesting matroids are highly symmetric,
including uniform matroids, matroids associated with finite Coxeter hyperplane arrangements, and matroids associated
with Steiner systems.}
The Orlik--Solomon algebra, the Chow ring, and the
augmented Chow ring all inherit actions of the symmetry group of $M$.  
\revise{For example, the Orlik--Solomon algebra
of the matroid associated with the type Coxeter arrangement of type $A_{n-1}$ is $S_n$-equivariantly isomorphic to the cohomology
ring of the configuration space of $n$ points in the plane.}
Similarly, the Kazhdan--Lusztig polynomial
and the $Z$-polynomial can be naturally lifted to ``equivariant'' polynomials whose coefficients are isomorphism
classes of representations of the symmetry group of $M$ \cite{GPY,PXY}.
In Example \ref{octahedron 1}, the dihedral group $D_4$ acts by symmetries of the square\footnote{The dihedral group $D_4$ is the subgroup of $\fS_4$ generated by $(12)$ and $(13)(24)$.}, preserving $M$ and $N''$
while permuting $N$ and $N'$.  With a small modification that accounts for the action of the group on the orientations of the various
polytopes involved, Equation \eqref{OS oct} can be regarded as an exact sequence in the category of graded representations
of $D_4$, and therefore allows us to relate the $D_4$-equivariant isomorphism class of $\OS(M)$ to those of the other terms
in the sequence.  The most general result along these lines, for arbitrary categorical valuative invariants, 
appears in Corollary \ref{the point}.

As a sample application, we compute the effect of relaxing a collection of stressed hyperplanes 
on the Orlik--Solomon algebra or the equivariant Kazhdan--Lusztig polynomial of a matroid (Corollaries \ref{OS relax} and
\ref{KL relax}), the latter of which recovers the main result of \cite{KNPV}.
A matroid that is related to a uniform matroid by a sequence of hyperplane relaxations is called {\bf paving}, so our corollaries provide explicit formulas for the 
Orlik--Solomon algebra and equivariant Kazhdan--Lusztig polynomial (as graded representations of the automorphism group) for any paving matroid.
The class of paving matroids is very large:  in particular, the probability that a random matroid is paving conjecturally goes to one as the size of the ground set goes to infinity \cite{MNWW}.
A matroid that is related to a uniform matroid by a sequence of arbitrary relaxations is called {\bf split} \cite{JS, FSVal}, so our corollaries in fact provide a method for performing equivariant
calculations of any of the invariants discussed above for any split matroid (Proposition \ref{virtual}), provided that one can compute it for certain special matroids of the form $\Pi_{r,k,\revise{F,E}}$ and $\Lambda_{r,k,\revise{F,E}}$ \revise{(defined in Section \ref{sec:relaxation})}.
\\

\revise{To formalize the properties shared by Orlik-Solomon algebras, Chow rings, and the other examples discussed above, we introduce the notion of a
{\bf valuative functor}.
There is a category $\cMat(E)$ whose objects are matroids on the ground set $E$ and whose morphisms are weak maps (including those represented by permutations of $E$), and a subcategory $\cMatid(E)\subset \cMat(E)$ consisting only of those weak maps represented by the identity map on $E$.
Decompositions like the one in Example \ref{octahedron 1} give rise to complexes in the additive closure of $\cMatid(E)$, and we
think of these complexes as being valuatively equivalent to zero. We call a functor $\Phi$ from $\cMatid(E)$ to an additive category $\cA$ {\bf valuative} if such complexes 
are sent to split-exact complexes in $\cA$. Taking the map on Grothendieck groups induced by $\Phi$, one obtains a valuative homomorphism from
$\Mat(E)$ to the split Grothendieck group\footnote{This is the abelian group with generators given by isomorphism classes of objects of $\cA$, and relations of the form
$[A\oplus B] - [A] - [B]$ for all objects $A$ and $B$.} of $\cA$. 
We say that the functor {\bf categorifies} the homomorphism.  If the functor extends to $\cM(E)$, then we call it a {\bf categorical valuative invariant},
and the homomorphism that it categorifies is also an invariant.
For example, there is a functor $\OS$ from $\cMat(E)$ to the category of graded vector spaces over $\Q$ that takes a matroid on $E$ to its
Orlik--Solomon algebra, and a weak map to the corresponding surjective algebra homomorphism.
This functor is valuative (Theorem \ref{thm:OS}), and it categorifies the Poincar\'e polynomial.}

\excise{
\nicktodo{I'm not convinced that this is the ideal place for the next paragraph, which seems to me to break the flow.}
\revise{We note that there can be more than one way to categorify a given valuative invariant. For example, our proof that $\OS$ is valuative runs along these lines: we prove that $\OS$ has a filtration, and that the associated graded functor is valuative. It is a straightforward spectral sequence argument to deduce that $\OS$ is valuative. This associated graded functor is a second categorification of the Poincar\'e polynomial. In many ways the associated graded functor is ``worse'' than $\OS$ itself: it is functorial for fewer maps (technically, this is a functor from a subcategory of $\cMat(E)$); there are morphisms which $\OS$ evaluates to nonzero linear maps which the associated graded functor annihilates \BE{is this true, or am I mixing it up with another similar issue}. It is a major technical asset that we can include such functors within our theory of valuativite functors. However, it also helps to be sensitive to the notion that some valuative functors are ``better'' than others, see \BE{ref to remarks}.}
}

In fact, we work in a broader setting than that of matroids.  We define a category $\cP(\bV)$ whose objects are polyhedra in a real vector space $\bV$ 
and whose morphisms are linear automorphisms of $\bV$ that induce
inclusions of polyhedra, along with a subcategory $\cPolid(\bV)$ consisting of morphisms given by the identity map on $\bV$.  
Then $\cMat(E)$ is isomorphic to the subcategory of $\cPol(\R^E)$ whose objects are base polytopes of matroids on $E$ and whose morphisms are induced by permutations of $E$, and $\cMatid(E)$ is the intersection of $\cMat(E)$ with $\cPolid(\R^E)$.
The notions of valuative equivalence, valuative homomorphisms and invariants, and valuative functors and categorical valuative invariants 
all generalize naturally from matroids to polyhedra, and much of what we do takes place in this more general framework.
\\

The most important tool developed in this paper is a method of combining simple categorical functors to obtain more complicated ones.
We begin with a brief review of the non-categorical story.  Let $\psi$ be a linear functional on the real vector space $\bV$.  For any polyhedron $P\subset\bV$,
let $P_\psi$ be the face of $\bV$ on which $\psi$ is maximized if such a face exists, and \revise{the empty polyhedron} if the restriction of $\psi$ to $P$ is unbounded. 
McMullen \cite[Theorem 4.6]{McMullen} proves that the assignment $P\mapsto P_\psi$ preserves valuative equivalence.
Now suppose that $E = E_1\sqcup E_2$, and $\psi$ is the linear functional on $\R^E$ that takes the sum of the coordinates corresponding to elements of $E_1$.
Then for any matroid $M$, $P(M)_\psi = P(M_1) \times P(M_2)$, where $M_1$ is a matroid on $E_1$ and $M_2$ is a matroid on $E_2$ (Lemma \ref{break}).
Combining this observation with McMullen's theorem, one can define an operation that takes a pair of valuative homomorphisms $f_i:\Mat(E_i)\to\Z[t]$
to a new valuative homomorphism $f_1*f_2:\Mat(E)\to\Z$, called the {\bf convolution} of $f_1$ and $f_2$.  This construction is due to Ardila and Sanchez, who give
a proof of valuativity that is independent from McMullen's result \cite[Theorem C]{ArdilaSanchez}.

In this paper, we categorify everything in the previous paragraph.  
The categorification of McMullen's theorem is Theorem \ref{everything}, and this is the most difficult result that we prove.
With some additional work, we prove Theorem \ref{hopf} and Corollary \ref{convolution}, which together categorify \cite[Theorems A and C]{ArdilaSanchez}.  The end result is a categorical convolution product that allows us to combine a categorical functor for matroids
on $E_1$ with a categorical functor for matroids on $E_2$ to obtain a categorical functor for matroids on $E_1\sqcup E_2$.
It is via this construction that we categorify the valuative invariants $\uCH_M(t)$, $\CH_M(t)$, $\tilde P_M(t,u)$, $\tilde Z_M(t,u)$, \revise{and $\cG(M)$}.
\\

We would like to highlight two additional results that we prove along the way.
First, we use our categorifications to prove that the Chow polynomial and augmented polynomial are {\bf monotonic}, meaning that their coefficients weakly decrease along rank-preserving weak maps (Corollary \ref{CH mon}).  We conjecture that an analogous statement holds for the Kazhdan--Lusztig polynomial and $Z$-polynomial (Conjecture \ref{P and Z mon}).
\revise{Second, we introduce the {\bf valuative category} $\cV(E)$, a triangulated category that is in some sense the universal source
for valuative functors (Section \ref{sec:valfun}), and we prove that the Grothendieck group of $\cV(E)$ is isomorphic
to the {\bf valuative group} $V(E)$, the universal source of valuative homomorphisms (Proposition \ref{prop:groth}).}\\

\revise{Finally, we note that there are typically
many different ways to categorify a given valuative homomorphism with a valuative functor.  For example, as a tool for proving that the Orlik--Solomon functor $\OS$
is valuative, we introduce an ``associated graded'' functor $\gr\OS$ (Section \ref{sec:degenerating}), which is valuative and also categorifies the Poincar\'e polyomial.
Though there are reasons to study both $\OS$ and $\gr\OS$, we generally regard $\OS$ as the nicer mathematical object.  (In particular,
it is a categorical invariant, while $\gr\OS$ is not:  the functor $\gr\OS$ does not allow nontrivial permutations of the ground set.)  A similar story holds for the Chow and augmented Chow functors:  there are associated graded
versions that implicitly show up in the proof of valuativity (see in particular Proposition \ref{Whitney} and Lemma \ref{g val}), but those functors are
less appealing.  (For example, they could not be used to prove monotonicity as in Corollary \ref{CH mon}, as we observe in Remark
\ref{rmk:monotonicbecauseyoudiditright}.)
Thus, our aim is not only to categorify as many invariants as possible, but also to categorify them in the nicest possible ways.
}

\vspace{\baselineskip}
\noindent
{\em Acknowledgments:}
The authors are grateful to Matt Larson for explaining the connection between \cite{McMullen} and \cite{ArdilaSanchez}
and for his contributions to Section \ref{sec:groth},
to George Nasr for his contributions to Section \ref{sec:monotonicity}, and to Kris Shaw for explaining the argument
in Remark \ref{shaw}.

\section{Additive homological algebra}
We begin with a review of the notions from homological algebra that we will need. \revise{Most of this material can be found in \cite{Weibel} or \cite[Chapter 19]{Elias-book}.}  The experienced reader can skip this
section and refer back to it as needed.

\subsection{Basics}
Let $\cA$ be a $\Q$-linear additive category.  We write $\Ch(\cA)$
to denote the additive category of chain complexes $(C_\bullet,\partial)$ in $\cA$, with the homological convention that differential $\partial$ decreases degree by one.  We will sometimes drop the differential from the notation and simply write $C_\bullet\in\Ch(\cA)$,
though the differential is always part of the data.
Morphisms in $\Ch(\cA)$ are chain maps between complexes.  Let $\cK(\cA)$ denote the {\bf homotopy category} of $\cA$, 
which is the triangulated category obtained as the quotient of $\Ch(\cA)$ by the ideal
of null-homotopic chain maps. 
We write $\Ch_b(\cA)$ and $\cK_b(\cA)$ to denote the full subcategories whose objects are bounded complexes.

We note that $\cK_b(\cA)$ can also be viewed as the quotient of $\Ch_b(\cA)$ by a class of objects. For an object $X \in \cA$ and an integer $k \in \Z$, one can consider the complex $\Nul(X,k)$
consisting only of two copies of $X$ in degrees $k$ and $k-1$, with differential the identity map. A bounded complex is called {\bf contractible} if it is isomorphic to a finite direct sum of objects of the form $\Nul(X,k)$. Then a chain map between bounded complexes is null-homotopic if and only if it factors through a contractible complex. Thus the ideal of null-homotopic maps is the same as the ideal generated by the identity maps of $\Nul(X,k)$ for various $X$ and $k$. This is an old perspective, but the first author learned it from \cite{hopfological}.

\begin{remark} The same statements cannot be made for unbounded complexes. Indeed, a null-homotopic chain map may be nonzero in infinitely many degrees, requiring an expression
using an infinite sum of chain maps which factor through various $\Nul(X,k)$. \end{remark}

\begin{remark} If $\cA$ is a semisimple abelian category, 
then a complex in $\Ch_b(\cA)$ is
contractible if and only if it is exact, i.e. its homology is trivial. 
\end{remark}


\begin{remark} A bounded complex is homotopy equivalent to the zero complex if and only if it is contractible. More generally, two bounded complexes $C_{\bullet}$ and $D_{\bullet}$ are homotopy
equivalent if and only if there are bounded contractible complexes $X_{\bullet}$ and $Y_{\bullet}$ such that $C_{\bullet} \oplus X_{\bullet} \cong D_{\bullet} \oplus Y_{\bullet}$.
\end{remark}

Given an arbitrary category $\cC$, we write $\cC^+$ for the $\Bbbk$-linear additive closure of $\cC$. Objects in $\cC^+$
are formal direct sums of objects in $\cC$. If $X$ and $Y$ are objects in $\cC$, then $\Hom_{\cC^+}(X,Y)$ is the vector space over $\Bbbk$ with basis given by the set
$\Hom_{\cC}(X,Y)$. Similarly, morphisms between formal direct sums are matrices of $\Bbbk$-linear combinations of morphisms in $\cC$.

Given two $\Bbbk$-linear additive categories $\cA_1$ and $\cA_2$, 
the {\bf Deligne tensor product} of $\cA_1$ and $\cA_2$ is defined as follows. First, we define an intermediate category
whose objects are symbols $X_1 \boxtimes X_2$ for $X_1 \in \cA_1$ and $X_2 \in \cA_2$, with morphisms from $X_1 \boxtimes X_2$
to $Y_1 \boxtimes Y_2$ given by the tensor product $$\Hom_{\cA_1}(X_1, Y_1) \otimes_{\Bbbk} \Hom_{\cA_2}(X_2, Y_2).$$
This intermediate category is not yet additive, because we cannot take direct sums of objects.  The Deligne tensor product 
$\cA_1 \boxtimes \cA_2$ is defined to be the additive closure of this intermediate category.
There is an external tensor product operation $\Ch_b(\cA_1) \boxtimes \Ch_b(\cA_2) \to \Ch_b(\cA_1 \boxtimes \cA_2)$, 
which mimics the usual tensor product of complexes (with its Koszul sign rule).
When $\cA_1 = \cC_1^+$ and $\cA_2 = \cC_2^+$, then $\cA_1 \boxtimes \cA_2 = (\cC_1 \times \cC_2)^+$. 

\subsection{Cones}\label{sec:cones}
We let $[1]$ denote the usual homological shift on complexes, so that $C[1]_i = C_{i+1}$
and differentials are negated. For an object $X \in\cA$, let $X[-i]$ denote the complex consisting of $X$ concentrated in degree $i$. There is a natural inclusion of $\cA$ into $\Ch(\cA)$ that sends $X$ to $X[0]$.

\excise{
Instead of writing out a complex in the space-consuming form
\[ C_{\bullet} = \left( \ldots \to C_1 \namedto{\pa_1} C_0 \namedto{\pa_0} C_{-1} \to \ldots \right), \]
where $C_i$ appears in homological degree $i$, it is often convenient to use more compact notation. We may write the same complex as \BE{improve style}
\[ C_{\bullet} = (\bigoplus_{i} C_i[-i], \pa), \qquad \pa = \left( \begin{array}{ccccc} \ddots & & & & \\ & 0 & 0 & 0 & \\ & \pa_1 & 0 & 0 & \\ & 0 & \pa_0 & 0 & \\ & & & & \ddots \end{array} \right). \]
We think of $\bigoplus_{i} C_i[-i]$ as the graded object of $\cA$ underlying the complex $C$ (i.e. an object in the graded closure of $\cA$), and $\pa = \sum \pa_i$ as the {\bf total differential}, written in matrix form. For example, one has
\[ \Nul(X,k) = (X[-k] \oplus X[1-k], \left( \begin{array}{cc} 0 & 0 \\ \id_X & 0 \end{array} \right) ). \]

In this context we might abusively write $C_{\bullet} = (C_{\bullet},\pa_C)$ and let the same symbol $C_{\bullet}$ refer to both the complex and the underlying graded object. The latter interpretation is only valid within the parentheses when paired with a differential, or when we explicitly state that we view $C_{\bullet}$ as a graded object. For example, if $C_{\bullet}$ is a complex, then $(C_{\bullet},0)$ is the complex with the same chain objects, but with the zero differential. For sanity, we omit the bullet in the subscript for the total differential.}

Let $f \colon C_{\bullet} \to D_{\bullet}$ be a chain map. The {\bf cone} of $f$ is the complex 
\[ \Cone(f) := \left(C_{\bullet}[-1] \oplus D_{\bullet}, \left( \begin{array}{cc} -\pa_C & 0 \\ f & \pa_D \end{array} \right) \right). \]
More explicitly, $\Cone(f)_i := C_{i-1} \oplus D_i$, and for $c\in C_{i-1}$ and $d\in D_i$, \revise{$\partial(c,d) := \left(-\partial_C(c), f(c) + \partial_D(d)\right)$.}

A {\bf termwise-split short exact sequence} of complexes in $\cA$ is a collection of complexes and chain maps
\begin{equation}\label{tsses}0 \to P_{\bullet} \to Q_{\bullet} \to R_{\bullet} \to 0\end{equation}
with the property that, in each homological degree $i$, the sequence
\[ 0 \to P_i \to Q_i \to R_i \to 0\]
is split exact in $\cA$.  For any chain map $f \colon C_{\bullet} \to D_{\bullet}$, one has a termwise-split short exact sequence
\begin{equation} \label{coneses} 0 \to D_{\bullet} \to \Cone(f) \to C_{\bullet}[-1] \to 0.\end{equation}
Conversely, for any termwise-split short exact sequence as in Equation \eqref{tsses}, 
there exists a chain map $f \colon R_{\bullet}[1] \to P_{\bullet}$ and an isomorphism $Q_{\bullet} \cong \Cone(f)$.


\excise{
\begin{remark} \label{rmk:cylinder} If $\iota\colon D_{\bullet} \to \Cone(f)$ is the canonical map in \eqref{coneses}, then the cone of $\iota$ is called the {\bf cylinder} of $f$, and denoted $\Cyl(f)$. There is a termwise-split short exact sequence of complexes
\begin{equation*} \label{cylinderses} 0 \to \Cone(f) \to \Cyl(f) \to D_{\bullet}[-1] \to 0. \end{equation*}
There is always a homotopy equivalence $\Cyl(f) \cong C_{\bullet}$, but there need not exist a termwise-split short exact sequence of the form $0 \to \Cone(f) \to C_{\bullet} \to D_{\bullet}[-1] \to 0$. \end{remark}
}

The following lemma is well-known, and admits an elementary proof.
\revise{It implies, as a special case, that the cone of a map between contractible complexes is itself contractible.}

\begin{lemma}\label{lem:conestuff} Let $f \colon C_\bullet \to D_\bullet$ be a map of bounded chain complexes.
\begin{itemize}
\item The map $f$ is a homotopy equivalence if and only if $\Cone(f)$ is contractible.
\item The map $f$ is null-homotopic if and only if the termwise-split short exact sequence \eqref{coneses} 
splits at the level of complexes.
\end{itemize} 
\end{lemma}

	


\excise{
\begin{remark}
Lemma \ref{lem:conestuff} implies that contractible complexes are projective: whenever they appear as the third term 
in a termwise-split short exact sequence of complexes, then that short exact sequence is genuinely split. 
In particular, the cone of a map between contractible complexes is itself contractible. 
\end{remark}
}

An iterated cone is often called a {\bf convolution}, which is the additive analogue of a filtered complex. For example, if $(A_\bullet,\pa_A)$, $(B_\bullet,\pa_B)$, and $(C_\bullet,\pa_C)$ are complexes, then a
complex of the form $D_\bullet = (A_\bullet \oplus B_\bullet \oplus C_\bullet, \pa)$ is a (three-part) convolution if $\pa$ is lower triangular and agrees with $(\pa_A, \pa_B, \pa_C)$ along the diagonal. If so, then $C_\bullet$
is a termwise-split subcomplex of $D_\bullet$, $A_\bullet$ is a termwise-split quotient complex of $D_\bullet$, and $B_\bullet$ is a 
termwise-split subquotient of $D_\bullet$. One can describe $D_\bullet$ as the cone of a chain map from $A[1]_\bullet$ to $E_\bullet$,
where $E_\bullet$ is the cone of a chain map from $B[1]_\bullet$ to $C_\bullet$. We call $A_\bullet$, $B_\bullet$, and $C_\bullet$ the {\bf parts} of the convolution $D_\bullet$. 


A functor $\Phi$ between additive categories is called {\bf additive} if it preserves addition of morphisms, or
(equivalently) if it preserves direct sum decompositions of objects.  
Additive functors extend to the category of complexes (and descend to the homotopy category), where they 
preserve cones and convolutions.
If $\Phi \colon \cC \to \cA$ is any functor from an arbitrary category $\cC$ to a $\Q$-linear additive category $\cA$, then it extends naturally to
an additive functor $\cC^+ \to \cA$, which we also denote by $\Phi$.

\excise{
\begin{remark} \label{rmk:triangulated} Short exact sequences are powerful tools, which is why abelian categories are so beloved. The main problem is that functors typically do not
preserve short exact sequences; only exact functors do. Triangulated categories were introduced to bridge this gap, the main example being $\cK_b(\cA)$. Triangulated categories
have a collection of {\bf distinguished triangles}, which are triples of objects $(P_{\bullet}, Q_{\bullet}, R_{\bullet})$ with morphisms $R_{\bullet}[1] \to P_{\bullet} \to
Q_{\bullet} \to R_{\bullet}$ satisfying some axioms. Distinguished triangles are analogous to short exact sequences. In $\cK_b(\cA)$, the distinguished triangles are the ones of
the form $C_{\bullet} \to D_{\bullet} \to \Cone(f) \to C_{\bullet}[-1]$, where the first map is $f$. A functor between triangulated categories is {\bf triangulated} if it
preserves distinguished triangles; they are the analogue of exact functors. Any additive functor between additive categories induces a triangulated functor between their homotopy
categories. \end{remark}
}

\subsection{Localizing subcategories}
A nonempty full subcategory $\cI$ of $\Ch_b(\cA)$ is called {\bf localizing} if it is closed under homotopy equivalences, shifts, cones, and direct summands.
Contractible complexes form the smallest localizing subcategory.
Localizing subcategories are like ideals: they are the ``kernels'' of
triangulated functors. More precisely, consider an additive functor $\Phi \co \cA \to \cA'$, which induces a functor $\Ch_b(\cA) \to \Ch_b(\cA')$. Let $\cI \subset \Ch_b(\cA)$ be the full subcategory
consisting of complexes $C_{\bullet}$ with $\Phi(C_{\bullet})$ being contractible; then $\cI$ is a localizing subcategory. Conversely, given a localizing subcategory $\cI$ of
$\Ch_b(\cA)$, the quotient category $\Ch_b(\cA)/\cI$ is triangulated, and $\cI$ is the kernel of the quotient functor.

\begin{remark}
Because of Lemma \ref{lem:conestuff}, there is a relationship between inverting morphisms and killing objects:  formally
inverting a chain map $f$ is equivalent to killing the object $\Cone(f)$, and killing an object $C$ is equivalent to formally inverting 
the zero map $0 \to C$. Thus the quotient
category $\Ch_b(\cA)/\cI$ can also be obtained by inverting morphisms whose cones live in $\cI$.
\end{remark}

\begin{remark} Localizing subcategories in the literature are typically defined within the triangulated category $\cK_b(\cA)$,
defined in the same way.  Since all localizing subcategories of $\Ch_b(\cA)$ contain all contractible objects,
there is a natural quotient-preserving bijection between localizing subcategories of $\Ch_b(\cA)$ and localizing subcategories of $\cK_b(\cA)$.
\end{remark}

Localizing subcategories satisfy the {\bf two-out-of-three rule}: if $0 \to P_{\bullet} \to Q_{\bullet} \to R_{\bullet} \to 0$ is a termwise-split short exact sequence, and two out of three of
the complexes $P_{\bullet}$, $Q_{\bullet}$, $R_{\bullet}$ live in a localizing subcategory $\cI$, then so does the third. 

\begin{lemma} \label{lem:convoinideal} Let $\cI$ be a localizing subcategory, and let $X_{\bullet}$ a complex built as a convolution. If all the parts of $X_{\bullet}$ are in $\cI$, then $X_{\bullet}$ is also in $\cI$. If $X_{\bullet}$ is in $\cI$ and all but one part is in $\cI$, then the remaining part is also in $\cI$. \end{lemma}

\begin{proof} This is an iterated application of the two-out-of-three rule. \end{proof}

Given a nonempty collection $\cY$ of complexes in $\Ch_b(\cA)$, there is a smallest localizing subcategory $\langle \cY \rangle$ containing those complexes. It contains precisely those
complexes homotopy equivalent to convolutions whose parts are shifts of direct summands of complexes in $\cY$.

\subsection{Thin categories and minimal complexes}\label{thin}
A category $\cC$ is called {\bf thin} if, for all objects $X$ and $Y$, there is at most one morphism from $X$ to $Y$.
Thin categories are also called {\bf poset categories}, as there is a natural partial order on isomorphism classes of objects given by putting $X\leq Y$ if and only if there exists
a morphism from $X$ to $Y$, and this partial order determines $\cC$ up to equivalence.  In this section, we assume that $\cC$ is thin, 
and we let $\cA = \cC^+$.
For any object $X$ of $\cC$, let $\cA^{<X}$ (respectively $\cA^{\leq X}$) be the full subcategory of $\cA$ consisting of direct sums of objects that are strictly less (respectively less than or equal to) $X$.

Let $C_\bullet$ be an object in $\Ch_b(\cA)$.  For each object $X$ of $\cC$, we have a termwise-split short exact sequence
$$0\to C^{<X}_\bullet \to C^{\leq X}_\bullet \to C^{X}_\bullet \to 0,$$
where $C^{<X}_\bullet$ (respectively $C^{\leq X}_\bullet$)  is the maximal termwise-split subcomplex of $C_\bullet$ whose underlying object lies in $\cA^{<X}$ (respectively $\cA^{\leq X}$),
and $C^{X}_\bullet$ is the termwise-split quotient of $C^{\leq X}_\bullet$ by $C^{<X}_\bullet$.
Then $C_\bullet$ is a convolution with parts $\{C_\bullet^X\mid X\in\cC \text{ and } C_\bullet^X \neq 0\}$.

A complex in $\Ch_b(\cA)$ is called {\bf minimal} if, for all $X\in\cC$, the $(X,X)$ component of the differential is trivial.  In other words, the differential is required to be strictly upper triangular with respect to the partial order
on objects of $\cC$. Because $\cC$ is thin, this definition of minimality is equivalent to other definitions in the literature, for example the absence of contractible summands. The category $\cA$ satisfies the Krull--Schmidt property. Consequently, any complex is homotopy equivalent to a minimal complex, and that minimal complex is unique up to isomorphism \cite[Lemma 19.15]{Elias-book}.
If $D_\bullet^X$ is the minimal complex of $C_\bullet^X$ (which will necessarily have trivial differential), then the minimal complex of $C_\bullet$ is a convolution with parts $\{D_\bullet^X\mid X\in\cC\}$.

\section{Decompositions}\label{sec:decomp}
We next review the literature that we will need on decompositions of polyhedra and matroids.

\subsection{Decompositions of polyhedra}\label{sec:decompositions}
Let $\bV$ be a finite dimensional real vector space.
A {\bf polyhedron} in $\bV$ is a subset of $\bV$ obtained by intersecting finitely many closed half-spaces.  A bounded polyhedron is called a {\bf polytope}. 
Given a polyhedron $P$, we denote its dimension by $d(P)$.

Let $\Pol(\bV)$ be the free abelian group with basis given by polyhedra in $\bV$.
Let $\I(\bV)\subset \Pol(\bV)$ be the kernel of the homomorphism from $\Pol(\bV)$ to the group of $\Z$-valued functions
on $\bV$ taking a polyhedron $P$ to its indicator function $1_P$.
For any abelian group $A$, a homomorphism $\Pol(\bV)\to A$ is called {\bf valuative} if it vanishes on $\I(\bV)$.

The subgroup $\I(\bV)\subset\Pol(\bV)$ admits a concrete presentation, which we now describe.  Let $P$ be a polyhedron in $\bV$ of dimension $d$.
A {\bf decomposition} of $P$ is a collection $\cQ$ of polyhedra in $\bV$ with the
following properties:
\begin{itemize}
\item If $Q\in\cQ$, then every nonempty face of $Q$ is in $\cQ$.
\item If $Q, Q'\in \cQ$, then $Q\cap Q'$ is a (possibly empty) face of both $Q$ and $Q'$.
\item We have $\displaystyle P = \bigcup_{Q\in\cQ} Q$.
\end{itemize}

Elements of $\cQ$ are called {\bf faces} of the decomposition.  We say that a face $Q\in\cQ$ is {\bf internal} if $Q$ is not contained in the boundary
of $P$.  For all $k\leq d = d(P)$, we write $\cQ_k$ to denote the set of internal faces of dimension $k$.
Note that $\cQ \ne \bigcup_{k=0}^d \cQ_k$. We also write $\cQ_{d+1} := \{P\}$. 

\begin{remark}
If $\cQ$ is a decomposition of $P$, then $P$ is typically not a face of $\cQ$, so $\cQ_{d+1}$ plays a fundamentally different
role from $\cQ_k$ for $k\leq d$.  Our use of this potentially confusing notation is motivated by the expression in 
Equation \eqref{alternatingsum} below.
In the special case where $P \in \cQ$, then $\cQ$ is precisely the set of faces of $P$; this is called the {\bf trivial decomposition}. In this case $P$ is the \emph{only} internal face of $\cQ$, and we have $\cQ_d = \cQ_{d+1} = \{P\}$, and $\cQ_k = \varnothing$ otherwise.
\end{remark}

\excise{
\BE{I would like to replace this last sentence with (in a new paragraph):} In contrast to the definition of $\cQ_k$ above, we define $\cQ_{d+1}$ to be the one-element set consisting of the polyhedron $P$. This element $P \in \cQ_{d+1}$ is not a face of $\cQ$, though it will appear in expressions like \eqref{alternatingsum} below. It is possible that $P$ is itself an internal face of $\cQ$, in which case $P$ also appears as an element of $\cQ_d$, see Example \ref{triv decomp}.
}

\begin{proposition}\label{decomp}
If $\cQ$ is a decomposition of $P$, then
\begin{equation} \label{alternatingsum} \sum_k (-1)^k \sum_{Q\in \cQ_k} Q\in \I(\bV).\end{equation}
Furthermore, $\I(\bV)$ is spanned by elements of this form.
\end{proposition}

\begin{proof}
We first show that this expression is contained in $\I(\bV)$.
For decompositions of matroid polytopes, this is proved in \cite[Theorem 3.5]{AFR},
and the same argument holds verbatim for arbitrary polytopes.  That argument does not immediately generalize to unbounded polyhedra, but 
we will show how to use the bounded case to deduce the general case.  

We need to show that
the function \begin{equation}\label{expression1}\sum_k (-1)^k \sum_{Q\in \cQ_k} 1_Q\end{equation}
evaluates to zero at an arbitrary point $v\in\bV$.  Let $v$ be given, and choose a polytope $R$ containing $v$ such that,
for all $Q\in \cQ$, $Q\cap R$ is a nonempty polytope of the same dimension as $Q$.  (For example, we can take $R$ to be a very large box centered at $v$.)
Let $\tilde\cQ$ be the decomposition of $Q\cap R$ consisting of $Q\cap F$ for all $Q\in \cQ$ and $F$ a face of $R$ (possibly equal to $R$ itself).
Then intersection with $R$ provides a dimension-preserving bijection from internal faces of $\cQ$ to internal faces of $\tilde\cQ$.
Since $P\cap R$ is a polytope, the function
\begin{equation}\label{expression2}
\sum_k (-1)^k \sum_{\tilde Q\in \tilde \cQ_k} 1_{\tilde Q}
= \sum_k (-1)^k \sum_{Q\in \cQ_k} 1_{Q\cap R}\end{equation} is identically zero.
Since $v\in R$, the functions \eqref{expression1} and \eqref{expression2} take the same value at $v$, so they are both zero.

This completes the proof that the expression in question is contained in $\I(\bV)$.
The fact that $\I(\bV)$ is generated by expressions of this form follows from \cite[Theorem A.2(2)]{EHL}.
\end{proof}

We conclude this section with two key lemmas about decompositions of polyhedra.
Let $\cQ$ be a decomposition of a polyhedron $P$ of dimension $d$ and let $1 \le k \le d$.  
Given $S \in \cQ_{k+1}$ and $R \in \cQ_{k-1}$ with $R\subset S$, let 
$$X(R,S) := \{Q\in \cQ_k \mid R \subset Q \subset S\}.$$  
Note that this set always has cardinality exactly equal to 2.
For any $R\in\cQ_{k-1}$, let $\Gamma_R$ be the simple graph with vertices $\{Q\in\cQ_k\mid R\subset Q\}$ and edges $\{X(R,S)\mid R\subset S\in\cQ_{k+1}\}$ (see Figure \ref{spokes}).

\begin{figure}[h]
    \centering
    \begin{tikzpicture}
     \node (c) at (5,0) {};

    \node at (-0.5,0) {$R$};
    \node at (72:2.25) {$Q$};
    \node at ({72*2}:2.25) {$Q'$};
    \node at ($({72*1.5}:2.5) + (c)$) {$S$};
    
    \filldraw[color=white, fill=black!10] (0,0) -- (72:2) arc (72:{72*2}:2) -- cycle;
    \filldraw[color=black!10, fill = black!10] ($(72:1.9)+(c)$) arc (72:{72*2}:1.9) -- ($({72*2}:2.1) + (c) $) arc ({72*2}:72:2.1) -- cycle;
    \filldraw[black] (0,0) circle (3pt) {};
    \node at ({72*1.5}:1.5) {$S$};
   \draw[] (c) circle (2) {};
    \foreach \a in {1,...,5}
    {
	\node (u\a) at ({\a*72}:2){};
	\draw (0,0) -- (u\a) {};
	\node (uu\a) at ($({\a*72}:2) + (5,0)$) {};
	\filldraw[black] (uu\a) circle (3pt) {};
    }
    \node at ($(72:2.5) + (c)$) {$Q$};
    \node at ($({72*2}:2.5) + (c)$) {$Q'$};
    \end{tikzpicture}\caption{A small piece of a decomposition including a vertex $R$ that is incident to five 1-dimensional polyhedra and five 2-dimensional polyhedra, along with a picture of the graph $\Gamma_R$.}\label{spokes}
\end{figure}
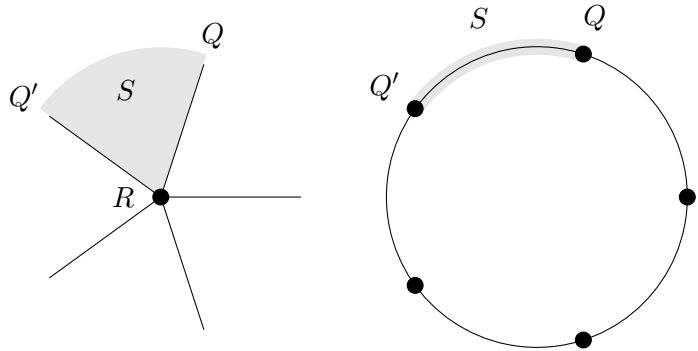

	


\begin{lemma} \label{lem:connected} The graph $\Gamma_R$ is connected.
\end{lemma}

\begin{proof} 
First suppose $1 \le k < d$. Let $D\subset \bV$ be a small disk of dimension $d-k+1$ that intersects $R$ transversely at a single point of the relative interior of $R$.  Then intersection with the elements of $\cQ$ defines a cellular decomposition of the boundary of $D$, and $\Gamma_R$
is isomorphic to the 1-skeleton of this decomposition.  The lemma now follows from the fact that the 1-skeleton of any cellular decomposition of the sphere $S^{d-k}$ is connected.

When $k=d$, the graph $\Gamma_R$ consists of two vertices connected by an edge. 
The vertices correspond to the two faces $Q_1, Q_2\in \cQ_d$ that have $R$ as a facet, and the edge is $X(R,P)$.
\end{proof}

\begin{lemma}\label{bounded complex}
Let $\cQ$ be a decomposition of a polyhedron $P$, and let $B$ be the union of the bounded faces of $\cQ$.  If $B$ is nonempty, then the inclusion of the pair $(B, \partial P\cap B)$ into $(P,\partial P)$ is a homotopy equivalence.
\end{lemma}

\begin{proof}
If there is any element of $\cQ$ with a nontrivial lineality space (equivalently, with no bounded faces), then every element of $\cQ$ has this property, and therefore $B$ is empty.
We may thus assume that every element of $\cQ$ has a bounded face.

Every unbounded polyhedron with a trivial lineality space admits a deformation retraction onto its boundary.  Applying these deformation retractions one at a time to the unbounded elements of $\cQ$, starting with those of maximal dimension,
we obtain a deformation retraction of $P$ onto $B$.  This restricts to a deformation retraction of $\partial P$ onto $\partial P \cap B$, and provides a homotopy inverse to the inclusion $(B, \partial P\cap B) \to (P,\partial P)$.
\end{proof}

\subsection{Decompositions of matroids}
Let $E$ be a finite set, and let $\R^E$ be the real vector space with basis $\{v_e \mid e \in E\}$. For any subset $S \subset E$, define
$$v_S := \sum_{e \in S} v_e\in\R^E.$$ 
For each $e\in E$, let \revise{$\delta_e$} be the linear functional on $\R^E$ defined by the property that $\revise{\delta_e(v_f) = | \{e\} \cap \{f\} |}$.
For any subset $S\subset E$, let
$$\revise{\delta_S} := \sum_{e\in S}\revise{\delta_e}.$$
Thus, for example, we have $$\revise{\delta}_{T}(v_S) = |S \cap T|.$$ 

Given a matroid $M$ on the ground set $E$, we define its {\bf base polytope} $P(M)\subset\R^E$
to be the convex hull of the set $\{v_B\mid \text{$B$ a basis for $M$}\}$. 
We write $d(M) := \dim P(M)$, which is equal to $|E|$ minus the number of connected components \revise{(i.e. minimal nonempty direct summands)} 
of $M$.  The entire polytope $P(M)$ lies in the affine subspace $\{v \mid \revise{\delta}_E(v) = \rk(M)\}$.

Let $\Mat(E)$ be the free abelian group with basis given by matroids on $E$, which embeds naturally in $\Pol(\R^E)$ \revise{via the map $M \mapsto P(M)$}. 
Let $\I(E) := \Mat(E)\cap \I(\R^E)$.
An abelian group homomorphism $\Mat(E)\to A$ is called {\bf valuative} if it vanishes on $\I(E)$.
Five such examples appear in the introduction,
all of which take values in the group $A = \Z[t]$.

Given a matroid $M$ on $E$, a {\bf decomposition} of $M$ is a collection $\cN$ of matroids on $E$ with the property that
$\cQ := \{P(N)\mid N\in\cN\}$ is a decomposition of $P(M)$.  We refer to elements of $\cN$ as {\bf faces} of the decomposition, 
and we say that a face $N\in\cN$ is {\bf internal}
if its base polytope is an internal face of $\cQ$.  We write $\cN_k$ to denote the set of internal faces $N\in\cN$ with $d(N) = k$ for all $k\leq d$, and we write $\cN_{d+1} := \{M\}$.
The following result follows from Proposition \ref{decomp} and \cite[Corollary 3.9]{derksen_fink}.

\begin{proposition}\label{mat-decomp}
If $\cN$ is a decomposition of $M$, then 
$$
\sum_k (-1)^k \sum_{N\in \cN_k} N\in\I(E).$$
Furthermore, $\I(E)$ is spanned by elements of this form.
\end{proposition}


\begin{example}\label{octahedron 2}
Example \ref{octahedron 1} describes a decomposition $\cN$ of the uniform matroid $M = U_{2,4}$.
The matroids $N$, $N'$, and $N''$ in that example are the three internal faces of $\cN$.
There are also many faces that are not internal, corresponding to the eight facets, twelve edges, and six vertices of $P(M)$. The generator 
of $\I(E)$
corresponding to this decomposition is depicted in Figure \ref{generator}. \revise{Note that, although $M$ appears in the expression $-M + N + N' - N''$, it is not a face of $\cN$.}
\end{example}

\begin{figure}[h]
    \centering
	\begin{tikzpicture}  
	[scale=0.5,auto=center,rotate around y=0] 
	\tikzstyle{edges} = [thick];

    \node[label=$-$] at (-3.5,-0.5,0) {};

    \node[style={circle,scale=0.8, fill=black, inner sep=0pt}] (14) at (-1.5,0,-2.398) {};
    \node[style={circle,scale=0.8, fill=black, inner sep=0pt},] (24) at (2.398,0,-1.5) {};
    \node[style={circle,scale=0.8, fill=black, inner sep=0pt},] (23) at (1.5,0,2.398) {};
    \node[style={circle,scale=0.8, fill=black, inner sep=0pt},] (13) at (-2.398,0,1.5) {};
    \node[style={circle,scale=0.8, fill=black, inner sep=0pt},] (12) at (0,2.5,0) {};
    \node[style={circle,scale=0.8, fill=black, inner sep=0pt},] (34) at (0,-2.5,0) {};

    \node[label=$+$] at (3.5,-0.5,0) {};
    
    \node[style={circle,scale=0.8, fill=black, inner sep=0pt}] (14N1) at ({-1.5+7},0,-2.398) {};
    \node[style={circle,scale=0.8, fill=black, inner sep=0pt}] (24N1) at ({2.398+7},0,-1.5) {};
    \node[style={circle,scale=0.8, fill=black, inner sep=0pt}] (23N1) at ({1.5+7},0,2.398) {};
    \node[style={circle,scale=0.8, fill=black, inner sep=0pt}] (13N1) at ({-2.398+7},0,1.5) {};
    \node[style={circle,scale=0.8, fill=black, inner sep=0pt}] (12N1) at (7,2.5,0) {};

    \node[label=$+$] at (10.5,-0.5,0) {};

    \node[style={circle,scale=0.8, fill=black, inner sep=0pt}] (14N2) at ({-1.5+14},0,-2.398) {};
    \node[style={circle,scale=0.8, fill=black, inner sep=0pt}] (24N2) at ({2.398+14},0,-1.5) {};
    \node[style={circle,scale=0.8, fill=black, inner sep=0pt}] (23N2) at ({1.5+14},0,2.398) {};
    \node[style={circle,scale=0.8, fill=black, inner sep=0pt}] (13N2) at ({-2.398+14},0,1.5) {};
    \node[style={circle,scale=0.8, fill=black, inner sep=0pt}] (34N2) at (14,-2.5,0) {};

    \node[label=$-$] at (17.5,-0.5,0) {};

    \node[style={circle,scale=0.8, fill=black, inner sep=0pt}] (14N12) at ({-1.5+21},0,-2.398) {};
    \node[style={circle,scale=0.8, fill=black, inner sep=0pt}] (24N12) at ({2.398+21},0,-1.5) {};
    \node[style={circle,scale=0.8, fill=black, inner sep=0pt}] (23N12) at ({1.5+21},0,2.398) {};
    \node[style={circle,scale=0.8, fill=black, inner sep=0pt}] (13N12) at ({-2.398+21},0,1.5) {};

    \draw[dashed] (13) -- (14);
    \draw[dashed] (14) -- (24);
    \draw (24) -- (23);
    \draw (23) -- (13); 
    \draw (12) -- (13);
    \draw[dashed] (12) -- (14);
    \draw (12) -- (24);
    \draw (12) -- (23);
    \draw (34) -- (13);
    \draw[dashed] (34) -- (14);
    \draw (34) -- (24);
    \draw (34) -- (23);

    \draw[dashed] (13N1) -- (14N1);
    \draw[dashed] (14N1) -- (24N1);
    \draw (24N1) -- (23N1);
    \draw (23N1) -- (13N1);
    \draw (12N1) -- (13N1);
    \draw[dashed] (12N1) -- (14N1);
    \draw (12N1) -- (24N1);
    \draw (12N1) -- (23N1);

    \draw (13N2) -- (14N2);
    \draw (14N2) -- (24N2);
    \draw (24N2) -- (23N2);
    \draw (23N2) -- (13N2);
    \draw (34N2) -- (13N2);
    \draw[dashed] (34N2) -- (14N2);
    \draw (34N2) -- (24N2);
    \draw (34N2) -- (23N2);

    \draw (13N12) -- (14N12);
    \draw (14N12) -- (24N12);
    \draw (24N12) -- (23N12);
    \draw (23N12) -- (13N12);

    \end{tikzpicture} \caption{The generator of $\I(E)$ from the decomposition of $M = U_{2,4}$.}\label{generator}
\end{figure}
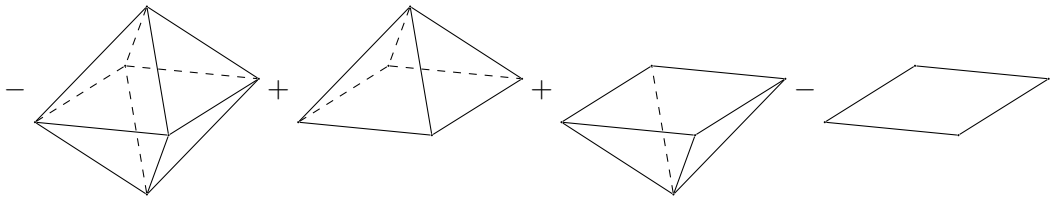

\begin{example}\label{triv decomp}
Any matroid $M$ has a {\bf trivial decomposition} consisting of $M$ itself along with all of the matroids $N$
such that $P(N)$ is a face of $P(M)$.  In this example, $M$ is the only internal face.  Note that the corresponding generator of $\I(E)$ is zero.
Moreover, the trivial decomposition is the only decomposition containing $M$ itself. 
\end{example}

\subsection{Relaxation}\label{sec:relaxation}
\revise{A large class of matroid decompositions can be constructed using relaxation, which we now review. If the reader is content without additional examples of decompositions, they are encouraged to skip this section. We return to the study of relaxations in Section \ref{sec:eq relax}.}

Let $M$ be a matroid 
on the ground set $E$.
A flat $F\subset M$ is called {\bf stressed} if the \revise{restriction} $M^F$ (obtained by deleting $E\setminus F$)
and the contraction $M_F$ (obtained by contracting a basis for $F$ and deleting the rest of $F$) 
are both uniform. Given a stressed flat $F$ of rank $r$, 
Ferroni and Schr\"oter define $\cusp(F)$ to be the collection of $k$-subsets $S\subset E$ such that $|S\cap F| = r+1$.  
If $\cB$ is the collection of bases of $M$, they prove that $\cB\cup\cusp(F)$ is the collection of bases for a new matroid $\tM$,
which they call the {\bf relaxation} of $M$ with respect to $F$ \cite[Theorem 3.12]{FSVal}.
If $F$ is a circuit-hyperplane, then $\cusp(F) = \{F\}$, and this coincides with the usual notion of relaxation. If $F$ is a hyperplane, then this coincides with the notion of relaxation of a stressed hyperplane studied in \cite{FNV}.

Let $M$, $F$ and $\tM$ be as in the previous paragraph, and let $k$ be the rank of $M$. 
Consider the matroid
$$\Pi_{r,k,\revise{F,E}} := U_{k-r,E\setminus F}\sqcup U_{r,F},$$
where $U_{d,S}$ denotes the uniform matroid of rank $d$ on the set $S$, and $\sqcup$ denotes the direct sum\footnote{We eschew the more standard notation of $\oplus$ for direct sum of matroids in order
to avoid conflict with formal direct sums in the additive closure of the category of matroids that we will introduce in the next section.} of matroids, so that a basis for $\Pi_{r,k,\revise{F,E}}$
is the disjoint union of a basis for $U_{k-r,E\setminus F}$ and a basis for $U_{r,F}$.
Let $\Lambda_{r,k,\revise{F,E}}$ denote the relaxation of $\Pi_{r,k,\revise{F,E}}$ with respect to the stressed flat $F$ of $\Pi_{r,k,\revise{F,E}}$.
The base polytope of $\Pi_{r,k,\revise{F,E}}$
is a face of the base polytopes of both $\Lambda_{r,k,\revise{F,E}}$ and $M$.
For $\Lambda_{r,k,\revise{F,E}}$, it is the facet on which the linear functional $\revise{\delta}_{E\setminus F}$ is maximized.
For $M$ it is the face on which the linear functional $\revise{\delta}_F$ is maximized, 
and it is a facet unless $M = \Pi_{r,k,\revise{F,E}}$.
Let $\cN$ be the collection of matroids consisting of $M$, $\Lambda_{r,k,\revise{F,E}}$ 
, and all of their faces.
\revise{The following theorem is proved in \cite[Theorem 3.30]{FSVal}}.

\begin{theorem}\label{relax, dude}
The collection $\cN$ is a decomposition of $\tM$.
If $M = \Pi_{r,k,\revise{F,E}}$, then $\cN$ is the trivial decomposition of $\Lambda_{r,k,\revise{F,E}}$.
If not, then the only internal faces of $\cN$ are $M$, $\Lambda_{r,k,\revise{F,E}}$, and $\Pi_{r,k,\revise{F,E}}$.
\end{theorem}

\begin{example}\label{octahedron 3}
In Example \ref{octahedron 1}, the matroid $N''$ has two stressed flats (both circuit-hyperplanes), namely $H = \{1,2\}$ and $H' = \{3,4\}$. \revise{Notice that $N'' = \Pi_{1,2,\revise{H,E}} = \Pi_{1,2,\revise{H',E}}$.} Relaxing $H$ gives us the trivial decomposition of 
$N = \Lambda_{1,2,\revise{H,E}}$.  If we then relax $H'$, which remains a stressed hyperplane of $N$, we obtain
the decomposition of $M$ from Example \ref{octahedron 2}.
Alternatively, we could have first relaxed $H'$ to obtain the trivial decomposition of $N'$, and then relaxed
$H$ to obtain the same decomposition of $M$.
\end{example}

Example \ref{octahedron 3} suggests a slight generalization of Theorem \ref{relax, dude}
 in which we relax more than one stressed flat at once. 
Suppose that $\Gamma$ is a finite group that acts on $E$ by permutations, with the property that $\Gamma$ fixes the matroid $M$.  Let $F$ be a stressed flat of $M$, and let $\cF := \{\gamma F\mid \gamma\in\Gamma\}$
be the set of all stressed flats in the same orbit as $F$.  Now define $\tM$ to be the relaxation of $M$ with respect to {\em all} of the elements of $\cF$.  More precisely, if $\cB$ is the collection of bases for $M$, then the collection of bases for $\tM$ is
$$\cB \cup \bigcup_{G\in\cF}\cusp(G).$$
This is a matroid because we can relax one flat at a time, and at each step, each element of $\cF$ that we have not yet relaxed remains a stressed flat.
Let $\cN$ be the collection of matroids consisting of $M$, $\Lambda_{r,k,\revise{G,E}}$ for all $G\in\cF$, and all matroids whose
polytopes are faces of $P(M)$ or $P(\Lambda_{r,k,\revise{G,E}})$.  

\begin{theorem}\label{relax symmetrically}
The collection $\cN$ is a decomposition of $\tM$.
If $M = \Pi_{r,k,\revise{F,E}}$, then $\cN$ is the trivial decomposition of $\Lambda_{r,k,\revise{F,E}}$.
If not, then the only internal faces of $\cN$ are $M$, $\Lambda_{r,k,\revise{G,E}}$ for all $G\in\cF$, and $\Pi_{r,k,\revise{G,E}}$ for all $G\in\cF$.
\end{theorem}

\begin{proof}
The case where $M = \Pi_{r,k,\revise{F,E}}$ is trivial.
Otherwise, by repeatedly applying Theorem \ref{relax, dude}, once for each element of $\cF$, we obtain a decomposition $\cN$ of $\tM$ with maximal faces consist of $M$ and $\{\Lambda_{r,k,\revise{G,E}}\mid G\in\cF\}$,
and whose internal faces include $\{\Pi_{r,k,\revise{G,E}}\mid G\in\cF\}$.  We need only prove that there are no additional internal faces.

Suppose that $N$ is a non-maximal internal face.  Then $P(N)$ is necessarily a face of the base polytope of at least two maximal faces of $\cN$.  
In particular, this implies that it is a face of $P(\Lambda_{r,k,\revise{G,E}})$ for some $G\in\cF$.  We may assume without loss of generality that $G$ was the last flat that we relaxed, 
in which case Theorem \ref{relax, dude} tells us that $N = \Pi_{r,k,\revise{G,E}}$.
\end{proof}

\begin{example}\label{octahedron 4}
Let $D_4$ act on the matroid $N''$ from Example \ref{octahedron 1} by symmetries of the square $P(N'')$.
If $F = H$, then $\cF = \{H, H'\}$.  We could also achieve this working only with the subgroup $\mathfrak{S}_2\subset D_4$ generated by the involution $\gamma = (13)(24)$.
\end{example}

\section{Valuative functors}\label{sec:categories}
Our goal in this section is to give precise definitions of categories of polyhedra and matroids, and what it means for a functor from
such a category to be valuative. 

\subsection{Categories of polyhedra and matroids}
We begin by defining a category $\cPol$ in which an object consists of a pair $(\bV,P)$, where $\bV$ is a finite dimensional real
vector space and $P$ is a nonempty polyhedron in $\bV$, and a morphism from $(\bV,P)$ to $(\bV',P')$ is a linear isomorphism $\varphi:\bV\to\bV'$ such that $P'\subset\varphi(P)$.
If $P' \subset P$, we will write $\iota_{P,P'}$ to denote the morphism from $(\bV,P)$ to $(\bV,P')$ given by the identity map $\id_{\bV}$.
For any $\bV$, we define $\cPol(\bV)$ to be the full subcategory of $\cPol$ consisting of polyhedra in $\bV$, and we define $\cPolid(\bV)$ to be the subcategory of $\cPol(\bV)$
consisting only of morphisms of the form $\iota_{P,P'}$.  Equivalently, $\cPolid(\bV)$ is the category associated with the poset of polyhedra in $\bV$, ordered by reverse inclusion.

Similarly, 
let $\cMat$ be the category in which an object consists of a pair $(E,M)$, where $E$ is a finite set and $M$ is a matroid on $E$, and a morphism from $(E,M)$ to $(E',M')$ is a bijection $\varphi:E\to E'$
such that $P(M')\subset \varphi(P(M))$.
In other words, $\cMat$ is the subcategory of $\cPol$ whose objects are base polytopes of matroids and whose morphisms come from bijections of ground sets.  Morphisms in $\cMat$ are sometimes called {\bf weak maps} of matroids.
For any finite set $E$, we define $\cMat(E)$ to be the full subcategory of $\cMat$ consisting of matroids on $E$, and we define $\cMatid(E)$
to be the subcategory of $\cMat(E)$ consisting of only morphisms $\iota_{M,M'}:M\to M'$ given by the identity map $\id_E$. 
Note that $\iota_{M,M'}:M\to M'$ is a morphism if and only if every basis for $M'$ is also a basis for $M$.

\begin{remark} The category $\cPolid(\bV)$ is thin in the sense of \S\ref{thin}, whereas the category $\cPol(\bV)$ is not. The $\Bbbk$-linear additive closure $\cPol^+(\bV)$ does not satisfy the Krull--Schmidt property, and complexes in $\Ch_b(\cPol^+(\bV))$ need not have well-defined minimal complexes. For this reason, we work with $\cPolid(\bV)$ when doing homological algebra. Similar statements apply to $\cMat(E)$ and $\cMatid(E)$. \end{remark}

Let $\cA$ be an additive category, and let $A$ be its split Grothendieck group. For an object $X$ of $\cA$, we write $[X]$ to denote its class in $A$. 
We will be interested in functors
$\Phi$ to $\cA$ from $\cPolid(\bV)$ or $\cMatid(E)$.  Such a functor induces a homomorphism 
from $\Pol(\bV)$ or $\Mat(E)$ to $A$, and we say that
the functor {\bf categorifies} the homomorphism. Often, but not always, the functors that interest us will extend naturally to the larger categories $\cPol$ or $\cMat$.

\begin{example}\label{exOS}
Let $\cA$ be the category of finite dimensional graded vector spaces over $\Q$. The {\bf Orlik--Solomon functor} $\OS:\cMat\to\cA$ takes a matroid $M$ to its Orlik--Solomon algebra $\OS(M)$,
and sends a weak map $\varphi:(E,M)\to (E',M')$ to the algebra homomorphism $\OS(\varphi):\OS(M)\to\OS(M')$ given by
sending the generator $u_e$ to the generator $u_{\varphi(e)}$ for all $e\in E$.
The split Grothendieck group of $\cA$ is isomorphic to the polynomial ring $\Z[t]$, and the functor $\OS$ categorifies the Poincar\'e polynomial.
See Section \ref{sec:OS} for a more detailed treatment of this example.
\end{example}

\begin{example} \label{ex:specificflagofflat}
Given a finite set $E$, a natural number $r$, an increasing $r$-tuple $\bk = (k_1,\ldots,k_r)$ of natural numbers, and an increasing $r$-tuple $\bS = (S_1,\ldots,S_r)$ of subsets of $E$,
we define a functor $$\Psi_{E,\bk,\bS}:\cMatid(E)\to \Vec_\Q$$ as follows.
On objects, $$\Psi_{E,\bk,\bS}(M) = \begin{cases}\Q & \text{if $S_i$ is a flat of rank $k_i$ for all $i$} \\ 0 & \text{otherwise.}\end{cases}$$
On morphisms, $\Psi_{E,\bk,\bS}(\id_E):\Psi_{E,\bk,\bS}(M)\to \Psi_{E,\bk,\bS}(M')$ is the identity map whenever each $S_i$ is a flat of rank $k_i$ for both $M$ and $M'$.
\end{example}

\begin{example} \label{ex:allflagofflat}
The functor $\Psi_{E,\bk,\bS}$ of Example \ref{ex:specificflagofflat} does not naturally extend from $\cMatid(E)$ to $\cMat(E)$. However, the direct sum
$$ \Psi_{E,\bk} := \bigoplus_{\bS} \Psi_{E,\bk,\bS} $$
sends each matroid $M$ to the vector space spanned by chains of flats with ranks given by $\bk$. This functor does extend to $\cMat(E)$, where for a bijection $\varphi:E \to E$ the summand $\Psi_{E,\bk,\bS}(M)$ is sent isomorphically to the summand $\Psi_{E,\bk,\varphi(\bS)}(M)$.
\end{example}

\subsection{The complex associated with a decomposition}
Let $P$ be a polyhedron in $\bV$ of dimension $d = d(P)$. An {\bf orientation} $\Omega_P$ of $P$ is an orientation of the relative interior of $P$,
\revise{regarded as a smooth manifold.  Equivalently, it is a choice of a positive ray in the top exterior power of the $d$-dimensional real
vector space $\R\{x-y\mid x,y\in P\}$.}
An orientation of $P$ induces an orientation of any facet $Q$ of $P$ by contracting with an outward normal vector.
Given orientations $\Omega_P$ and $\Omega_Q$ of $P$ and $Q$, we say that they {\bf match} if the orientation of $Q$ induced by 
$\Omega_P$ is equal to $\Omega_Q$.

Let $\cQ$ be a decomposition of a polyhedron $P$.
We define an {\bf orientation} $\Omega$ of $\cQ$ to be an arbitrary choice of orientation of each polyhedron in $\cQ$, 
along with a choice of orientation of $P$ itself.

Given the pair $(\cQ,\Omega)$, we define a chain complex $(C_\bullet^\Omega(\cQ),\partial^\Omega)\in\Ch_b(\cPolid^+(\bV))$ 
as follows.  First, we set \[C_k^\Omega(\cQ) := \bigoplus_{Q\in\cQ_k} Q.\]
If $1 \leq k\leq d$ and $R\in \cQ_{k-1}$ is a facet of $Q\in \cQ_k$,
then the $(Q,R)$ component of the differential $\partial_k^\Omega:C_k^\Omega(\cQ)\to C_{k-1}^\Omega(\cQ)$ is given by $\pm \iota_{Q,R}$, depending on whether or not the orientation of $R$ matches the orientation of $Q$. 
Similarly, for each $Q\in\cQ_{d}$, the relative interior of $Q$ is an open submanifold of the relative interior of $P$, and
the $(P,Q)$ component of the differential $\partial_{d+1}^\Omega$ is given by $\pm \iota_{P,Q}$, depending on whether or not the orientation of $Q$ agrees with the restriction of the orientation of $P$.
As noted in Section \ref{sec:decompositions}, if $R \in \cQ_{k-1}$ and $S\in \cQ_{k+1}$ for some $1\leq k \leq d$, then the set $X(R,S) = \{Q_1, Q_2\}$ has cardinality exactly two, giving two contributions to the $(S,R)$ component of $\pa^2$. The normal vectors of the two
inclusions $R \subset Q_i$ are opposite, so these two contributions cancel each other out.  This proves that $\partial^2 = 0$.

Let $C_{\leq d}^{\Omega}(\cQ)$ be the subcomplex of $C_{\bullet}^{\Omega}(\cQ)$ consisting of everything in degree less than or equal to $d$.
There is a chain map $\alpha_{\cQ}^\Omega : P[-d] \to C_{\leq d}^{\Omega}(\cQ)$ given by the first differential in $C_{\bullet}^{\Omega}(\cQ)$, and we have an isomorphism
\begin{equation*} C_{\bullet}^{\Omega}(\cQ) \cong \Cone(\alpha_{\cQ}^\Omega). \end{equation*}

If $\cN$ is a decomposition of a matroid $M$ on $E$, we define an orientation $\Omega$ of $\cN$ to be an orientation of the induced
decomposition of base polytopes, and we define the analogous chain complexes $C_{\leq d}^\Omega(\cN)\subset C_\bullet^\Omega(\cN)\in\Ch_b(\cMatid^+(E))$.

\begin{example}\label{octahedron 5}
Consider the decomposition $\cN$ 
from Examples \ref{octahedron 1} and \ref{octahedron 2}.
The complex $C_\bullet^\Omega(\cN)$ takes the form depicted in Figure \ref{complex}. 
Choose an orientation of the 3-dimensional vector space $\{v\mid \revise{\delta}_{E}(v) = 2\}\subset \R^E$.
The relative interiors of $P(M)$, $P(N)$, and $P(N')$ are all open subsets of this vector space, so our
choice of orientation induces orientations $\Omega(M)$, $\Omega(N)$, and $\Omega(N')$.
Choose $\Omega(N'')$ to be the orientation induced by realizing $P(N'')$ as a facet of $P(N)$,
which is the opposite of the orientation induced by realizing $P(N'')$ as a facet of $P(N')$.
We have $$\Hom_{\cMatid^+(E)}(M,N\oplus N')\cong \Q^2,$$ and our first differential corresponds to the element $(1,1)$.
We also have $$\Hom_{\cMatid^+(E)}(N\oplus N',N'')\cong \Q^2,$$ and our second differential corresponds to the element $(1,-1)$.
The composition is given by dot product, and our differential squares to zero because $(1,1)$ is orthogonal to $(1,-1)$.
\end{example}

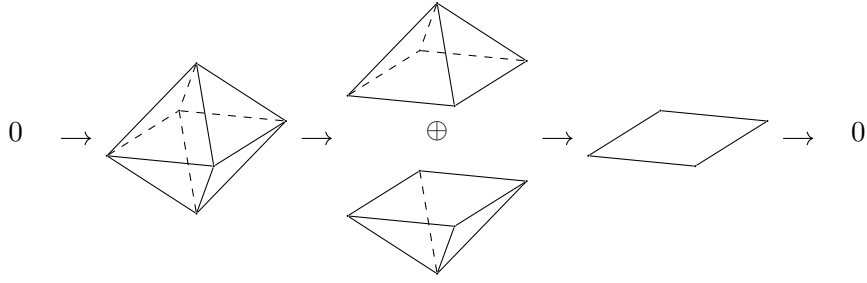
\begin{figure}[h]
    \centering
	\begin{tikzpicture}  
	[scale=0.4,auto=center] 
	\tikzstyle{edges} = [thick];

    \node[label=$0$] at (-7,-0.75,0) {};

    \draw[->] (-5.5,0,0) -- (-4.5,0,0);

    \node[style={circle,scale=0.8, fill=black, inner sep=0pt}] (14) at ({-1.5-1},0,-2.398) {};
    \node[style={circle,scale=0.8, fill=black, inner sep=0pt},] (24) at ({2.398-1},0,-1.5) {};
    \node[style={circle,scale=0.8, fill=black, inner sep=0pt},] (23) at ({1.5-1},0,2.398) {};
    \node[style={circle,scale=0.8, fill=black, inner sep=0pt},] (13) at ({-2.398-1},0,1.5) {};
    \node[style={circle,scale=0.8, fill=black, inner sep=0pt},] (12) at (-1,2.5,0) {};
    \node[style={circle,scale=0.8, fill=black, inner sep=0pt},] (34) at (-1,-2.5,0) {};

    \draw[->] (2.5,0,0) -- (3.5,0,0);
    
    \node[style={circle,scale=0.8, fill=black, inner sep=0pt}] (14N1) at ({-1.5+7},2,-2.398)  {};
    \node[style={circle,scale=0.8, fill=black, inner sep=0pt}] (24N1) at ({2.398+7},2,-1.5)  {};
    \node[style={circle,scale=0.8, fill=black, inner sep=0pt}] (23N1) at ({1.5+7},2,2.398)  {};
    \node[style={circle,scale=0.8, fill=black, inner sep=0pt}] (13N1) at ({-2.398+7},2,1.5) {};
    \node[style={circle,scale=0.8, fill=black, inner sep=0pt}] (12N1) at (7,4.5,0) {};

    \node[label=$\oplus$] at (7,-0.75,0) {};

    \node[style={circle,scale=0.8, fill=black, inner sep=0pt}] (14N2) at ({-1.5+7},-2,-2.398) {};
    \node[style={circle,scale=0.8, fill=black, inner sep=0pt}] (24N2) at ({2.398+7},-2,-1.5) {};
    \node[style={circle,scale=0.8, fill=black, inner sep=0pt}] (23N2) at ({1.5+7},-2,2.398) {};
    \node[style={circle,scale=0.8, fill=black, inner sep=0pt}] (13N2) at ({-2.398+7},-2,1.5) {};
    \node[style={circle,scale=0.8, fill=black, inner sep=0pt}] (34N2) at (7,-4.5,0) {};

    \draw[->] (10.5,0,0) -- (11.5,0,0);

    \node[style={circle,scale=0.8, fill=black, inner sep=0pt}] (14N12) at ({-1.5+15},0,-2.398) {};
    \node[style={circle,scale=0.8, fill=black, inner sep=0pt}] (24N12) at ({2.398+15},0,-1.5) {};
    \node[style={circle,scale=0.8, fill=black, inner sep=0pt}] (23N12) at ({1.5+15},0,2.398) {};
    \node[style={circle,scale=0.8, fill=black, inner sep=0pt}] (13N12) at ({-2.398+15},0,1.5) {};

    \draw[->] (18.5,0,0) -- (19.5,0,0);

    \node[label=$0$] at (21,-0.75,0) {};

    \draw[dashed] (13) -- (14);
    \draw[dashed] (14) -- (24);
    \draw (24) -- (23);
    \draw (23) -- (13); 
    \draw (12) -- (13);
    \draw[dashed] (12) -- (14);
    \draw (12) -- (24);
    \draw (12) -- (23);
    \draw (34) -- (13);
    \draw[dashed] (34) -- (14);
    \draw (34) -- (24);
    \draw (34) -- (23);

    \draw[dashed] (13N1) -- (14N1);
    \draw[dashed] (14N1) -- (24N1);
    \draw (24N1) -- (23N1);
    \draw (23N1) -- (13N1);
    \draw (12N1) -- (13N1);
    \draw[dashed] (12N1) -- (14N1);
    \draw (12N1) -- (24N1);
    \draw (12N1) -- (23N1);

    \draw (13N2) -- (14N2);
    \draw (14N2) -- (24N2);
    \draw (24N2) -- (23N2);
    \draw (23N2) -- (13N2);
    \draw (34N2) -- (13N2);
    \draw[dashed] (34N2) -- (14N2);
    \draw (34N2) -- (24N2);
    \draw (34N2) -- (23N2);

    \draw (13N12) -- (14N12);
    \draw (14N12) -- (24N12);
    \draw (24N12) -- (23N12);
    \draw (23N12) -- (13N12);

    \end{tikzpicture} \caption{The complex $C_\bullet^\Omega(\cN)$ arising from the decomposition of $M = U_{2,4}$. This complex is supported in degrees $1$, $2$, and $3$.}\label{complex}
\end{figure}

\begin{remark} \label{rmk:orientationirrelevant}
Let $\Omega$ be an orientation of a decomposition $\cQ$ of a polyhedron $P$, and let $\Omega'$ be the orientation obtained from $\Omega$ by reversing the orientation on a single face $Q \in \cQ$. The only difference between $C_\bullet^\Omega(\cN)$ and $C_\bullet^{\Omega'}(\cN)$ is that the signs of the morphisms going into and out of the summand $Q$ are reversed. There is an isomorphism $C_{\bullet}^{\Omega}(\cN) \to C_\bullet^{\Omega'}(\cN)$ given by the identity map on all faces $Q' \ne Q$, and minus the identity map on $Q$. Thus the choice of orientation does not affect the isomorphism class of the complex.
\end{remark}

The following lemma is a strengthening of Remark \ref{rmk:orientationirrelevant}.  Not only is the isomorphism class of
the complex $C_\bullet^\Omega(\cQ)$ independent of $\Omega$, but any complex 
that looks as if it could be isomorphic to $C_\bullet^\Omega(\cQ)$ is indeed isomorphic to it.  This lemma will be a key technical ingredient in Section \ref{sec:convolution}. 

\begin{lemma}\label{there's an orientation}
Fix a decomposition $\cQ$ of $P$ and an orientation $\Omega$ of $\cQ$.
Let $(C_\bullet,\partial)\in\Ch_b(\cPolid^+(\bV))$ be any complex with the following properties:
\begin{itemize}
\item For all $k$, $\displaystyle C_k= \bigoplus_{Q\in\cQ_k} Q = C_k^\Omega(\cQ)$.  
\item If $Q\in\cQ_k$, $R\in \cQ_{k-1}$, and $R\subset Q$,
then the $(Q,R)$ component of the differential $\partial_k$ is an invertible multiple of $\iota_{Q,R}$. Otherwise, the $(Q,R)$ component is zero.
\end{itemize}
Then there exists an isomorphism of complexes $(C_\bullet, \partial) \cong (C_\bullet^\Omega(\cQ), \partial^\Omega)$.
\end{lemma} 

\begin{proof} Choose an element $Q\in\cQ_k$ for some $k$, and let $(C_{\bullet},\partial')$ be the complex obtained from $(C_{\bullet},\partial)$ by multiplying all maps out of $Q$ by an invertible scalar
$\lambda \in \Bbbk^{\times}$, and multiplying all maps into $Q$ by $\lambda^{-1}$, an operation which we call {\bf rescaling} at $Q$ by $\lambda$. There is an isomorphism $(C_{\bullet},\partial) \to
(C_{\bullet},\partial')$ given by the identity map on all faces $Q' \ne Q$ and $\lambda^{-1}$ times the identity map on $Q$. 
We will show that $(C_\bullet,\partial)$ can be transformed into $(C_\bullet^\Omega(\cQ),\partial^\Omega)$ by a finite sequence of rescalings. One can begin by rescaling at all $Q \in \cQ_d$, so that $\pa_{d+1} = \pa_{d+1}^{\Omega}$.
We will assume inductively that $\partial_l = \partial_l^\Omega$ for all $l>k$, and we will show that $(C_\bullet,\partial)$ can be rescaled to a complex $(C_\bullet,\partial')$ with $\partial'_l = \partial_l^\Omega$ for all $l\geq k$.

Given any $Q \in \cQ_l$ and $R \in \cQ_{l-1}$ with $R\subset Q$, let $a_{Q,R}$ be the coefficient of $\iota_{Q,R}$ in the $(Q,R)$ component of $\partial_l$ and let $b_{Q,R}$ be the coefficient of $\iota_{Q,R}$ in the $(Q,R)$ component of $\partial_l^\Omega$. 
Note that $a_{Q,R}$ and $b_{Q,R}$ are both invertible.  
Now fix a specific $Q \in \cQ_k$ and $R \in \cQ_{k-1}$ with $R\subset Q$.
By rescaling $(C_\bullet,\partial)$ at $R$, we may assume that $a_{Q,R} = b_{Q,R}$.
Let us say that an element $Q'\in\cQ_k$ with $R\subset Q'$ is \revise{{\bf simpatico}} if $a_{Q'R} = b_{Q'R}$.
By assumption, $Q$ is \revise{simpatico}.  We claim that {\em every} $Q'\in\cQ_k$ that contains $R$ is \revise{simpatico}.
If we can show this, then we may complete the inductive step by rescaling once at each $R\in\cQ_{k-1}$.

Recall that we defined a graph $\Gamma_R$ with vertex set $\{Q'\in \cQ_k\mid R\subset Q'\}$ and edge set $$\{X(R,S)\mid R\subset S\in\cQ_{k+1}\},$$ and Lemma \ref{lem:connected} states that this graph is connected.  Thus it will be sufficient to prove that, if $X(R,S) = \{Q',Q''\}$ is an edge of $\Gamma_R$, then $Q'$ is \revise{simpatico} if and only if $Q''$ is \revise{simpatico}.

Examining the $(S,R)$ component of the composition $\partial_k\circ\partial_{k+1} = 0$, we see that
$$a_{SQ'}a_{QR'} + a_{SQ''}a_{Q''R} = 0.$$
Similar reasoning for the differential $\partial^\Omega$ tells us that
$$b_{SQ'}b_{QR'} + b_{SQ''}b_{Q''R} = 0.$$
Since we have assumed that $\partial_{k+1} = \partial_{k+1}^\Omega$, we have $b_{SQ'} = a_{SQ'}$ and $b_{SQ''} = a_{SQ''}$.
Taking the difference of the two equations, we find that
$$a_{SQ'}(a_{Q'R}-b_{Q'R}) + a_{SQ''}(a_{Q''R}-b_{Q''R}) = 0.$$
Thus $Q'$ is \revise{simpatico} if and only if $Q''$ is \revise{simpatico}.
\end{proof}

\subsection{Valuative functors}\label{sec:valfun}

Let $\cA$ be an $\Bbbk$-linear additive category. 
We say that a functor $\Phi:\cPolid^+(\bV)\to\cA$ is {\bf valuative} if, for any pair $(\cQ,\Omega)$, the complex $\Phi(C_\bullet^\Omega(\cQ))$ is contractible. By Lemma \ref{lem:conestuff}, this is equivalent to the condition
that $\Phi(\alpha_{\cQ}^\Omega)$ is a homotopy equivalence. 
We say that a functor from $\cPolid(\bV)$, $\cPol(\bV)$, or $\cPol$ to $\cA$ is valuative if the induced functor from
$\cPolid^+(\bV)$ to $\cA$ is valuative.

Similarly, we say that $\Phi:\cMatid^+(E)\to\cA$ is valuative if, for any pair $(\cN,\Omega)$, $\Phi(C_\bullet^\Omega(\cN))$ is
contractible.
We say that a functor from $\cMatid(E)$, $\cMat(E)$, or $\cMat$ to $\cA$ is valuative if the induced functor from
$\cMatid^+(\bV)$ to $\cA$ is valuative.
Note that any valuative functor on $\cPolid^+(\R^E)$ restricts to a valuative functor on $\cMatid^+(E)$.
By Propositions \ref{decomp} and \ref{mat-decomp}, any valuative functor categorifies a valuative homomorphism.

\begin{remark} When the target category $\cA$ is semisimple, $\Phi$ is valuative if and only if $\Phi(C_\bullet^\Omega(\cQ))$ is exact for all $(\cQ,\Omega)$.  In all of our examples, $\cA$ will be the category of (possibly graded or bigraded) finite dimensional
$\Q$-vector spaces, which is indeed semisimple.
\end{remark}

\begin{remark}
A direct sum of valuative functors is valuative, and a direct summand of a valuative functor is valuative.
These statements follow from the corresponding statements about contractible complexes.
\end{remark}

Let $\cI(\bV)$ be the localizing subcategory inside $\Ch_b(\cPol^+(\bV))$ generated by complexes of the form $C_\bullet^\Omega(\cQ)$. Let $\cV(\bV)$ be the quotient of $\Ch_b(\cPolid^+(\bV))$ by $\cI(\bV)$. A functor $\cPolid^+(\bV)\to\cA$ is valuative if and only if it descends to a triangulated functor $\cV(\bV) \to \cK_b(\cA)$.
Similarly, let $\cI(E)$ be the localizing subcategory inside $\Ch_b(\cMatid^+(E))$ generated by complexes of the form $C_\bullet^\Omega(\cN)$. Let $\cV(E)$ be the quotient of
$\Ch_b(\cMatid^+(E))$ by $\cI(E)$. A functor $\cMatid^+(E)\to\cA$ is valuative if and only if it descends to a triangulated functor $\cV(E) \to \cK_b(\cA)$.

\begin{remark}\label{rmk:groth}
The triangulated Grothendieck group of $\cV(E)$
is {\em a priori} isomorphic to a quotient of the valuative group $\Val(E) := \Mat(E)/\I(E)$; we will prove
in Section \ref{sec:groth} that it is in fact isomorphic to $\Val(E)$.
The valuative group $\Val(E)$ is canonically isomorphic to the homology of the stellahedral toric variety
\cite[Theorem 1.5]{EHL}, with the homological grading corresponding to the grading of $\Val(E)$ by rank.  It would be interesting
to find a corresponding geometric interpretation of the triangulated category $\cV(E)$ in terms of the same toric variety.
\end{remark}

As a basic example, consider the trivial functor $\tau:\cPol\to \Vec_\Q$ that takes all polyhedra to $\Q$ and all morphisms to the identity map.  This categorifies the homomorphism that evaluates to 1 on every polyhedron.

\begin{proposition}\label{triv val}
The trivial functor $\tau$ is valuative.
\end{proposition}

\begin{proof}
We need to show that, for any decomposition $\cQ$ of a polyhedron $P$ in a vector space $\bV$
and any orientation $\Omega$ of $\cQ$, the complex $\tau(C_\bullet^\Omega(\cQ))$ is exact. Since $C_{\bullet}^{\Omega}(\cQ)$ is the cone of $$\alpha_\cQ^\Omega : P[-d] \to C_{\leq d}^{\Omega}(\cQ),$$ there is a termwise-split short exact sequence of complexes
$$0 \to C_{\leq d}^{\Omega}(\cQ) \to C_{\bullet}^\Omega(\cQ) \to P[-d-1]\to 0.$$
Additive functors preserve cones, so we also have a short exact sequence of vector spaces
$$0 \to \tau(C_{\leq d}^\Omega(\cQ)) \to \tau(C_{\bullet}^\Omega(\cQ)) \to \Q[-d-1]\to 0.$$
The boundary map in the long exact sequence in cohomology is induced by $\tau(\alpha_\cQ^\Omega)$, which is a general fact about cones.

\revise{The complex $\tau(C_{\leq d}^{\Omega}(\cQ))$ coincides with the cellular chain complex that computes the homology of the one point compactification
of $P$ relative to the one point compactification of $\partial P$. Indeed, the one point compactification of $P$ has one cell for each face of $P$ (bounded or not, internal or not) and one additional cell for the extra point. All but the internal faces are included in the one point compactification of $\partial P$. Thus the complex computing relative homology has one term for each internal face of $P$, 
and precisely matches $\tau(C_{\leq d}^{\Omega}(\cQ))$.}
This relative homology is 1-dimensional and concentrated in degree $d$. The boundary map to degree $d$ must be an isomorphism, since $\tau(\alpha_\cQ^\Omega)$ is evidently injective.
Thus the homology of $\tau(C_{\bullet}^\Omega(\cQ))$ vanishes.
\end{proof}

\begin{remark}\label{topological interpretation of chain complex}
In Section \ref{sec:algebra}, we will need a slight generalization of the observation that we used at the end of the proof of Proposition
\ref{triv val}.  Let $\cQ$ be a decomposition of a polyhedron of dimension $d$, and let $\Omega$ be an orientation of $\cQ$.  
Let $\cS\subset\cR\subset\cQ$ be subsets of $\cQ$
that are closed under taking faces.  
Let $$D_k^\Omega(\cR,\cS) := \bigoplus_{\substack{Q\in\cR\setminus\cS\\ \dim Q = k}} Q,$$
and define a differential $\partial^\Omega$ as before.
For example, if $\cR = \cQ$ and $\cS$ is the set of non-internal faces, then
$$D_\bullet^\Omega(\cR,\cS) = C_{\leq d}^\Omega(\cQ).$$
Define $S\subset R\subset \bV$ by taking $S$ to be the union of the elements of $\cS$,
and $\cR$ to be the union of the elements of $\cR$.  
If $R$ is bounded, then it admits the structure of a CW complex with closed cells $\cR$, or with open cells $\{\becircled Q\mid Q\in\cR\}$, where $\becircled Q$ denotes the relative interior of $Q$.  
In this case, $\tau(D_\bullet^\Omega(\cR,\cS),\partial^\Omega)$ may be identified with the cellular chain complex for the pair $(R, S)$.
More generally, the one point compactification
$\hat R := R\sqcup \{\star\}$ admits the structure of a CW complex with open cells
$\{\becircled Q\mid Q\in\cR\}\sqcup\{\{\star\}\}$, and $\tau(D_\bullet^\Omega(\cR,\cS),\partial^\Omega)$ may be identified with the cellular chain complex for the pair $(\hat R, \hat S)$.
\end{remark}

\section{The Orlik--Solomon functor}\label{sec:OS}
The purpose of this section is to prove that the Orlik--Solomon functor of Example \ref{exOS} is valuative.  

\subsection{The Orlik--Solomon algebra}
Let $E$ be a finite set, and let $\revise{\bigwedge}_E$ be the exterior algebra over $\Q$ with generators $\{u_e\mid e\in E\}$, \revise{graded in the usual way with $\deg u_e = 1$}.
Let $n$ be the cardinality of $E$, and fix an identification of $E$ with the \revise{ordered} set $\{1,\ldots,n\}$.
For any subset $S = \{e_1,\ldots,e_k\}\subset E$ with $e_1<e_2<\cdots<e_k$, consider 
the monomial $u_S := u_{e_1}\cdots u_{e_k}\in \revise{\bigwedge}_E$ and 
the element $$w_S := \sum_{i=1}^k (-1)^{i-1} u_{e_1}\cdots\widehat{u_{e_i}}\cdots u_{e_k}\in\revise{\bigwedge}_E.$$

A set $S$ is called {\bf independent} if it is contained in some basis and {\bf dependent} otherwise.
A minimal dependent set is called a {\bf circuit}.
The {\bf Orlik--Solomon algebra} $\OS(M)$ is defined as the quotient of $\revise{\bigwedge}_E$ by the \revise{(homogeneous)} ideal generated by $\{w_S\mid \text{$S$ a circuit}\}$.
We observe that changing the order on $E$ changes $w_S$ by a sign, therefore the Orlik--Solomon algebra does not in fact depend on the identification of $E$ with $\{1,\ldots,n\}$.
We also observe that $w_S$ divides $w_T$ whenever $S\subset T$, thus $\OS(M)$ may also be defined as the quotient of $\revise{\bigwedge}_E$ by the ideal generated by $\{w_S\mid \text{$S$ dependent}\}$.
This makes it clear that the homomorphisms in Example \ref{exOS} are well defined.
Though these are in fact algebra homomorphisms, we will only regard $\OS$ as a functor from $\cMat$ to the category $\cA$ of 
finite dimensional graded vector spaces over $\Q$.

\subsection{Degenerating}\label{sec:degenerating}
\revise{We now place a second $\N$-grading on the exterior algebra $\bigwedge_E$, distinct from the standard one used above. We set the degree of $u_e$ equal to $e \in E = \{1,\ldots,n\}$, viewed as a natural number. The elements $w_S$ (for $S$ a circuit) are not homogeneous in this grading. Consequently the quotient $\OS(M)$ does not inherit a grading but a filtration, } whose $i^\text{th}$ piece is equal to the image of classes
of degree $\leq i$ in $\revise{\bigwedge}_E$. For any circuit $S\subset E$, we define the associated {\bf broken circuit} $\bar S$
to be the set obtained from $S$ by removing the minimal element \revise{(using the usual order on $E$). Then } the associated graded ring $\grOS(M)$ is isomorphic to the quotient of $\revise{\bigwedge}_E$ by the ideal generated by $\{u_{\bar S}\mid\text{$S$ a circuit}\}$ \cite[Theorem 3.43]{OT}.
Note that this filtration is functorial with respect to morphisms in $\cMatid(E)$ (though not for morphisms in $\cMat(E)$), so we obtain a functor $\grOS\co\cMatid(E)\to\cA$.


Let us explicitly describe the functor $\grOS$ on morphisms. 
We define $$\nbc(M) := \{S\subset E\mid \text{$S$ does not contain any broken circuit}\}.$$
Then the set $\{u_S\mid S\in\nbc(M)\}$ is a basis for $\grOS(M)$, where $\deg(u_S) = |S|$. If $\iota_{M,M'}$ is a weak map, then every (broken) circuit for $M$ contains a (broken) circuit for $M'$, hence 
we have an inclusion $\nbc(M') \subset \nbc(M)$.
The map $\grOS(\iota_{M,M'}) \colon \grOS(M)\to\grOS(M')$ takes $u_S$ to $u_S$ if $S\in\nbc(M')$ and to 0 otherwise. 

Consider the functor $V(-,S)\co\cMatid(E)\to\cA$ given by putting $$V(M,S) := \begin{cases} \Bbbk &\text{if $S \in \nbc(M)$}\\ 0 &\text{otherwise,}\end{cases}$$
with the morphism $\iota_{M,M'}$ sent to the identity map whenever $S\in\nbc(M')$.
The previous paragraph can be summarized by saying that there is a natural isomorphism of functors
\begin{equation}\label{decomposegros} \grOS \cong \bigoplus_{S \subset E} V(-,S)\big( -|S| \big).\end{equation}
We use round brackets to denote grading shifts, so as not to confuse
with the square brackets that we use to denote homological shifts; thus $V(-,S)\big( -|S| \big)$ takes a matroid $M$ with $S\in\nbc(M)$ to a single copy of $\Q$ in degree $|S|$.

\begin{remark}\label{basis remark}
With Equation \eqref{decomposegros}, we are decomposing the functor $\grOS$ as a sum of functors that send every matroid to either a shift of $\Q$ or to $0$.
For any particular $M$, this corresponds to a certain basis for $\grOS(M)$, namely the nbc basis.
We employ a similar approach with the Chow ring and augmented Chow ring in Section \ref{sec:chow functors}.
\end{remark}

Let $\cN$ be a decomposition of a matroid $M$ on the ground set $E$, and let $d = d(M)$.
For any $S\in\nbc(M)$, consider the quotient complex $V^\Omega_\bullet(\cN,S)$ of $\tau(C_\bullet^\Omega(\cN))$
given by putting 
$$V^\Omega_k(\cN,S) := \bigoplus_{\substack{N\in\cN_k\\ S\in\nbc(N)}}\Q.$$
More informally, $V^\Omega_\bullet(\cN,S)$ is obtained from $\tau(C_\bullet^\Omega(\cN))$ by killing the termwise-split subcomplex consisting of all terms corresponding to internal faces $N\in\cN$ for which $S\notin\nbc(N)$.
By \eqref{decomposegros} we have an isomorphism of complexes of graded vector spaces
\begin{equation}\label{direct sum}\grOS(C_\bullet^\Omega(\cN))\;\; \cong \bigoplus_{S\in\nbc(M)} V^\Omega_\bullet (\cN,S)\big(-|S|\big).
\end{equation}
Our strategy will be to prove that $V^\Omega_\bullet(\cN,S)$ is exact, and use this to prove Theorem \ref{thm:OS}.

\subsection{Characterizing the nbc condition}
Fix a subset $S\subset E$.  For each $e\in E$, let $S_e := \{s\in S\mid s>e\}$, and consider the open half-space
$$H_{e,S}^+ := \Big\{\,v\in \R^E\;\bigmid \revise{\delta}_{S_e\cup\{e\}}(v)
> |S_e|\,\Big\}.$$

\begin{lemma}\label{nbc conditions}
If $M$ is a matroid on $E$, the following statements are equivalent:
\begin{itemize}
\item[{\em (i)}] $S\in\nbc(M)$
\item[{\em (ii)}] $S_e \cup \{e\}$ is independent for all $e\in E$
\item[{\em (iii)}] $P(M)\cap H_{e,S}^+ \neq \varnothing$ for all $e\in E$
\item[{\em (iv)}] $P(M) \cap \,\displaystyle\bigcap_{e\in E}H_{e,S}^+ \neq \varnothing$.
\end{itemize}
\end{lemma}

\begin{proof}
The equivalence of (i) and (ii) is immediate from the definition of a broken circuit.  We next prove the equivalence of (ii)
and (iii).  If $S_e\cup \{e\}$ is independent, then it is contained in some basis $B$, and $v_B\in P(M)\cap H^+_{e,S}$.
Conversely, suppose that $v\in P(M)\cap H^+_{e,S}$.  Then we have
$$|S_e| < \revise{\delta}_{S_e\cup\{e\}}(v) \leq \rk\big(S_e\cup\{e\}\big),$$
where the first inequality comes from the fact that $v\in H_{e,S}^+$ and the second comes from the fact that $v\in P(M)$.
This implies that the cardinality of $S_e\cup\{e\}$ is equal to its rank, which means that it is independent.

We have now established the equivalence of (i), (ii), and (iii).
The fact that (iv) implies (iii) is obvious, thus we can finish the proof by showing that (ii) implies (iv).
Assume that (ii) holds, and for each $e\in E$, choose a basis $B_e$ containing $S_e\cup\{e\}$.
In addition, choose real numbers $\revise{r}_0,\ldots,\revise{r}_n$
with $\revise{r}_0 = 1$, $\revise{r}_n = 0$, and $\revise{r}_e < \revise{r}_{e-1}/\big(|S_e|+1\big)$ for all $e\in E$.
Let $$v := \sum_{e\in E}(\revise{r}_{e-1} - \revise{r}_e)\,v_{B_e}\in\R^E.$$
The sum of the coefficients appearing in the definition of $v$ is equal to $\revise{r}_0-\revise{r}_n = 1$,
thus $v$ is in the convex hull of $\{v_{B_e}\mid e\in E\}$, which is contained in $P(M)$.  It thus remains 
only to prove that $v\in H_{e,S}^+$ for all $e\in E$.
We have
\begin{equation}\label{first}\revise{\delta}_{S_e\cup\{e\}}(v) = \sum_{f\in S_e\cup\{e\}}\sum_{B_g\ni f} (\revise{r}_{g-1}-\revise{r}_g)
= \sum_{f\in S_e}\sum_{B_g\ni f} (\revise{r}_{g-1}-\revise{r}_g) + \sum_{B_g\ni e} (\revise{r}_{g-1}-\revise{r}_g).\end{equation}
Note that, if $g\leq f$ and $f\in S$, then $f\in S_g\cup\{g\}\subset B_g$.  This implies that
\begin{equation}\label{second}\sum_{f\in S_e}\sum_{B_g\ni f} (\revise{r}_{g-1}-\revise{r}_g)\geq 
\sum_{f\in S_e}\sum_{g\leq f} (\revise{r}_{g-1}-\revise{r}_g) = \sum_{f\in S_e} (\revise{r}_0 - \revise{r}_f)
\geq \sum_{f\in S_e} (\revise{r}_0 - \revise{r}_e) = |S_e|(1-\revise{r}_e).\end{equation}
In addition, we have $e\in B_e$, and therefore 
\begin{equation}\label{third}\sum_{B_g\ni e} (\revise{r}_{g-1}-\revise{r}_g) \geq \revise{r}_{e-1} - \revise{r}_e> |S_e|\revise{r}_e.\end{equation}
Combining Equations \eqref{first}, \eqref{second}, and \eqref{third}, we find that
$$\revise{\delta}_{S_e\cup\{e\}}(v) > |S_e|(1-\revise{r}_e) + |S_e|\revise{r}_e = |S_e|,$$
and therefore $v\in H_{e,S}^+$.
\end{proof}

\subsection{Exactness of the summands}
Fix a matroid $M$ on the ground set $E$, a decomposition $\cN$ of $M$ with orientation $\Omega$, 
and a set $S\in\nbc(M)$.  We now use Lemma \ref{nbc conditions} to prove the following proposition.

\begin{proposition}\label{V exact}
The complex $V_\bullet^\Omega(\cN,S)$ is exact.
\end{proposition}

\begin{remark}
Let $f_S:\Mat(E)\to\Z$ be the homomorphism characterized by putting $f_S(M) = 1$ if $S\in\nbc(M)$ and $f_S(M)=0$ otherwise.
Lemma \ref{nbc conditions}, combined with \cite[Theorem 4.2]{AFR}, implies that $f_S$ is valuative.
The functor $V(-,S)$ categorifies the homomorphism $f_S$, and Proposition \ref{V exact} says precisely that
the functor $V(-,S)$ is valuative.
\end{remark}

\begin{proof}[Proof of Proposition \ref{V exact}]
We will proceed in the same manner as the proof of Proposition \ref{triv val}.\footnote{In the special case where 
$M$ is loopless and $S=\varnothing$, Proposition \ref{V exact} follows from Proposition \ref{triv val}.}
As in that argument, let $d = d(M)$, and let 
$V_{\leq d}^\Omega(\cN,S)$ be the complex obtained from $V_\bullet^\Omega(\cN,S)$ by removing the term in degree $d+1$.
We will give a topological interpretation of this complex that is slightly different from the interpretation in Remark \ref{topological interpretation of chain complex}.

Let$$U := \becircled P(M)\cap \bigcap_{e\in E} H_{e,S}^+.$$
Since $U$ is an intersection of convex open subsets of $P(M)$, it is itself a convex open subset of $P(M)$.
By Lemma \ref{nbc conditions}, $U$ is nonempty, therefore $(\bar U,\partial U)\cong (B^d,S^{d-1})$.

For all $N\in\cN$, let $U_N := U\cap \becircled P(N)$.
Lemma \ref{nbc conditions} implies that $U_N\neq \varnothing$ if and only if $N$ is an internal face and $S\in\nbc(N)$.
The set $U$ is the disjoint union of the convex open sets $U_N$, and adding a single 0-cell gives us a cell decomposition of the
quotient $\bar U/\partial U$.
The complex $V_{\leq d}^\Omega(\cN,S)$ is precisely the cell complex that computes the reduced homology
$\tilde H_*(\bar U,\partial U) \cong \Q[-d]$.

We have an exact sequence of chain complexes
$$0 \to V_{\leq d}^\Omega(\cN,S) \to V_{\bullet}^\Omega(\cN,S) \to \Q[-d-1]\to 0.$$
We have observed that $V_{\leq d}^\Omega(\cN,S)$ has 1-dimensional homology concentrated in degree $d$,
while $\Q[-d-1]$ has 1-dimensional homology concentrated in degree $d+1$.
Just as in the proof of Proposition \ref{triv val}, the boundary map in the long exact sequence in homology is an isomorphism, 
which implies that the homology of $V_{\bullet}^\Omega(\cN,S)$
vanishes.
\end{proof}

\begin{theorem}\label{thm:OS}
The categorical invariant $\OS$ is valuative.
\end{theorem}

\begin{proof}
We need to show that, for any matroid $M$ on $E$ and any decomposition $\cN$ of $M$ with orientation $\Omega$,
$\OS(C_\bullet^\Omega(\cN))$ is exact.  By Equation \eqref{direct sum} and Proposition \ref{V exact}, $\OS(C_\bullet^\Omega(\cN))$
admits a filtration whose associated graded is exact.  
The spectral sequence of the filtered complex has $E_1$ page equal to the homology of the associated graded
and converges to the homology of the original complex.  In this case, the $E_1$ page is zero, so the original complex must be exact, as well.
\end{proof}

\revise{We end the section on the Orlik--Solomon functor with a couple of speculative remarks.}

\begin{remark}\label{shaw}
\revise{A decomposition $\cN$ of $M$ induced by a height function on $P(M)$ is called {\bf regular}.  In this case, $\cN_k$ is in canonical bijection with 
the set of codimension $k$ bounded faces of a tropical linear space $L$, and the reduced Orlik--Solomon algebra of $M$ is isomorphic to the 
degree $q$ tropical cohomology group of the Bergman fan of $M$ (see \cite{shaw-thesis} and \cite{zharkov} for precise definitions).
With this perspective, we expect that one can find a tropical proof of the exactness of $\OS(C_\bullet^\Omega(\cN))$ for regular decompositions.
We thank Kris Shaw for this observation.}
\excise{Given a tropical linear space $L$, there is an associated matroid decomposition $\cN$ of a matroid $M$; 
decompositions that arise this way are called {\bf regular}.
When $\cN$ is a regular decomposition, it is likely that there is an tropical proof of the fact that 
$\OS(C_\bullet^\Omega(\cN))$ is exact, as we outline below.

There is an inclusion reversing correspondence between $\cN_k$ and the set of codimension $k$ bounded faces of $L$, via which the complex 
$\OS(C_\bullet^\Omega(\cN))$ can be interpreted as a cellular sheaf on $L$.
If we remove the first term of this complex and take the $q^\text{th}$ homology of this complex in graded degree $q$, we obtain the {\bf tropical cohomology group}
$$H^{p,q}(L) = H_q\Big(\OS^p(C_{\leq d}^\Omega(\cN))\Big).$$
Using a deletion/contraction induction, it should be possible to prove that
this group vanishes unless $q=0$, and that $H^{p,0}(L) \cong \OS^p(M)$.  This will imply exactness of $\OS(C_\bullet^\Omega(\cN))$.
We thank Kris Shaw for explaining this approach to us, and we refer the reader to \cite{shaw-thesis} and \cite{zharkov} for precise definitions.}
\end{remark}

\begin{remark}
\revise{As was mentioned in the introduction, Speyer's proof of valuativity of the Poincar\'e polynomial $\pi_M(t)$
proceeds by first proving that the Tutte polynomial $T_M(x,y)$ is valuative, and then using the fact that $\pi_M(t)$ is a specialization of $T_M(x,y)$.
This suggests that it would be interesting to find an additive category $\cA$ with Grothendieck group $\Z[x,y]$, a categorical valuative
invariant $\mathcal{T}:\cMat\to\cA$ that categorifies the Tutte polynomial, and a functor $\Phi$ from $\cA$ to the category of graded vector spaces
over $\Q$ along with a natural isomorphism $\OS \simeq \Phi\circ\mathcal{T}$.}
\end{remark}

\section{Maximizing a linear functional}\label{sec:functional}
In this section, we state and prove Theorem \ref{everything}, which categorifies of a theorem of McMullen \cite[Theorem 4.6]{McMullen}.
Theorem \ref{everything} is the technical heart of the paper, and will be the key ingredient to the proof of Theorem \ref{convolution}.

\subsection{The statement}\label{sec:statement}
Let $\bV$ be a finite dimensional real vector space, and fix throughout this section a linear functional $\psi:\bV\to\R$.
If $P\subset\bV$ is a polyhedron with the property that the restriction of $\psi$ to $P$ is bounded above, then we define $P_\psi\subset P$
to be the face on which $\psi$ obtains its maximum value. 
If $\cQ$ is a decomposition of $P$, then $$\cQ_\psi := \{Q\in\cQ\mid Q\subset P_\psi\}$$
is a decomposition of $P_\psi$.

Consider the additive functor $$\Delta_\psi:\cPolid^+(\bV)\to\cPolid^+(\bV)$$ characterized by the following properties:
\begin{itemize}
\item If the restriction of $\psi$ to $P$ is not bounded above, then $\Delta_\psi(P)=0$.
\item If the restriction of $\psi$ to $P$ is bounded above, then $\Delta_\psi(P) = P_\psi$.
\item If $Q\subset P$ and $Q_\psi\subset P_\psi$, then $\Delta_\psi$ takes $\iota_{P,Q}\in\Hom(P,Q)$ to $\iota_{P_\psi,Q_\psi}\in\Hom(P_\psi,Q_\psi)$.
\end{itemize}
\begin{remark}
If $Q\subset P$ and the restriction of $\psi$ to $P$ is bounded above but the maximum value of $\psi$ on $P$ is strictly greater
than the maximum value of $\psi$ on $Q$, then $\Hom(P_\psi,Q_\psi) = 0$, so $\Delta_\psi$ necessarily takes $\iota_{P,Q}\in\Hom(P,Q)$ to zero.
\end{remark}

\excise{
For a polyhedron $P$ we have $\Delta_{\psi}(P) = P$ if and only if $\psi$ is constant on $P$. For any $P$, clearly $\psi$ is constant on the face $P_{\psi}$. Thus
the functor $\Delta_{\psi}$ is idempotent: $\Delta_{\psi} \circ \Delta_{\psi} = \Delta_{\psi}$ is an equality of functors.  Moreover, if $\psi$ is constant on $P$, then it is constant on any face of $P$.} 

\begin{remark}
We may think of the functor $\Delta_{\psi}$ as projection from $\cPolid^+(\bV)$ to the full subcategory $\cPolid^+(\bV)_{\psi}$, whose objects are formal sums of polyhedra on which $\psi$ is a constant function. This category $\cPolid^+(\bV)_{\psi}$ splits as a direct sum of categories of polyhedra on each level set of $\psi$.
\end{remark}

Let $\cI_{\psi}(\bV)\subset \cI(\bV)$ be the localizing subcategory of $\Ch_b(\cPolid^+(\bV)_{\psi})$ generated by the complexes $C_{\bullet}^{\Omega}(\cQ)$ for oriented decompositions $\cQ$ of polyhedra $P$ on which $\psi$ is constant.

\begin{theorem}\label{everything}
Suppose that $\cQ$ is a decomposition of $P\subset\bV$ and $\Omega$ is an orientation of $\cQ$.
\begin{itemize}
\item If the restriction of $\psi$ to $P$ is not bounded above, then the complex $\Delta_\psi(C_\bullet^\Omega(\cQ))$ is contractible.
\item If the restriction of $\psi$ to $P$ is bounded above, 
the complex $\Delta_\psi(C_\bullet^\Omega(\cQ))$ is homotopy equivalent to a shift of $C_\bullet^{\Omega_\psi}(\cQ_\psi)$ for some
(equivalently any) orientation $\Omega_\psi$ of $\cQ_\psi$.
\end{itemize}
Thus $\Delta_{\psi}$ sends $\cI(\bV)$ to $\cI_{\psi}(\bV) \subset \cI(\bV)$. 
\end{theorem}

\begin{example}
Consider the triangle $P\subset \mathbb{R}^2$ along with the decomposition $\cQ$ shown in Figure \ref{triangle} whose maximal faces are four smaller triangles.
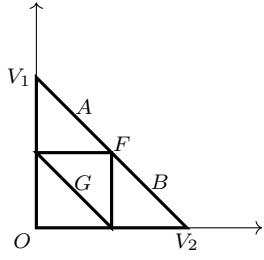
\begin{figure}[h] 
\begin{center}
\begin{tikzpicture}
	\draw[->] (0,0) -- (3,0);
	\draw[->] (0,0) -- (0,3);
	\draw[very thick] (0,0) -- (2,0) -- (0,2) -- cycle;
	\draw[very thick] (1,0) -- (1,1) -- (0,1) -- cycle;
	\node[label=above right:{\scriptsize $F$}] at (0.75,0.75) {};
	\node[label=above right:{\scriptsize $A$}] at (0.25,1.25) {};
	\node[label=above right:{\scriptsize $B$}] at (1.25,0.25) {};
	\node[label=above right:{\scriptsize $G$}] at (0.23,0.23) {};
	\node[label=below left:{\scriptsize $O$}] at (0.2,0.2) {};
	\node[label=left:{\scriptsize $V_1$}] at (0.2,2) {};
	\node[label=below:{\scriptsize $V_2$}] at (2,0.2) {};
\end{tikzpicture}
\caption{A polytope in $\mathbb{R}^2$ and its decomposition.}\label{triangle}
\end{center}
\end{figure}
Consider the linear functional $\psi = x_1+x_2$, which is maximized on $P$ by the hypotenuse $P_\psi$. The family $Q_\psi = \{V_1, A, F, B, V_2\}$ is a decomposition of $P_\psi$. 
The complex $\Delta_\psi(C_\bullet^\Omega(\cQ))$ decomposes as a direct sum of two pieces.  The first piece is the contractible complex $\Nul(G,2)$, coming from the lower-left triangle and the edge $G$.
The second piece has the following shape (ignoring signs):
\[
\begin{tikzcd}
& & A \arrow[r] & F \arrow[rd] &   \\
 & & F \arrow[ru] \arrow[r] & F \arrow[r]  &   0. \\
0 \arrow[r] & P_\psi \arrow[ruu] \arrow[ru] \arrow[r] & B \arrow[ru] 
\end{tikzcd}
\]
Here the three copies of $F$ come from applying $\Delta_\psi$ to the middle triangle, its northern edge, and its eastern edge, all of which are internal faces of $\cQ$.
This complex is not minimal, as the $(F,F)$ component of the differential is nontrivial.  However, it is homotopy equivalent to the complex $C_\bullet^{\Omega_\psi}(\cQ_\psi)$.

In general, the complex $\Delta_\psi(C_\bullet^\Omega(\cQ))$ will decompose as a direct sum, with summands indexed by the maximum values achieved by $\psi$ on various internal faces of $\cQ$.
Theorem \ref{everything} says that all but one of those summands will be contractible, and the one corresponding to the maximum of $\psi$ on $P$ (assuming that $\psi$ is bounded on $P$) 
will be homotopy equivalent to $C_\bullet^{\Omega_\psi}(\cQ_\psi)$.
\end{example}

\subsection{Geometry}\label{sec:geometry}
In this section, we give the geometric constructions that we will need for the proof of Theorem \ref{everything}.
Let $\cQ$ be a decomposition of a polyhedron in $\bV$, and let $F\in\cQ$ be any face.
Informally, we define the {\bf local fan} $\Sigma_F(\cQ)$ to be the fan that one sees when one looks at $\cQ$ in a small neighborhood
of a point in the relative interior of $F$.  More precisely, for any $G\in\cQ$ with $F\subset G$, we define the cone
$$\sigma_G := \{\lambda (v-v')\mid v\in G, v'\in F, \lambda \in \revise{\R_{\geq 0}}\},$$
and we put $\Sigma_F(\cQ) := \{\sigma_G\mid F\subset G\in\cQ\}$.  Let $\bV_F := \sigma_F$, which is a linear subspace of $\bV$.
The vector space $\bV_F$ acts freely by translation on every cone in $\Sigma_F(\cQ)$, and we may therefore define
the cone $\tilde\sigma_G := \sigma_G/\bV_F$ for every $F\subset G\in \cQ$ and 
the {\bf reduced local fan}
$$\tilde\Sigma_F(\cQ) := \{\tilde \sigma_G\mid F\subset G\in\cQ\},$$
which is a pointed fan in the vector space $\bV/\bV_F$.

Let $\psi$ be a nonzero linear functional on $\bV$.  
Let $\cQ$ be a decomposition of a polyhedron $P\subset \bV$ on which $\psi$ is bounded above,
and let $F\in\cQ$ be any face on which $\psi$ is constant.
Then $\psi$ descends to a linear functional $\tilde\psi$ on $\bV/\bV_F$.  Let $\mathbb{H}_{F,\psi}\subset \bV/\bV_F$ \revise{be the preimage of $-1$ under the map $\tilde\psi$},
let $R_{F,\psi} := \mathbb{H}_{F,\psi}\cap \operatorname{Supp}\tilde\Sigma_F(\cQ_\psi)$, and let $$\cR_{F,\psi} := \{\tilde\sigma\cap \mathbb{H}\mid \tilde\sigma\in \tilde\Sigma_F(\cQ)\}$$
be the induced decomposition of $R_{F,\psi}$.  

Let $\cB_{F,\psi}\subset\cR_{F,\psi}$ be the collection of bounded faces.
Let $$\cQ_{F,\psi} := \{Q\in\cQ\mid Q_\psi = F\text{ and } Q \neq F\}.$$  We have a bijection
\begin{equation}\label{B}
\begin{aligned}
\cQ_{F,\psi} &\to \cB_{F,\psi}\\
Q &\mapsto \tilde\sigma_Q\cap\bH_{F,\psi}.
\end{aligned}
\end{equation}
This bijection takes faces of dimension $k+1+\dim F$ to faces of dimension $k$, and it restricts to a bijection
between elements of $\cQ_{F,\psi}$ that lie on the boundary of $P$ and faces of $\cB_{F,\psi}$ that lie on the boundary of $R_{F,\psi}$.

\begin{example}\label{slice}
Consider the decomposition shown in Figure \ref{diamond}, with $\psi$ equal to the height function.
The fan $\tilde\Sigma_F(\cQ) = \Sigma_F(\cQ)$ is complete, with one vertex, four rays, and four cones of dimension 2.
The polyhedron $\cR_{F,\psi}$ is equal to the line $\bH_{F,\psi}$, and the decomposition $\cR_{F,\psi}$ of $R_{F,\psi}$
has two rays, two vertices, and one interval.  The rays are equal to $\sigma_{H_1}\cap \bH_{F,\psi}$ and $\sigma_{H_3}\cap \bH_{F,\psi}$,
the two vertices are equal to $\sigma_{G_2}\cap \bH_{F,\psi}$ and $\sigma_{G_3}\cap \bH_{F,\psi}$,
and the interval is equal to $\sigma_{H_2}\cap \bH_{F,\psi}$.  Thus the set $\cQ_{F,\psi} = \{G_2,G_3,H_2\}$ is in 
canonical bijection with the bounded complex $\cB_{F,\psi}$.

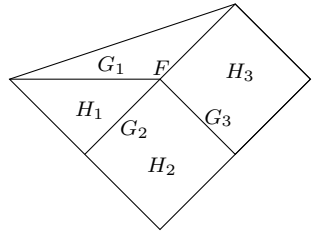
\begin{figure}[h]
\begin{center}
	\begin{tikzpicture}  
	\draw (1,1) -- (2,0) -- (0,-2) -- (-2,0) -- cycle;
	\draw (0,0) -- (1,1) -- (2,0) -- (1,-1) -- cycle;
	\draw (0,0) -- (-1,-1);
	\draw (0,0) -- (-2,0);
	\node[label=above right:{\scriptsize $G_1$}] at (-1.1,-0.2) {};
	\node[label=above right:{\scriptsize $H_3$}] at (0.6,-0.3) {};
	\node[label=above right:{\scriptsize $H_1$}] at (-1.4,-0.8) {};
	\node[label=above right:{\scriptsize $H_2$}] at (-0.44,-1.55) {};
	\node[label=above right:{\scriptsize $G_2$}] at (-0.8,-1.05) {};
	\node[label=above right:{\scriptsize $G_3$}] at (0.32,-0.9) {};
	\node[label=below left:{\scriptsize $F$}] at (0.4,0.52) {};
	\end{tikzpicture} 
\caption{A decomposition $\cQ$ of a quadrilateral, with some of the internal faces labeled.}
\label{diamond}
\end{center}
\end{figure}
\end{example}

We will also need a relative version of this construction.
Suppose we are given $F\subset G\in\cQ$, 
with $F$ a facet of $G$ and $\psi$ constant on $G$.
Let $\cL_{F,G,\psi}\subset\cR_{F,\psi}$ be the collection of faces 
that are either bounded or have recession cone equal to the ray $\tilde\sigma_G$.  Informally, these are the polyhedra
that are unbounded in at most one direction, namely that of the inward normal vector to $F$ in $G$.
We then have a bijection 
\begin{equation}\label{L}
\begin{aligned}
\cQ_{G,\psi} &\to \cL_{F,G,\psi}\setminus \cB_{F,\psi}\\
Q &\mapsto \tilde\sigma_Q\cap\bH_{F,\psi}.
\end{aligned}
\end{equation}

\begin{example}
Continuing with the picture in Example \ref{slice}, we have $$\cQ_{G_1,\psi} = \{H_1\} \and 
\cL_{F,G_1,\psi}\setminus \cB_{F,\psi} = \sigma_{H_1}\cap \bH_{F,\psi}.$$
\end{example}

\subsection{Algebra}\label{sec:algebra}
In this section, we prove Theorem \ref{everything}.  Let $P$ be a polyhedron in $\bV$, let $\cQ$ be a decomposition of $P$,
and let $\psi$ be a linear functional on $\bV$.  We may assume that $\psi$ is nonconstant on $P$, as Theorem \ref{everything}
is trivial in this case.

Let $F\in\cQ$ be any face on which $\psi$ is constant, and let $C(F)_\bullet := \Delta_\psi(C_\bullet^\Omega(\cQ))^F$ be the subquotient of $\Delta_\psi(C_\bullet^\Omega(\cQ))$ consisting of all copies of the object $F$
(see Section \ref{thin}).
Recall that $\tau:\cP\to\Vec_\Q$ is the functor that takes every polyhedron to the vector space $\Q$ and every linear 
automorphism of $\bV$ to the identity morphism.

\begin{lemma}\label{topological interpretation of DF}
If $F$ lies on the boundary of $P$, then $\tau\big(C(F)_\bullet\big)[1+\dim F]$ is homotopy equivalent to the singular chain complex for the pair $(R_{F,\psi},\partial R_{F,\psi})$.  If $F$ is an internal face of $\cQ$, then $\tau\big(C(F)_\bullet\big)$ is contractible.
\end{lemma}

\begin{proof}
First suppose that $F$ lies on the boundary of $P$, which means that $F$ itself does not appear in $C_\bullet^\Omega(\cQ)$.
Let $B_{F,\psi}\subset R_{F,\psi}$ be the union of all of the elements of $\cB_{F,\psi}$.
Combining Remark \ref{topological interpretation of chain complex} with
the bijection \eqref{B} from Section \ref{sec:geometry}, we may identify the complex
$\tau\big(C(F)_\bullet\big)[1+\dim F]$ with the cellular chain complex for the pair
$(B_{F,\psi}, B_{F,\psi}\cap \partial R_{F,\psi})$,
which is homotopy equivalent to $(R_{F,\psi},\partial R_{F,\psi})$ by Lemma \ref{bounded complex}.

Now suppose that $F$ is an internal face of $\cQ$.  In this case, $F$ does appear in $C_\bullet^\Omega(\cQ)$, which leads to a small modification of the argument above.  We now have a termwise-split short exact sequence
\begin{equation}\label{augmentation}0 \to \Q[1] \to \tau\big(C(F)_\bullet\big)[1+\dim F]\to \tau\big(C(F)_{> \dim F}\big)[1+\dim F]\to 0.\end{equation}
The complex $\tau\big(C(F)_\bullet\big)[\dim F]$ is isomorphic to the cone of the augmentation map 
$$f:\tau\big(C(F)_{> \dim F}\big)[1+\dim F]\to\Q,$$ so that Equation \eqref{augmentation} may be identified with a shift of Equation \eqref{coneses}.

As in the first paragraph of this proof, the quotient complex
$\tau\big(C(F)_{> \dim F}\big)[1+\dim F]$ may be identified with the cellular chain complex for the pair $(B_{F,\psi},B_{F,\psi}\cap \partial R_{F,\psi})$,
which is homotopy equivalent to $(R_{F,\psi},\partial R_{F,\psi})$.  
The complex $\tau\big(C(F)_\bullet\big)[\dim F]$, being the cone of the augmentation map, may be identified with the {\em reduced}
cellular chain complex.
Since $F$ is internal, the reduced local fan $\Sigma_F(\cQ)$ is complete, 
which implies that $R_{F,\psi} = \bH_{F,\psi}$ and $\partial R_{F,\psi} = \varnothing$.  The reduced homology of the pair $(R_{F,\psi},\partial R_{F,\psi})$ is trivial, which implies that $\tau\big(C(F)_\bullet\big)[\dim F]$ is contractible.
\end{proof}

\begin{lemma}\label{ignore the submaximal stuff and the boundary}
If any of the following three conditions hold, then $C(F)_\bullet$ is contractible:
\begin{enumerate}
\item The restriction of $\psi$ to $P$ is not bounded above.
\item The restriction of $\psi$ to $P$ is bounded above but $F$ is not contained in $P_\psi$.
\item The restriction of $\psi$ to $P$ is bounded above and $F$ is contained in $\partial P_\psi$.
\end{enumerate}
\end{lemma}

\begin{proof}
The complex $C(F)_\bullet$ lives in the full subcategory of $\cPol_{\id}(\bV)$ consisting of direct sums of copies of $F$.
This subcategory is equivalent via the trivial functor $\tau$ to the category of finite dimensional $\Q$-vector spaces, thus
it is sufficient to prove that $\tau(C(F)_\bullet)$ is contractible.

The case where $F$ is an internal face of $\cQ$ is treated in Lemma \ref{topological interpretation of DF}, so we may assume that $F$ is contained in the boundary of $P$.
In this case, Lemma \ref{topological interpretation of DF} tells us that it is sufficient to prove that $\partial R_{F,\psi}$ is a deformation retract of $R_{F,\psi}$.
If $\psi$ is bounded below on $P$ and achieves its minimum on $F$, then $R_{F,\psi} = \varnothing$, and we are done.
Thus we may assume that $\psi(F)$ is neither the minimum nor the maximum of $\psi$ on $P$.  

Since we know that $F$
is contained in the boundary of $P$, this implies that there exists a point $v\in P\setminus F$ with $\psi(v) = \psi(F)$.
Choose a point $v'\in F$, and let $w$ be the image of $v-v'$ in $\bV/\bV_F$.  Then $w$ is contained in the recession cone
of $R_{F,\psi}$, and $-w$ is not.  This implies that $R_{F,\psi}$ is unbounded and $\partial R_{F,\psi}$ is nonempty, which in turn
implies that $\partial R_{F,\psi}$ is a deformation retract of $R_{F,\psi}$.
\end{proof}

Our next lemma addresses the one case not covered by Lemma \ref{ignore the submaximal stuff and the boundary}.

\begin{lemma}\label{here's an important term}
Suppose that the restriction of $\psi$ to $P$ is bounded above and
$F$ is an internal face of $\cQ_\psi$.
Then $C(F)_\bullet$ is homotopy equivalent to $F[-\dim P + \dim P_\psi - \dim F]$.
\end{lemma}

\begin{proof}
As in the proof of Lemma \ref{ignore the submaximal stuff and the boundary}, 
it is sufficient to prove that $\tau(C(F)_\bullet)$ has one dimensional homology concentrated in degree $\dim P - \dim P_\psi + \dim F$.
By Lemma \ref{topological interpretation of DF}, this is equivalent to proving that the pair $(R_{F,\psi},\partial R_{F,\psi})$
has one dimensional homology, concentrated in degree $\dim P - \dim P_\psi - 1$.
The polyhedron $R_{F,\psi}$ is isomorphic to the product of the vector space $\bV_P/\bV_F$ with the quotient polytope
$P/P_\psi$.
The result then follows from the fact that $P/P_\psi$ is homeomorphic to a closed ball of dimension $\dim P - \dim P_\psi - 1$.
\end{proof}

Lemmas \ref{ignore the submaximal stuff and the boundary} and \ref{here's an important term} together allow us to
identify the minimal complex of $C(F)_\bullet$ for any $F\in\cQ$ on which $\psi$ is constant.  The next lemma
tells us how two of these minimal complexes interact.
Suppose that the restriction of $\psi$ to $P$ is bounded above, $F$ and $G$ are both internal faces of $\cQ_\psi$, and
$F$ is a facet of $G$.
Let $C(F,G)_\bullet$ be the subquotient of $\Delta_\psi(C_\bullet^\Omega(\cQ))$ consisting of all copies of $F$ and $G$.
Note that this complex has $C(F)_\bullet$ as a termwise-split subcomplex, and the termwise-split quotient is isomorphic to $C(G)_\bullet$.

\begin{lemma}\label{what are the maps?}
The complex $\tau\big(D(F,G)_\bullet\big)$ is contractible.
\end{lemma}

\begin{proof}
Let $L_{F,G,\psi}$ be the union of the elements of $\cL_{F,G,\psi}$.
Combining Remark \ref{topological interpretation of chain complex} with the two bijections \eqref{B} and \eqref{L} in Section \ref{sec:geometry} allows us to identify $\tau\big(D(F,G)_\bullet\big)$ 
with a shift of the cellular chain complex for the pair $(\hat L_{F,G,\psi}, L_{F,G,\psi} \cap \partial R_{F,\psi} \cup \{\star\})$,
where $\star\in\hat L_{F,G,\psi}$ is the point at infinity.
Translation in the direction of the ray $\tilde\sigma_G\subset \bH_{F,\psi}$ defines a deformation retraction
 from this pair to the pair $(\{\star\}, \{\star\})$, therefore this chain complex is contractible.
\end{proof}

We are now ready to prove Theorem \ref{everything}.

\begin{proof}[Proof of Theorem \ref{everything}.]
Let $D_\bullet$ be a minimal complex for $\Delta_\psi(C_\bullet^\Omega(\cQ))$.
By Lemmas \ref{ignore the submaximal stuff and the boundary} and \ref{here's an important term} and the discussion in Section \ref{thin}, $D_\bullet$ has either one or zero copies of each face $F \in \cQ$, depending on whether or not $F$ is an internal face of $\cQ_\psi$, in which case that copy appears 
in degree $\dim F + \dim P - \dim P_\psi$. Thus $D_\bullet = 0$ if the restriction of $\psi$ to $P$ is not bounded above, and otherwise
\begin{equation*} D_\bullet[\dim P - \dim P_\psi] \cong C_\bullet^{\Omega_\psi}(\cQ_\psi)\end{equation*}
as graded objects of $\cPolid(\bV)$ for some (equivalently any) orientation $\Omega_\psi$ of $\cQ_\psi$. Given a pair $F\subset G$ of internal faces of $\cQ_\psi$ with $F$ a facet of $G$, 
Lemma \ref{what are the maps?} implies that
the corresponding component of the differential in $D_\bullet$ is equal to an invertible multiple of $\iota_{G,F}$.
By Proposition \ref{there's an orientation}, this implies that $D_\bullet[\dim P - \dim P_\psi]$ is isomorphic to 
$C_\bullet^{\Omega_\psi}(\cQ_\psi)$ as a complex.
\end{proof}

\section{Convolution}\label{sec:convolution}
The main results of this section are Theorem \ref{hopf} and Corollary \ref{convolution}, which categorify \cite[Theorems A and C]{ArdilaSanchez}.
In particular, Theorem \ref{convolution} will provide a valuable tool for constructing new valuative functors from old ones.

\subsection{A product and coproduct for matroids}\label{sec:product}
Consider a finite set $E = E_1 \sqcup E_2$. Let $M_1$ be a matroid $E_1$ and $M_2$ a matroid on $E_2$.  As in Section \ref{sec:relaxation}, we write $M_1 \sqcup M_2$ to denote the direct sum of $M_1$ and $M_2$, which is a matroid on $E_1 \sqcup E_2$.  We write $\cMatid(E)_{\sqcup}$ to denote the full subcategory of $\cMatid(E)$ whose objects are matroids of the form $M_1 \sqcup M_2$.

In this section, we will define functors
\[ m:\cMatid^+(E_1)\boxtimes\cMatid^+(E_2)\to\cMatid^+(E)\and \Delta:\cMatid^+(E)\to \cMatid^+(E_1)\boxtimes\cMatid^+(E_2). \]
The {\bf product functor} $m$ is characterized by putting $m(M_1\boxtimes M_2) := M_1\sqcup M_2$ and $$m(\iota_{M_1,M_1'}\boxtimes\iota_{M_2,M_2'}) = \iota_{M_1\sqcup M_2,M_1'\sqcup M_2'}.$$ We observe that $m$ is a fully faithful embedding,
with essential image $\cMatid^+(E)_{\sqcup}$. In particular, we may write $m^{-1}$ for the inverse functor from the essential image. For example, we have $$m^{-1}(M_1 \sqcup M_2) = M_1 \boxtimes M_2 \in \cMatid^+(E_1)\boxtimes\cMatid^+(E_2).$$


Given a matroid $M$ on $E$, the \revise{{\bf restriction}} $M^{E_1}$ is defined to be the matroid on $E_1$ obtained by deleting $E_2$,
and the {\bf contraction} $M_{E_1}$ is defined to be the matroid on $E_2$ obtained by contracting a basis for $M^{E_1}$ and deleting the remaining elements of $E_1$.
We also define 
\begin{equation} M(E_1,E_2) := M^{E_1}\sqcup M_{E_1},\end{equation}
which is again a matroid on $E$. 
If $M = M_1 \sqcup M_2$, then $M(E_1, E_2) = M$ , thus the operation $M \mapsto M(E_1, E_2)$ is idempotent.  

Of crucial importance is the following connection between the operation $M \mapsto M(E_1, E_2)$ and the maximization of a linear functional on base polytopes.
Recall that we have a linear functional $\revise{\delta}_{E_1}:\R^E\to\R$ that takes the sum of the $E_1$ coordinates.

\begin{lemma}\label{break}
For any matroid $M$ on the set $E$, we have
\begin{equation} P(M)_{\revise{\delta}_{E_1}}= P(M(E_1,E_2)) = P(M^{E_1}) \times P(M_{E_1}).\end{equation}
\end{lemma}

\begin{proof}
\revise{Given an arbitrary basis $B$ for $M$, we have $|B\cap E_1| \leq \rk_M(E_1)$. 
The bases for $M(E_1,E_2)$ are precisely those bases $B$ for $M$ such that $|B\cap E_1| = \rk_M(E_1)$.}
Then $\revise{\delta}_{E_1}(v_B) = |B\cap E_1|$, thus $v_B\in P(M)_{\revise{\delta_{E_1}}}$ if and only if $B$ is a basis for $M(E_1,E_2)$.
\end{proof}

Using Lemma \ref{break}, we may define a functor $$\Delta_{\sqcup} := \Delta_{\revise{\delta}_{E_1}}:\cMatid^+(E)\to\cMatid^+(E)_\sqcup.$$
We then use this to define the {\bf coproduct functor} $$\Delta := m^{-1}\circ\Delta_{\sqcup}:\cMatid^+(E)\to \cMatid^+(E_1)\boxtimes\cMatid^+(E_2).$$
In concrete terms, we have $\Delta(M) = M^{E_1}\boxtimes M_{E_1}$ and 
$$\Delta(\iota_{M,M'}) = \begin{cases} \iota_{M^{E_1}, (M')^{E_1}}\boxtimes \iota_{M_{E_1},M'_{E_1}} & \text{if $\rk_M(E_1) = \rk_{M'}(E_1)$}\\ 0 & \text{otherwise.}\end{cases}$$

\begin{remark}
The product functor $m:\cMatid^+(E_1)\boxtimes\cMatid^+(E_2)\to\cMatid^+(E)$ is induced by a functor $\cMatid(E_1)\times\cMatid(E_2)\to\cMatid(E)$, but the coproduct functor
$\Delta:\cMatid^+(E)\to \cMatid^+(E_1)\boxtimes\cMatid^+(E_2)$ is not induced by any functor $\cMatid(E)\to \cMatid(E_1)\times\cMatid(E_2)$.  This is because the coproduct functor sends some morphisms to zero.
\end{remark}

Let $M = M_1 \sqcup M_2$.  If $\cN_1$ is a decomposition of $M_1$ and $\cN_2$ is a decomposition of $M_2$, then
$$\cN_1\sqcup \cN_2 := \{N_1\sqcup N_2\mid N_1\in\cN_1 \text{ and } N_2\in\cN_2\}$$
is a decomposition of $M$.  Furthermore, every decomposition of $M$ is of this form \cite[Corollary 4.8]{LdMRS}.

\subsection{Categorical Hopf ideals}
The purpose of this section is to prove that the product and coproduct functors interact nicely with the localizing subcategories $\cI(E)\subset \Ch_b(\cMatid^+(E))$ defined in Section \ref{sec:valfun}.
The following theorem, which categorifies \cite[Theorem A]{ArdilaSanchez}, says that these subcategories satisfy a condition analogous to that of a Hopf ideal in a Hopf monoid.

\begin{theorem}\label{hopf}
We have the following inclusions of subcategories:
\begin{equation} \label{mpreservesI} m\big(\cI(E_1)\boxtimes\cMatid^+(E_2)\big)\subset \cI(E), \qquad m\big(\cMatid^+(E_1)\boxtimes\cI(E_2)\big)\subset \cI(E) \end{equation}
and
\begin{equation} \label{deltapreservesI} \Delta\big(\cI(E)\big)\subset \Big\langle \cI(E_1)\boxtimes\cMatid^+(E_2),\, \cMatid^+(E_1)\boxtimes\cI(E_2) \Big\rangle. \end{equation}
\end{theorem}

There are three localizing subcategories of $\Ch_b(\cMatid^+(E))$ at play here.  The first is $\cI(E)$.
The second, which we will call $\cI(E)_\sqcup$, is generated by those complexes of the form $C_\bullet^\Omega(\cN_1\sqcup\cN_2)$, where $\cN_1$ is a decomposition of a matroid on $E_1$
and $\cN_2$ is a decomposition of a matroid on $E_2$.  Finally, we will need to consider the localizing subcategory
$$\cI(E_1,E_2) := \left\langle m(\cI(E_1)\boxtimes\cMatid^+(E_2)), m(\cMatid^+(E_1)\boxtimes\cI(E_2)) \right\rangle.$$
This is the localizing subcategory generated by those complexes of the form $C_\bullet^\Omega(\cN_1\sqcup\cN_2)$, where either $\cN_1$ is the trivial decomposition of a matroid on $E_1$
or $\cN_2$ is the trivial decomposition of a matroid on $E_2$.  We clearly 
have $$\cI(E_1,E_2)\subset \cI(E)_{\sqcup}\subset\cI(E),$$
which in particular implies Equation \eqref{mpreservesI}.

\begin{lemma} \label{lem:summarizeeverything} We have $\Delta_\sqcup(\cI(E))\subset \cI(E)_\sqcup$. \end{lemma}

\begin{proof} Theorem \ref{everything} and Lemma \ref{break} together imply that, for any matroid $M$ on $E$ with decomposition $\cN$ and orientation $\Omega$, the complex $\Delta_{\sqcup}(C^{\Omega}_{\bullet}(\cN))$ is either homotopy equivalent to zero or homotopy equivalent to a shift of the complex $C^{\Omega'}_{\bullet}(\cN')$, where $\cN'$ is a decomposition of $M(E_1,E_2)$.
As we noted at the end of Section \ref{sec:product}, $\cN'$ is necessarily equal to $\cN_1\sqcup\cN_2$ for some decompositions $\cN_1$ of $M^{E_1}$ and $\cN_2$ of $M_{E_1}$ \cite[Corollary 4.8]{LdMRS}.
\end{proof}

\begin{lemma}\label{compare1}
We have $\cI(E_1,E_2)= \cI(E)_{\sqcup}$.
\end{lemma}

\begin{proof} 
Let $\cN = \cN_1\sqcup\cN_2$ be a decomposition of $M = M_1 \sqcup M_2$, and let $\Omega$ be an orientation of $\cN$. We need to show that $C_\bullet^\Omega(\cN)\in\cI(E_1,E_2)$. Because the isomorphism class of $C_{\bullet}^{\Omega}(\cN)$ is independent of the choice of orientation, we may assume that $\Omega$ is induced from an orientation $\Omega_1$ of $\cN_1$ and an orientation $\Omega_2$ of $\cN_2$. 

Let $d_i = d(M_i)$ and $C_i := C_{\le d_i}^{\Omega_i}(\cN_i)$.
Then 
\begin{equation} \label{bigpartthesame} C_{\le d_1 + d_2}^\Omega(\cN) = m(C_1 \boxtimes C_2). \end{equation}
The complex $C_{\bullet}^\Omega(\cN)= \Cone(\alpha_{\cN})$ is obtained from the complex of \eqref{bigpartthesame} by adding one new summand $M_1 \sqcup M_2$ in degree $d_1+d_2+1$. 

Let $D_\bullet := C_\bullet^{\Omega_1}(\cN_1)\boxtimes C_\bullet^{\Omega_2}(\cN_2)\in \Ch_b(\cMatid^+(E_1)\boxtimes \cMatid^+(E_2))$.  Explicitly, we have
\begin{equation*} \label{FOIL} D_\bullet =
\left(
\begin{array}{c} M_1[-d_1-1] \boxtimes M_2[-d_2-1] \\ \oplus \\ M_1[-d_1-1] \boxtimes C_2 \\ \oplus \\ C_1 \boxtimes M_2[-d_2-1] \\ \oplus \\ C_1 \boxtimes C_2 \end{array},
\left( \begin{array}{cccc}
0 & 0 & 0 & 0 \\
\pm \id_{M_1} \boxtimes \alpha_{\cN_2} & \pm \id \boxtimes \pa  & 0 & 0 \\
\alpha_{\cN_1} \boxtimes \id_{M_2} & 0 & \pa \boxtimes \id & 0 \\
0 & \alpha_{\cN_1} \boxtimes \id & \mp \id \boxtimes \alpha_{\cN_2} & \pa \boxtimes \pa \end{array} \right)
\right). \end{equation*} 
We represent $D_\bullet$ by the following schematic diagram (with signs suppressed):
\begin{equation*} D_\bullet = \left( \begin{tikzpicture}
\node (a) at (0,2) {$M_1 \sqcup M_2$};
\node (d) at (4,2) {$M_1 \sqcup C_2$};
\node (e) at (4,0) {$C_1 \sqcup M_2$};
\node (f) at (8,0) {$C_1 \sqcup C_2$};
\path
	(a) edge[->] node[descr] {\tiny $\id \sqcup \alpha_{\cN_2}$} (d)
	(a) edge[->] node[descr] {\tiny $\alpha_{\cN_1} \sqcup \id$} (e)
	(d) edge[->] node[descr] {\tiny $\alpha_{\cN_1} \sqcup \id$} (f)
	(e) edge[->] node[descr] {\tiny $\id \sqcup \alpha_{\cN_2}$} (f);
\end{tikzpicture} \right).
\end{equation*}

We next construct a chain map $$\beta\co\Nul(M_1 \sqcup M_2,d_1+d_2+1)\to m(D_\bullet).$$ 
As objects, we have
$$\Nul(M_1 \sqcup M_2,d_1+d_2+1) = (M_1\sqcup M_2)[-d_1-d_2-1] \oplus (M_1 \sqcup M_2)[-d_1-d_2]$$
and
$$\begin{aligned}
m(D_\bullet) &= (M_1\sqcup M_2)[-d_1-d_2-2] \oplus m(M_1\boxtimes C_2)[-d_1-1]  \oplus m(C_1 \boxtimes M_2)[-d_2-1] \oplus 
m(C_1 \boxtimes C_2)\\
&= (M_1\sqcup M_2)[-d_1-d_2-2] \oplus m(M_1\boxtimes C_2)[-d_1-1]  \oplus m(C_1 \boxtimes M_2)[-d_2-1] \oplus 
C_{\le d_1+d_2}^\Omega(\cN).
\end{aligned}$$
With respect to these decompositions, we encode $\beta$ as the following matrix:

\begin{equation*} \beta = \left( \begin{array}{cc} 0 & 0 \\ 0 & 0 \\ m(\alpha_{\cN_1} \boxtimes \id_{M_2}) & 0 \\ 0 & \alpha_{\cN} \end{array} \right). \end{equation*}
Here the two nontrivial pieces of $\beta$ go from the first summand of $\Nul(M_1 \sqcup M_2,d_1+d_2+1)$
to the third summand of $m(D_\bullet)$ and from the second summand of $\Nul(M_1 \sqcup M_2,d_1+d_2+1)$
to the fourth summand of $m(D_\bullet)$.
We define $B_\bullet := \Cone(\beta)$, which we represent by the following schematic diagram (again with signs suppressed):
\begin{equation*} \label{conebeta} B_\bullet = \left( \begin{tikzpicture}
\node (a) at (0,2) {$M_1 \sqcup M_2$};
\node (b) at (0,-2) {{\color{red}$M_1 \sqcup M_2$}};
\node (c) at (4,-2) {{\color{red}$M_1 \sqcup M_2$}};
\node (d) at (4,2) {$M_1 \sqcup C_2$};
\node (e) at (4,0) {$C_1 \sqcup M_2$};
\node (f) at (8,0) {$C_1 \sqcup C_2$};
\path
	(a) edge[->] node[descr] {\tiny $\id \sqcup \alpha_{\cN_2}$} (d)
	(a) edge[->] node[descr] {\tiny $\alpha_{\cN_1} \sqcup \id$} (e)
	(d) edge[->] node[descr] {\tiny $\alpha_{\cN_1} \sqcup \id$} (f)
	(e) edge[->] node[descr] {\tiny $\id \sqcup \alpha_{\cN_2}$} (f)
	(b) edge[->] node[descr] {\tiny {\color{red} $\id$}} (c)
	(b) edge[->] node[descr] {\tiny $\alpha_{\cN_1} \sqcup \id$} (e)
	(c) edge[->] node[descr] {\tiny $\alpha_{\cN}$} (f);
\end{tikzpicture} \right).
\end{equation*}
\excise{
To verify that $\beta$ is a chain map, we verify that $\pa \circ \beta - \beta \circ \pa = 0$ when restricted to each summand of $\Nul(M_1 \sqcup M_2,d_1+d_2+1)$. On $M_1 \sqcup M_2[-d_1-d_2]$, the result is zero since $\alpha_{\cN}$ is a chain map to $C_{\le d_1+d_2}^{\Omega}(\cN)$. On $M_1 \sqcup M_2[-d_1-d_2-1]$, the result is zero because 
\begin{equation} \label{eq:alphas} m(\alpha_{\cN_1} \boxtimes \alpha_{\cN_2}) = \alpha_{\cN}\end{equation} via the isomorphism \eqref{bigpartthesame}. Both sides of \eqref{eq:alphas} send $M_1 \sqcup M_2$ to the (signed) sum over top dimensional internal faces of $\cN$, which are all obtained as $N_1 \sqcup N_2$ for top dimensional internal faces of $\cN_1$ and $\cN_2$ respectively.
}
We observe that $B_\bullet$ is a convolution with six parts, three of which are shifts of $M_1 \sqcup M_2$.
The leftmost terms live in homological degree $d_1 +d_2 + 2$, and the other copy of $M_1 \sqcup M_2$ is in homological degree $d_1 + d_2 + 1$.  The remaining terms are complexes that begin in the homological degree indicated in the diagram. 
The terms appearing in red make up the termwise-split quotient complex $\Nul(M_1 \sqcup M_2,d_1+d_2+1)[1]$.

Let us draw $B_\bullet$ again, regrouping the six parts into three groups of two which we emphasize with color:
\begin{equation} \label{conebeta2} B_\bullet = \left( \begin{tikzpicture}
\node (a) at (0,1) {{\color{blue}$M_1 \sqcup M_2$}};
\node (b) at (0,-1) {{\color{olive}$M_1 \sqcup M_2$}};
\node (c) at (4,-1) {{\color{purple}$M_1 \sqcup M_2$}};
\node (d) at (4,1) {{\color{blue}$M_1 \sqcup C_2$}};
\node (e) at (4,0) {{\color{olive}$C_1 \sqcup M_2$}};
\node (f) at (8,0) {{\color{purple}$C_1 \sqcup C_2$}};
\path
	(a) edge[->] (d)
	(a) edge[->] (e)
	(d) edge[->] (f)
	(e) edge[->] (f)
	(b) edge[->] (c)
	(b) edge[->] (e)
	(c) edge[->] (f);
\end{tikzpicture} \right).
\end{equation}
The red terms $M_1 \sqcup M_2 \to C_1 \sqcup C_2$ form a termwise-split subcomplex isomorphic to $C_{\bullet}(\cN)$. 
The green terms $M_1 \sqcup M_2 \to C_1 \sqcup M_2$ form a termwise-split
subquotient complex isomorphic to a shift of $m(C_{\bullet}(\cN_1) \boxtimes M_2)$. The blue terms $M_1 \sqcup M_2 \to M_1 \sqcup C_2$ form a termwise-split quotient complex isomorphic to a shift of $m(M_1 \boxtimes C_{\bullet}(\cN_2))$. 
This demonstrates that the complex $B_\bullet$ may be constructed as a convolution with three parts as just described.

Now we conclude.  The complex $m(D_\bullet)$ clearly lives in $\cI(E_1,E_2)$.  So does 
the contractible complex $\Nul(M_1 \sqcup M_2,d_1+d_2+1)$. By Lemma \ref{lem:convoinideal}, $B_\bullet = \Cone(\beta)$ is 
also in $\cI(E_1,E_2)$. Two of the three parts in a convolution describing $B_\bullet$ are $m(C_{\bullet}(\cN_1) \boxtimes M_2)$ and $m(M_1 \boxtimes C_{\bullet}(\cN_2))$, both of which live in $\cI(E_1,E_2)$. Applying Lemma \ref{lem:convoinideal} again, 
the remaining part of this convolution also lies in $\cI(E_1,E_2)$. This remaining part is $C_{\bullet}^\Omega(\cN)$, 
which completes the proof.
\end{proof}

\begin{proof}[Proof of Theorem \ref{hopf}]
We have already established Equation \eqref{mpreservesI}.
Since $m$ is an equivalence from $\cMatid^+(E_1)\boxtimes\cMatid^+(E_2)$ to $\cMatid^+(E)_\sqcup$ and $\Delta = m^{-1} \circ \Delta_\sqcup$, Equation \eqref{deltapreservesI} is equivalent to the statement
that we have $\Delta_\sqcup(\cI(E))\subset \cI(E_1,E_2)$.
This follows from Lemmas \ref{lem:summarizeeverything} and \ref{compare1}.
\end{proof}

\subsection{Convolution of valuative functors} \label{ssec:convolve}
Suppose that $\cA$ is a monoidal additive category, and let $\otimes$ denote the tensor product in $\cA$.  The example to have in mind is the category of finite dimensional graded vector spaces.
Let $\Phi \co \cMatid^+(E_1)\to\cA$ and $\Psi \co \cMatid^+(E_2)\to\cA$ be additive functors. We define the additive functor $$\Phi \boxtimes \Psi \co \cMatid^+(E_1) \boxtimes \cMatid^+(E_2)\to \cA$$ 
by putting $$\Phi \boxtimes \Psi(M_1\boxtimes M_2) =  \Phi(M_1) \otimes \Phi(M_2) \and  \Phi \boxtimes \Psi(f \boxtimes g) = \Phi(f) \otimes \Psi(g).$$
We then define the {\bf convolution} $$\Phi * \Psi := (\Phi \boxtimes \Psi)\circ\Delta \co \cMatid^+(E) \to \cA.$$
In particular, we have $\Phi * \Psi(M) = \Phi(M^{E_1}) \otimes \Psi(M_{E_1})$ for any matroid $M$ on $E$.

\begin{remark}
If $\Phi$ and $\Psi$ categorify the homomorphisms $f$ and $g$, then the convolution $\Phi * \Psi$ categorifies the convolution $f*g$ as defined in \cite[Definition 6.2]{ArdilaSanchez}.
There is no relationship between this use of the word convolution and the notion of convolution of complexes discussed in Section \ref{sec:cones}.
\end{remark}

Theorem \ref{hopf} has the following corollary.

\begin{corollary} \label{convolution}
If $\Phi$ and $\Psi$ are valuative, then so is $\Phi * \Psi$. \end{corollary}

\begin{proof} Theorem \ref{hopf} tells us that $\Delta$ takes $\cI(E)$ to $\big\langle \cI(E_1)\boxtimes\cMatid^+(E_2),
\cMatid^+(E_1)\boxtimes\cI(E_2) \big\rangle$. Since $\Phi$ and $\Psi$ are both valuative, $\Phi$ kills $\cI(E_1)$ and $\Psi$ kills $\cI(E_2)$. Therefore $\Phi * \Psi$ kills $\cI(E)$.
\end{proof}

\section{Examples of valuative categorical invariants}\label{sec:examples}
In this section, we use Corollary \ref{convolution} to derive new examples of valuative categorical invariants of matroids.

\subsection{Whitney functors}\label{sec:Whitneys}
We begin with a simple lemma.
Let $\Phi:\cMatid(E)\to\cA$ be a valuative functor, and let $k$ be a natural number.  Define a new functor $[\Phi]_k:\cMatid(E)\to\cA$
by putting $[\Phi]_k(M) = \Phi(M)$ if $\rk M = k$ and 0 otherwise.  Since all morphisms in $\cMatid(E)$ relate matroids of the same rank,
$\Phi$ is naturally isomorphic to the direct sum over all $k$ of $[\Phi]_k$.  This immediately implies the following result.

\begin{lemma}\label{summand}
If $\Phi$ is valuative, then so is $[\Phi]_k$.
\end{lemma}

Fix a natural number $r$ and an increasing $r$-tuple of natural numbers $\bk = (k_1,\ldots,k_r)$.
For any matroid $M$, let 
$$\cL_{\bk}(M) := \{(F_1,\ldots,F_r)\mid \text{$F_i$ is a flat of rank $k_i$ and $F_1\subset\cdots\subset F_r$}\}.$$
We define the {\bf Whitney functor} $$\Phi_{\bk}:\cMat\to\Vec_\Q$$ on objects by taking $\Phi_{\bk}(M)$ to be a vector space with basis $\cL_{\bk}(M)$.
If $\varphi:(E,M)\to (E',M')$ is a morphism and $(F_1,\ldots,F_r)\in \cL_{\bk}(M)$, then we define
$$\Phi_{\bk}(\varphi)(F_1,\ldots,F_r) = \begin{cases} \left(\overline{\varphi(F_1)},\ldots,\overline{\varphi(F_r)}\right) & \text{if $\rk_{M'}(\varphi(F_i)) = k_i$ for all $i$} \\ 0 & \text{otherwise.}\end{cases}$$
The main result of this section is that the functor $\Phi_{\bk}$ is valuative.

\begin{proposition}\label{Whitney}
For any $r$ and $\mathbf{k}$, 
the functor $\Phi_{\mathbf{k}}$ is valuative.
\end{proposition}

To prove Proposition \ref{Whitney}, we first consider the related functors introduced in Example \ref{ex:specificflagofflat}.

\begin{lemma}\label{g val}
The functor $\Psi_{E,\mathbf{k},\bS}$ from Example \ref{ex:specificflagofflat} is valuative.
\end{lemma}

\begin{proof}
We proceed by induction on $r$.  The base case $r=0$ is Proposition \ref{triv val}.
For the inductive step, let $r\geq 1$ be given and
assume that the lemma holds for $r-1$.  Fix a set $E$, an increasing $r$-tuple $\bk = (k_1,\ldots,k_r)$ of natural numbers, and an increasing $r$-tuple $\bS = (S_1,\ldots,S_r)$ of subsets of $E$.
By our inductive hypothesis, the functor
$$\Psi_{S_r,(k_1,\ldots,k_{r-1}),(S_1,\ldots,S_{r-1})}:\cMatid(S_r)\to \Vec_\Q$$
is valuative.  By Lemma \ref{summand}, so is the functor
$$\left[\Psi_{S_r,(k_1,\ldots,k_{r-1}),(S_1,\ldots,S_{r-1})}\right]_{k_r}:\cMatid(S_r)\to \Vec_\Q.$$
By Theorem \ref{thm:OS}, the functor 
$\OS^0:\cMatid(E\setminus S_r)\to\Vec_\Q$ is valuative.
Note that the degree zero part of the Orlik--Solomon algebra of a matroid is equal to $\Q$ if that matroid is loopless and to 0 otherwise, thus for any matroid $M$ on $E$, $\OS^0(M_{S_r})$ is equal to $\Q$ if $S_r$ is a flat and 0 otherwise.
It follows that $$\Psi_{E,\bk,\bS} = \left[\Psi_{S_r,(k_1,\ldots,k_{r-1}),(S_1,\ldots,S_{r-1})}\right]_{k_r} * \OS^0,$$
and therefore $\Psi_{E,\bk,\bS}$ is valuative by Corollary \ref{convolution}.\end{proof}

\begin{proof}[Proof of Proposition \ref{Whitney}.]
We need to show that, for any decomposition $\cN$ of a matroid $M$ on the ground
set $E$, and any orientation $\Omega$ of $\cN$, the complex $\Phi_{\bk}(C_\bullet^\Omega(\cN))$ is exact.

For any matroid $N$, define a filtration on $\Phi_{\bk}(N)$ by taking the $i^\text{th}$ filtered piece to be spanned
by those tuples $(F_1,\ldots,F_r)\in \cL_{\bk}(N)$ such that $|F_1|+\cdots+|F_r| \geq i$.  The linear map of vector
spaces induced by a weak map $\iota_{N,N'}:N\to N'$ takes the $i^\text{th}$ filtered piece of $\Phi_{\bk}(N)$ to 
the $i^\text{th}$ filtered piece of $\Phi_{\bk}(N')$, hence we obtain a filtered complex $\Phi_{\bk}(C_\bullet^\Omega(\cN))$.
The associated graded of this complex is isomorphic to the complex
$$\Psi_{E,\bk}(C_\bullet^\Omega(\cN)) = \bigoplus_{\bS} \Psi_{E,\bk,\bS}(C_\bullet^\Omega(\cN)),$$
where the sum is over all increasing $r$-tuples $\bS = (S_1,\ldots,S_r)$ of subsets of $E$.
By Lemma \ref{g val}, this associated graded complex is exact.  
By considering the spectral sequence associated with the filtered complex (as in the proof of Theorem \ref{thm:OS}),
we may conclude that the filtered complex $\Phi_{\bk}(C_\bullet^\Omega(\cN))$ is exact, as well.
\end{proof}

\subsection{Chow functors}\label{sec:chow functors}
Let $M$ be a matroid on the ground set $E$.  The {\bf augmented Chow ring} $\CH(M)$ is defined as the quotient of the polynomial ring $$\Q[x_F\mid \text{$F$ a flat}] \otimes \Q[y_e\mid e\in E]$$
by the ideal $$\left\langle \sum_F x_F\right\rangle + \left\langle y_e - \sum_{e\notin F} x_F\;\Big{|}\; e\in E\right\rangle + \left\langle y_e x_F\;\Big{|}\;  e\notin F\right\rangle + \left\langle x_F x_G\;\Big{|}\;  \text{$F,G$ incomparable}\right\rangle.$$
If $M$ has no loops, the {\bf Chow ring} $\uCH(M)$ is defined as the quotient of $\CH(M)$ by the ideal generated by $\{y_e\mid e\in E\}$.  
If $M$ has loops, then $\uCH(M)$ is by definition 0.

We would like to promote these rings to categorical invariants.  Our first tool will be the following theorem of Feichtner and Yusvinsky \cite[Corollary 1]{FY}.

\begin{theorem}\label{FY basis}
If $M$ has no loops, then $\uCH(M)$ has a basis consisting of monomials of the form $x_{F_1}^{m_1}\cdots x_{F_r}^{m_r}$, where
$r\in\N$, $\varnothing = F_0 \subsetneq F_1\subsetneq \cdots\subsetneq F_r$, and $0<m_i<\rk F_i - \rk F_{i-1}$ for all $i\in\{1,\ldots,r\}$.
\end{theorem}

For any positive integer $k$, consider the graded vector space
\revise{
\[
Q_k := \bigoplus_{i=1}^{k-1}\mathbb{Q}(-i),
\]
}
of total dimension $k-1$, with a piece in every positive degree less than $k$.
Theorem \ref{FY basis} has the following corollary.

\begin{corollary}\label{uCH functor}
If $M$ has no loops, then
there is a canonical isomorphism of graded vector spaces $$\uCH(M)\;\cong\; \bigoplus_{r\geq 0}\;\bigoplus_{\substack{\mathbf{k} \,= (k_1,\ldots,k_r)\\ 0<k_1<\cdots<k_r}} \Phi_{\mathbf{k}}(M) \otimes Q_{k_1}\otimes Q_{k_2-k_1}\otimes\cdots\otimes Q_{k_r-k_{r-1}}.$$
\end{corollary}

Based on Corollary \ref{uCH functor}, we define a functor $\uCH$ from $\cMat$ to the category of finite dimensional graded vector spaces over $\Q$ by setting
$$\uCH := \bigoplus_{r\geq 0}\;\bigoplus_{\substack{\bk \,= (k_1,\ldots,k_r)\\ 0<k_1<\cdots<k_r}} \Phi_{\bk} \otimes Q_{k_1}\otimes Q_{k_2-k_1}\otimes\cdots\otimes Q_{k_r-k_{r-1}}$$
on the full subcategory spanned by matroids without loops, and $\uCH = 0$ on the full subcategory spanned by matroids with loops.  We observe that there are no morphisms in $\cMat$ from a matroid with loops to a matroid without loops,
so these conditions uniquely characterize a well-defined functor $\uCH$.  

\begin{corollary}\label{uCH val}
The categorical invariant $\uCH$ is valuative.
\end{corollary}

\begin{proof}
Corollary \ref{uCH functor} tells us that $\uCH$ is a direct sum of shifts of functors of the form $\Phi_{\mathbf{k}}$, which is valuative by Proposition \ref{Whitney}.
\end{proof}

Our next task is to construct an analogous basis for the augmented Chow ring $\CH(M)$.
Given a flat $F$ of $M$, choose any maximal independent set $I\subset F$, and let $y_F := \prod_{e\in I} y_e\in\CH(M)$.
The element $y_F$ does not depend on the choice of $I$ \cite[Lemma 2.11(2)]{BHMPW1}.

\begin{proposition}\label{CH basis}
For any matroid $M$, the augmented Chow ring $\CH(M)$ has a basis consisting of monomials
of the form $y_{F_0}x_{F_1}^{m_1}\cdots x_{F_r}^{m_r}$, where
$r\in\N$, $F_0 \subsetneq F_1\subsetneq \cdots\subsetneq F_r$, and $0<m_i<\rk F_i - \rk F_{i-1}$ for all $i\in\{1,\ldots,r\}$.
\end{proposition}

\begin{proof}
Let $\mathfrak{m}\subset\CH(M)$ be the ideal generated by $\{y_e\mid e\in E\}$.
The argument in the proof of \cite[Proposition 1.8]{BHMPW2} shows that we have an isomorphism
\begin{equation}\label{gr CH}\gr\CH(M) :=  \bigoplus_{p\geq 0}\frac{\mathfrak{m}^p\CH(M)}{\mathfrak{m}^{p+1}\CH(M)}
\cong \bigoplus_F \uCH(M_F)(-\rk F),\end{equation}
where $\uCH(M_F)(-\rk F)$ embeds into $\gr\CH(M)$ by sending a polynomial $\eta$ in 
$\{x_G \mid F\subset G \subset E\}$ to the polynomial $y_F\eta$.  We may therefore use the basis for each $\uCH(M_F)$
from Theorem \ref{FY basis} to construct a basis for $\gr \CH(M)$, and this lifts to a basis for $\CH(M)$.
\end{proof}



\begin{remark}\label{decat FMSV}
The decategorified version of Equation \eqref{gr CH}, which says that 
$$H_M(t) = \sum_F t^{\rk F} \uH_{M_F}(t),$$
appears in \cite[Theorems 1.3, 1.4, and 1.5]{FMSV}.
\end{remark}

Proposition \ref{CH basis} has the following corollary.

\begin{corollary}\label{CH functor}
There is a canonical isomorphism of graded vector spaces $$\CH(M)\;\cong\; \bigoplus_{r\geq 0}\;\bigoplus_{\substack{\mathbf{k} \,= (k_0,k_1,\ldots,k_r)\\ k_0<k_1<\cdots<k_r}} \Phi_{\mathbf{k}}(M) \otimes Q_{k_1-k_0}\otimes Q_{k_2-k_1}\otimes\cdots\otimes Q_{k_r-k_{r-1}}(-k_0).$$
\end{corollary}

Motivated by Corollary \ref{CH functor}, we define a functor $\CH$ from $\cMat$ to the category of finite dimensional graded vector spaces over $\Q$ 
by putting $$\CH\; :=\; \bigoplus_{r\geq 0}\;\bigoplus_{\substack{\bk \,= (k_0,k_1,\ldots,k_r)\\ k_0<k_1<\cdots<k_r}} \Phi_{\bk}(M) \otimes Q_{k_1-k_0}\otimes Q_{k_2-k_1}\otimes\cdots\otimes Q_{k_r-k_{r-1}}(-k_0).$$
Corollary \ref{uCH val} has the following augmented analogue.

\begin{corollary}\label{CH val}
The categorical invariant $\CH$ is valuative.
\end{corollary}

\begin{remark}
Corollaries \ref{uCH val} and \ref{CH val}  categorify \cite[Theorem 8.14]{FSVal} and \cite[Theorem 10]{fmsv-fpsac}, 
which say that the Chow polynomial $\uH_M(t)$ and the augmented Chow polynomial $\H_M(t)$ are valuative invariants.
\end{remark}

\vspace{-\baselineskip}
\revise{\begin{remark}
If $M\to M'$ is a weak map of matroids, then the associated map $\OS(M)\to\OS(M')$ is an algebra homomorphism.
In contrast, unless $M$ and $M'$ are isomorphic, the associated maps $\uCH(M)\to\uCH(M')$ and $\CH(M)\to\CH(M')$ are only maps of graded vector spaces, not of algebras.
This does not have any bearing on our results, but we mention it to avoid any possible confusion.
\end{remark}}

\subsection{Kazhdan--Lusztig functors}
Given positive integers $j$ and $r$ along with a subset $R\subset [r]$, let 
$$s_j(R):= \min\!\big(\Z_{\geq j}\setminus R\big)
\in \{1,\ldots,r+1\}.$$
The following theorem is proved in \cite[Theorems 6.1]{PXY}.

\begin{theorem}\label{thm KL}
Let $M$ be a loopless matroid of rank $k$ on the ground set $E$. 
The Kazhdan--Lusztig polynomial 
$P_{M}(t)$ is equal to
$$1 + \sum_{i\geq 0}\sum_{r=1}^i t^i \sum_{R\subset [r]} (-1)^{|R|} \sum_{\substack{a_0<a_1<\cdots < a_r<a_{r+1}\\ a_0 = 0 \\ a_r = i\\ a_{r+1} = k-i}} \dim \Phi_{\mathbf{k}}(M),$$
where $\mathbf{k} = (k_1, \ldots , k_r)$ and $k_j = k - a_{s_{r+1-j}(R)} - a_{r-j}$.
\end{theorem}

\begin{remark}\label{other stuff}
The polynomial $P_M(t)$ is equal to zero for any matroid with loops, thus
Theorem \ref{thm KL} provides a full description of the Kazhdan--Lusztig polynomials of all matroids.
\end{remark}

\excise{
The analogue of Theorem \ref{thm KL} for the $Z$-polynomial does not appear in the literature, so we include it here.

\begin{proposition}\label{thm Z}
Let $M$ be an arbitrary matroid of rank $k$ on the ground set $E$.
The $Z$-polynomial 
$Z_{M}(t)$ is equal to
$$1 + \sum_{1\leq r\leq j\leq i}^i t^i \sum_{R\subset [r]} (-1)^{|R|} \sum_{\substack{a_0<a_1<\cdots < a_r<a_{r+1}\\ a_0 = 0 \\ a_r = j\\ a_{r+1} = k-i}} \dim \Phi_{\textbf{k}}(M),$$
where $\textbf{k} = (i-j,k_1, \ldots , k_r)$ and $k_j = k - a_{s_{r+1-j}(R)} - a_{r-j} + i - j$.
\end{proposition}

\nicktodo{This is wrong, because it doesn't account for the fact that $c_{M_F}(0) = 1$ rather than 0.}

\begin{proof}
By definition, we have $Z_M(t) = \sum_F t^{\rk F}P_{M_F}(t)$.  This means that the constant term is always 1, and when $i>0$, the coefficient of $t^i$
is equal to the sum over flats $F$ of positive rank of the coefficient of $i-\rk F$ in $P_{M_F}(t)$.
By Theorem \ref{thm KL}, this is equal to
$$\sum_{j=1}^i \sum_{\rk F = j} \sum

 = 1 + \sum_{j=1}^k t^j \sum_{\rk F = j} P_{M_F}(t).$$

\end{proof}
}

Based on Theorem \ref{thm KL}, we define a functor $\uKL$ from $\cMat$ to the category of finite dimensional bigraded vector spaces over $\Q$ as follows.
On the full subcategory of $\cMat$ consisting of loopless matroids of rank $k$, we define the functor
$$\uKL := \tau \oplus \bigoplus_{1\leq r\leq i}\;\bigoplus_{R\subset [r]} \;\bigoplus_{\substack{a_0<a_1<\cdots < a_r<a_{r+1}\\ a_0 = 0 \\ a_r = i\\ a_{r+1} = k-i}} \Phi_{\textbf{k}}(-i,-|R|),$$
where $\textbf{k} = (k_1, \ldots , k_r)$ and $k_j = k - a_{s_{r+1-j}(R)} - a_{r-j}$.
Here the notation means that the summand $\Phi_{\bk}$ appears in bidegree $(i,|R|)$.
We define $\uKL = 0$ on the full subcategory of $\cMat$ consisting of matroids with loops.
By definition, the functor $\uKL$ categorifies the matroid invariant
$$\tilde{P}_M(t,u) := \sum_{i,j} \dim \uKL^{i,j}(M)\, t^iu^j,$$
and Theorem \ref{thm KL} and Remark \ref{other stuff} imply that $\tilde{P}_M(t,-1)$ is equal to the Kazhdan--Lusztig polynomial $P_M(t)$.
The following result follows immediately from Propositions \ref{triv val} and \ref{Whitney}. 

\begin{corollary}\label{KL val}
The categorical invariant $\uKL$ is valuative.
\end{corollary}

The $Z$-polynomial relates to the Kazhdan--Lusztig polynomial in the same way that the augmented Chow polynomial
relates to the Chow polynomial.  That is, we have
\begin{equation}\label{Z-recursion}Z_M(t) := \sum_{F} t^{\rk F} P_{M_F}(t) = \sum_{S\subset E} t^{\rk S} P_{M_S}(t),\end{equation}
where the second equality comes from the fact that, whenever $S$ is not a flat, $M_S$ has a loop, and therefore $P_{M_S}(t)=0$.
We therefore define the functor
$$\KL := \bigoplus_{k\geq 0}\bigoplus_{S\subset E} \Big([\tau]_k * \uKL\Big)(-k,0),$$
where the summand indexed by $S$ is understood to be the convolution of the functor $[\tau]_k$ on $\cMatid(S)$ with the functor $\uKL$ on $\cMatid(E\setminus S)$.
Each summand of this functor is defined only on the category $\cMatid(E)$, but the direct sum extends to the entire category $\cMat$.
By definition, the functor $\KL$ categorifies the polynomial
$$\tilde{Z}_M(t,u) := \sum_{i,j} \dim \KL^{i,j}(M)\, t^iu^j,$$
and Equation \eqref{Z-recursion} implies that $\tilde{Z}_M(t,-1)=Z_M(t)$.
We obtain the following corollary from Proposition \ref{triv val},
Corollary \ref{convolution}, Lemma \ref{summand}, and Corollary \ref{KL val}.

\begin{corollary}\label{Z val}
The categorical invariant $\KL$ is valuative.
\end{corollary}

\begin{remark}
Corollaries \ref{KL val} and \ref{Z val} categorify \cite[Theorem 8.9]{ArdilaSanchez} and \cite[Theorem 9.3]{FSVal}, which say that the Kazhdan--Lusztig polynomial $P_M(t)$ and the $Z$-polynomial $Z_{M}(t)$ are valuative invariants.
\end{remark}

\begin{remark}\label{rmk:better categorification Z}
In light of our definition of the functor $\KL$, one might ask if we should have defined the functor $\CH$ in an analogous way.
More precisely, consider the functor $$\Theta := \bigoplus_{k\geq 0} \bigoplus_{S\subset E}\Big([\tau]_k * \uCH\Big)(-k).$$
As in the definition of $\KL$, each summand is defined only on the category $\cMatid(E)$, but the direct sum extends to the entire category $\cMat$.
By Remark \ref{decat FMSV} and the fact that $\uCH$ vanishes on matroids with loops, $\Theta$ categorifies the augmented Chow polynomial $H_M(t)$.
Moreover, a slight modification of Proposition \ref{CH basis} implies that $\Theta(M)$ is canonically isomorphic to $\CH(M)$ for any matroid $M$.
In addition, Proposition \ref{triv val},
Corollary \ref{convolution}, Lemma \ref{summand}, and Corollary \ref{uCH val} imply that $\Theta$ is valuative.

However, $\Theta$ is not naturally isomorphic to the functor $\CH$ because it behaves differently on weak maps that are not isomorphisms.
For example, the degree zero part of $\Theta$ is isomorphic to $\Psi_{(0)}$, whereas the degree zero part of $\CH$ is isomorphic to $\Phi_{(0)}$.
That is, if $\varphi:M\to M'$ is a morphism in $\cMat$ and $M'$ has strictly more loops than $M$, then $\CH(\varphi)$ will be an isomorphism but $\Theta(\varphi)$ will be zero.
In particular, Corollary \ref{CH mon} (which will be stated and proved in the next section) would fail for the functor $\Theta$.  This is why we regard the functor $\CH$ as a ``better'' categorification
of the augmented Chow polynomial than the functor $\Theta$.

Similarly, the functor $\KL$ is not the only valuative categorical invariant that categorifies the $Z$-polynomial.  However, since there is no analogue of Corollary \ref{CH mon} for the $Z$-polynomial,
we know of no reason to regard one categorification as better than another.  See Remark \ref{if only} for more on this topic.
\end{remark}

\subsection{Monotonicity}\label{sec:monotonicity}
If $\varphi:M\to M'$ is a morphism in $\cMat$, it is clear from the definition of the functor $\OS$ that the ring
homomorphism $\OS(\varphi):\OS(M)\to\OS(M')$ is surjective, and therefore that the polynomial
$\pi_M(t)-\pi_{M'}(t)$ has non-negative coefficients.  
We express this statement by saying that the Poincar\'e polynomial is {\bf monotonic} with respect to weak maps. 
\excise{\revise{If a valuative invariant takes non-negative values on the whole class of matroids, then it is easy to see that for fixed rank and size it takes its maximal value on uniform matroids. One can show that $f(U_{k,n}) - f(M)$ is non-negative by building a regular decomposition of the uniform matroid that contains $M$ as one of the maximal cells. It is therefore always appealing to extend this result to general weak maps.}}
The aim of this section is to prove a similar result for the Chow and augmented Chow polynomials, and to discuss
the possibility of extending it to the Kazhdan--Lusztig and $Z$-polynomials.

Fix a natural number $r$ and an increasing $r$-tuple of natural numbers $\bk = (k_1,\ldots,k_r)$.  The following result is equivalent to a statement conjectured in \cite[Conjecture 3.32]{FMSV}.

\begin{proposition}\label{monotonic}
For any morphism $\varphi:M\to M'$ in $\cMat$, the linear map $$\Phi_{\mathbf{k}}(\varphi):\Phi_{\mathbf{k}}(M)\to\Phi_{\mathbf{k}}(M')$$
is surjective.\footnote{We thank George Nasr for his help with this result.}
\end{proposition}

\begin{proof}
We may use $\varphi$ to identify the ground sets of $M$ and $M'$, and thus reduce to the case of a morphism $\iota_{M,M'}:M\to M'$
in $\cMatid(E)$.
Suppose we are given $(F'_1,\ldots,F'_r)\in\cL_{\bk}(M')$.
We need to find $(F_1,\ldots,F_r)\in\cL_{\bk}(M)$ such that $\overline{F_i} = F_i'$ for all $1\leq i\leq r$.

We proceed by induction on $r$.  When $r=0$, there is nothing to prove.  The $r=1$ case is proved in
\cite[Proposition 5.12]{Luc75}.  For the inductive step, we may assume that we have flats $F_2\subset\cdots\subset F_r$
with $\rk_M(F_i) = k_i$ and $\overline{F_i} = F_i'$ for all $2\leq i\leq r$.

Since $\overline{F_2} = F'_2$, 
we have $$\rk_{M'}(F_2) = \rk_{M'}(F'_2) = k_2 = \rk_M(F_2),$$ thus we have a morphism
$M^{F_2}\to (M')^{F_2}$ in $\cMatid(F_2)$, and $F'_1\cap F_2$ is a flat of rank $k_1$ in $(M')^{F_2}$.
We may therefore apply \cite[Proposition 5.12]{Luc75} again to find a flat $F_1$ of $M^{F_2}$ of rank $k_1$
whose closure in $(M')^{F_2}$ is equal to $F'_1\cap F_2$.

We claim that the closure of $F_1$ in $M'$ is equal to $F'_1$.  Indeed,
the rank of $F_1$ in $M'$ is the same as its rank in $(M')^{F_2}$, which is
equal to $k_1$, thus its closure is a flat of rank $k_1$.
Since $F_1\subset F_1'\cap F_2\subset F_1'$, the closure is contained in $F_1'$, 
and must be equal by comparison of the ranks.
\end{proof}

As a consequence, we obtain a strengthening of the numerical monotonicity result in \cite[Theorem 1.11]{FMSV}.

\begin{corollary}\label{CH mon}
For any morphism $\varphi:M\to M'$ in $\cMat$, the linear maps $$\uCH(\varphi):\uCH(M)\to\uCH(M')
\and \CH(\varphi):\CH(M)\to\CH(M')$$
are both surjective.  In particular, the polynomials
$$\uH_M(t) - \uH_{M'}(t) \and \H_M(t) - \H_{M'}(t)$$ have non-negative coefficients.
\end{corollary}

\begin{proof}
The maps $\uCH(\varphi)$ and $\CH(\varphi)$ are surjective by Proposition \ref{monotonic} and the definitions of $\uCH$ 
and $\CH$ as direct sums of shifts
of functors of the form $\Phi_{\bk}$. 
\end{proof}

\vspace{-\baselineskip}
\revise{
\begin{remark} \label{rmk:monotonicbecauseyoudiditright} Monotonicity is the numerical shadow of a natural surjection at the level of the (underlying vector spaces of the) Chow rings.
Note that the proof of Proposition \ref{monotonic} does not use categorical valuativity. However, this result highlights the importance of choosing
the right categorification of a given invariant. The two functors $\Phi_{\bk}$ and $\Psi_{E,\bk}$ (Example \ref{ex:allflagofflat}) categorify the
same homomorphism, but Proposition \ref{monotonic} would fail for the functor $\Psi_{E,\bk}$. 
We regard the fact that the functor $\Phi_{\bk}$ is valuative {\em and} satisfies Proposition \ref{monotonic} is 
as strong evidence that it is a categorification worthy of study.
\end{remark}
}

The Kazhdan--Lusztig polynomial and $Z$-polynomial are conjectured to have the same monotonicity property as the Chow polynomial and augmented Chow polynomial.
The following conjecture was first formulated by Nasr, generalizing an unpublished conjecture of Gedeon in the case where $M$
is uniform.

\begin{conjecture}\label{P and Z mon}
For any morphism $\varphi:M\to M'$ in $\cMat$, the polynomials
$$P_M(t) - P_{M'}(t) \and Z_M(t) - Z_{M'}(t)$$ have non-negative coefficients.
\end{conjecture}

\begin{remark}\label{if only}
There exist graded vector subspaces $\IH(M)\subset \CH(M)$ and $\IH(M)_\varnothing\subset \uCH(M)$ with Poincar\'e polynomials $Z_M(t)$ and $P_M(t)$, respectively \cite[Theorem 1.9]{BHMPW2}.
It would be interesting to define functors to the abelian category of finite dimensional graded vector spaces over $\Q$ taking a matroid $M$ to $\IH(M)$ and $\IH(M)_\varnothing$, and then to strengthen Conjecture
\ref{P and Z mon} along the same lines as Conjecture \ref{CH mon}
by conjecturing that the linear maps induced by a morphism in $\cMat$ are surjective.
Unfortunately, we do not currently know how to define these functors on morphisms that are not isomorphisms
(other than by defining all such maps to be zero, which would be neither valuative nor surjective).
This is why we work instead with the functors $\uKL$ and $\KL$ to the category of bigraded vector spaces. 
\end{remark}

\subsection{\revise{The \boldmath{$\mathcal{G}$}-invariant}}\label{sec:G}
Let $\cA$ be the free additive monoidal category on two generators $X$ and $Y$.  Let $A \cong \Z\langle x, y\rangle$ be its Grothendieck ring,
which is a noncommutative ring freely generated by the elements $x = [X]$ and $y = [Y]$.
For any single element set $\{e\}$, consider the functor $\Xi_e:\cMat(\{e\})\to\cA$ that takes the loop to $X$ and the coloop to $Y$.
There are no morphisms in $\cMat(\{e\})$ other than the two identity morphisms, so this fully determines the functor, which is trivially valuative.
For any ordered set $E = \{e_1,\ldots,e_n\}$, let 
$$\Xi_E := \bigoplus_{\pi\in S_n} \Xi_{e_{\pi(1)}} * \cdots * \Xi_{e_{\pi(n)}}.$$  Then $\Xi_E:\cMat(E)\to\cA$
is a categorical valuative invariant that
categorifies Derksen's $\mathcal{G}$-invariant \cite{Derksen}, 
and the proof of valuativity by Corollary \ref{convolution}
categorifies the proof of valuativity of the $\mathcal{G}$-invariant given in \cite[Theorem 8.15]{ArdilaSanchez}.

Derksen and Fink \cite[Theorem 1.4]{derksen_fink} prove that the $\mathcal{G}$-invariant is {\bf rationally universal}, in the sense that every $\Q$-valued valuative invariant of matroids on $E$
can be obtained by composing the $\mathcal{G}$-invariant with a homomorphism $A \to \Q$. However, the $\mathcal{G}$-invariant is not universal for valuative invariants valued in other
abelian groups, not even for the integers. For example, the $\mathcal{G}$-invariant of the the uniform matroid of rank 2 on a ground set of size $4$ (the octahedron in Example
\ref{octahedron 1}) is $24xxyy$. There is no homomorphism $A \to \Z$ sending $24xxyy$ to $1$, so the $\mathcal{G}$-invariant does not specialize to the trivial invariant. This
observation has categorical implications as well: there is no functor from $\cA$ to vector spaces sending $(X \otimes X \otimes Y \otimes Y)^{\oplus 24}$ to a one-dimensional vector
space, and thus $\Xi_E$ does not ``specialize'' to the trivial categorical invariant $\tau$.

\begin{remark} There are potentially weaker notions of universality that one might consider for categorical valuative invariants: rather than insisting that a given invariant is isomorphic (as functors) to the composition of the universal invariant with a specialization functor, one could merely ask for a non-invertible natural transformation between them. It would be interesting to search for universal categorical valuative invariants, but this problem is beyond the scope of this paper. \end{remark}

\subsection{\revise{The Grothendieck group of the valuative category}}\label{sec:groth}
We observed in Remark \ref{rmk:groth} that the Grothendieck group of the valuative category $\cV(E)$ is {\em a priori}
isomorphic to a quotient of the valuative group $\Val(E)$.  The subtlety here is that localizing subcategories are by definition
closed under passing to direct summands, which means that the Grothendieck classes of objects of $\cI(E)$
could in theory include elements of the group $M(E) \cong K(\Ch_b(\cMatid^+(E)))$ that do not lie in the subgroup $I(E)$.  
In this section, we prove that the Grothendieck
group of $\cV(E)$ is in fact isomorphic to $\Val(E)$; we thank Matt Larson for outlining this argument.

Consider 
an increasing $r$-tuple $\bk = (k_1,\ldots,k_r)$ of natural numbers and an increasing $r$-tuple $\bS = (S_1,\ldots,S_r)$ of subsets of $E$.
The functor $\Psi_{E,\bk,\bS}:\cMatid(E)\to\Vec_\Q$ was defined in Example \ref{ex:specificflagofflat} and proved to be valuative in Lemma \ref{g val}.  In the proof of Proposition \ref{Whitney}, we fixed $\bk$ and considered the direct sum of these functors over all possible $\bS$.
In this section, we do the opposite: we fix $\bS$, and define
$\Psi_{E,\bS}:\cMatid(E)\to\Vec_\Q$ by putting $$\Psi_{E,\bS}(M) := \bigoplus_{\bk}\Psi_{E,\bk,\bS}(M).$$
Concretely, $\Psi_{E,\bS}(M)$ is equal to $\Q$ if $S_j$ is a flat (of any rank) for all $j$, and zero otherwise.

Now fix a subset $I\subset S_1$ of cardinality $k$.  
Let $\bS\setminus I := (S_1\setminus I,\ldots, S_r\setminus I)$, which we regard as an increasing chain of subsets of $E\setminus I$.
We will be interested in the functor $$\Psi_{E\setminus I,\bS\setminus I}:\cMatid(E\setminus I)\to\Vec_\Q.$$
We will also consider the functor $[\tau]_k:\cMatid(I)\to\Vec_\Q$, which takes the Boolean matroid on $I$
to $\Q$, and all other matroids on $I$ to zero.  Convolving them, we obtain a new functor
$$\Psi_{I\leq\bS}:= [\tau_k] * \Psi_{E\setminus I,\bS\setminus I}:\cMatid(E)\to\Vec_\Q.$$ 
On the level of objects, $\Psi_{I\leq\bS}(M)$ is equal to $\Q$ if $I$ is independent in $M$ and 
$S_j\setminus I$ is a flat of the contraction $M_I$ for all $j$, or equivalently if 
$I$ is independent in $M$ and $S_j$ is a flat of $M$ for all $j$.  Otherwise, $\Psi_{I\leq\bS}(M) = 0$.

\begin{lemma}\label{Matt}
The functor $\Psi_{I\leq\bS}$ is valuative.
\end{lemma}

\begin{proof}
Valuativity of the first convolution factor follows from Proposition \ref{triv val} and Lemma \ref{summand}, while
valuativity of the second convolution factor follows from Lemma \ref{g val}.  The result then follows from Corollary \ref{convolution}.
\end{proof}

The {\bf stellahedral fan} \cite[Definition 3.2]{EHL} $\Sigma_E$ is a complete unimodular fan in $\R_E$ with cones $\sigma_{I\leq \bS}$
indexed by chains $$\varnothing \subseteq I\subseteq S_1 \subsetneq S_2\subsetneq \cdots \subsetneq S_r\subsetneq E.$$
Note that we are allowed to have $r=0$, in which case $I$ is allowed to coincide with $E$.
Explicitly, we have $$\sigma_{I\leq \bS} := \R_{\geq 0}\{v_e\mid e\in I\} + \R_{\geq 0}\{-v_{E\setminus S_j}\mid 1\leq j\leq r\}.$$
For any matroid $M$ on $E$, the {\bf augmented Bergman fan} \cite[Definition 2.4]{BHMPW1}
$\Sigma_M$ is defined to be the subfan of $\Sigma_E$
consisting of those $\sigma_{I\leq \bS}$ for which $I$ is independent and $S_j$ is a flat for all $j$.
Note that the stellahedral fan is equal to the augmented Bergman fan of the Boolean matroid.

Let $\Z^{\Sigma_E}$ be the vector space with basis $\{w_\sigma\mid \sigma\in\Sigma_E\}$.  The following
theorem follows immediately from \cite[Theorems 1.5 and 5.6]{EHL}.

\begin{theorem}\label{EHL-val}
The homomorphism $M(E) \to \Z^{\Sigma_E}$ taking $M$ to $\sum_{\sigma\in\Sigma_M} w_\sigma$ has kernel equal to $I(E)$,
and therefore descends to an embedding of $\Val(E)$ into $\Z^{\Sigma_E}$.
\end{theorem}

We are now ready to prove our proposition.

\begin{proposition}\label{prop:groth}
The Grothendieck group of the valuative category $\cV(E)$ is isomorphic to the valuative group $\Val(E)$.
\end{proposition}

\begin{proof}
We have already noted that there is a canonical surjection $\Val(E)\to K(\cV(E))$ taking the class of a matroid $M$
to the Grothendieck class of the object $M$ of $\cV(E)$.  We will now prove that this is an isomorphism by constructing its inverse.

Let $\Vec_\Q^{\Sigma_E}$ be the category of finite dimensional $\Q$-vector spaces equipped with a direct sum decomposition indexed by cones $\sigma\in\Sigma_E$, with morphisms given by linear maps that are compatible with these decompositions.
Consider the functor
$\Psi:\cMatid(E)\to\Vec_\Q^{\Sigma_E}$ obtained as the direct sum of the functors $\Psi_{I\leq\bS}$.  Lemma \ref{Matt} immediately
implies that $\Psi$ is valuative, thus $\Psi$ descends to the valuative category $\cV(E)$, and therefore the induced homomorphism $$M(E) \cong K(\cMatid(E)) \to K(\Vec_\Q^{\Sigma_E}) \cong \Z^{\Sigma_E}$$
descends to $K(\cV(E))$.
This is precisely the homomorphism from Theorem \ref{EHL-val}, whose image is isomorphic to $\Val(E)$.
The induced homomorphism from $K(\cV(E))$ to $\Val(E)$ takes the Grothendieck class of $M$ to the class of $M$ in $\Val(E)$, and is therefore
inverse to the homomorphism in the first paragraph.
\end{proof}

\section{Characters}\label{sec:characters}
In this section, we explain how to use valuative
categorical matroid invariants to obtain relations among isomorphism classes of graded representations of the symmetry group
of a matroid decomposition.  Everything that follows works equally well for decompositions of polyhedra, but since our examples are all matroid invariants, we will work entirely in that setting.

\subsection{The oriented case}
Let $\cN$ be a decomposition of a matroid $M$ on the ground set $E$, and let $\Omega$ be an orientation of $\cN$.  
Let $\Gamma$ be a finite group acting via permutations of $E$ that preserve
$\cN$ and $\Omega$.  In other words, for every $N\in\cN$ and $\gamma\in\Gamma$, $\gamma(N)\in\cN$,
and the map $P(N)\to P(\gamma(N))$ on base polytopes induced by $\gamma$ is orientation preserving.  This
implies that $\gamma$ acts on the complex $C_\bullet^\Omega(\cN)$.
Fix an abelian group, and let $\cA$ be the category of finite dimensional vector spaces over $\Q$ that are graded in that group.\footnote{In
practice, we will always take our gradings in $\Z$ or $\Z^2$.}
Suppose that $\Phi:\cMat(E)\to\cA$ is a categorical valuative invariant.
Then $\Phi(C_\bullet^\Omega(\cN))$ is exact, and we therefore have an equation \begin{eqnarray}\label{characters}\notag 0 = \chi^\Gamma\!\Big(\Phi(C_\bullet^\Omega(\cN))\Big) &:=&  \sum_k (-1)^k \,\Phi(C_k^\Omega(\cN))\\ &=& 
\sum_k(-1)^k \bigoplus_{N\in\cN_k}\Phi(N)\\ \notag &=& \sum_k (-1)^k \!\bigoplus_{N\in\cN_k/\Gamma} \Ind_{\revise{\Gamma_N}}^\Gamma\Phi(N)\end{eqnarray}
of virtual graded representations of $\Gamma$, where for the last sum one takes a single representative of each $\Gamma$ orbit in $\cN_k$.

Examples from previous sections of valuative categorical invariants
include $\OS$ (Theorem \ref{thm:OS}), $\uCH$ (Corollary \ref{uCH val}),
$\CH$ (Corollary \ref{CH val}), all of which take values in finite dimensional graded vector spaces.
In these cases, Equation \eqref{characters} gives a relationship between Orlik--Solomon algebras, Chow rings,
and augmented Chow rings of all of the matroids in $\cN$, regarded as isomorphism classes of representations of $\Gamma$.
We also have the examples $\uKL$ (Corollary \ref{KL val}) 
and $\KL$ (Corollary \ref{KL val}), both of which take values in bigraded vector spaces.
The virtual graded $\Gamma$-representations
$$\sum_j (-1)^j \uKL^{i,j}(M)\and \sum_j (-1)^j\, \KL^{i,j}(M)$$ are equal to the coefficients of $t^i$ in the $\Gamma$-equivariant Kazhdan--Lusztig polynomial $P_M^\Gamma(t)$ and the $\Gamma$-equivariant $Z$-polynomial $Z_{M}^\Gamma(t)$, respectively \cite[Theorem 6.1]{PXY}.
Thus Equation \eqref{characters} also allows us to relate these equivariant matroid invariants for $M$ to those of
the various $N\in\cN$.

There is, however, a serious limitation to the usefulness of Equation \eqref{characters}.  
When given a decomposition $\cN$ with an action of $\Gamma$, it is usually
impossible to choose an orientation $\Omega$ that is preserved by $\Gamma$.  For instance, consider the decomposition
in Examples \ref{octahedron 1} and \ref{octahedron 2}.  We have $E = \{1,2,3,4\}$, and the group $\fS_2$ acts by swapping 1 with 3 and 2 with 4.
This action preserves $\cN$, swapping $N$ with $N'$ and taking $N''$ to itself.  However, the action of $\fS_2$ on $P(N'')$ is orientation reversing, so there is no way to choose $\Omega$.
If there were, then Equation \eqref{characters} applied to the trivial categorical invariant $\tau$ would tell us that $\tau(M) \oplus \tau(N'')$ (two copies of the trivial representation of $\fS_2$) is isomorphic to $\tau(N) \oplus \tau(N')$ (the regular representation of $\fS_2$), which is of course false.  

What we would like to do is replace $\tau(N'')$ with the sign representation
of $\fS_2$, to reflect the fact that the action on $P(N'')$ is orientation reversing (see Example \ref{finally}).
The rest of this section is devoted to developing the machinery needed to make this precise, in the form of Corollary \ref{the point}, 
and then applying it to certain families of examples.

\subsection{The determinant category} \label{ssec:detcat}
If $M$ is a matroid on the ground set $E$, consider the vector space
$$\revise{V(M)} := \operatorname{Span}\{x-y\mid x,y\in P(M)\}\cap \Q^E.$$
If $P(M')\subset P(M)$, we define 
the 1-dimensional vector space $$L(M,M') := \wedge^{d(M) - d(M')} \left(\revise{V(M)}/\revise{V(M')}\right).$$
Given $M$, $M'$, and $M''$ with $P(M'')\subset P(M')\subset P(M)$, wedge product induces a canonical isomorphism
\begin{equation}\label{wedge}L(M,M')\otimes L(M',M'')\to L(M,M'').\end{equation}
We now define a new category $\cMatid^\wedge(E)$ by taking the objects to be formal direct sums of matroids on $E$,
and taking $$\Hom_{\cMatid^\wedge(E)}(M,M') := L(M,M').$$
Composition is given by Equation \eqref{wedge}, and
morphisms between formal direct sums are matrices whose entries consist of morphisms between the individual matroids.

\begin{remark}
The category $\cMatid^\wedge(E)$ should be regarded as a variant of $\cMatid^+(E)$.
In both categories, the space of homomorphisms from $M$ to $M'$ is a 1-dimensional $\Q$-vector space.  
The difference is that, in $\cMatid^+(E)$, this vector space is canonically identified with $\Q$, 
whereas the same is not true in $\cMatid^\wedge(E)$.
\end{remark}

Given a decomposition $\cN$ of a matroid $M$ on the ground set $E$, we defined a large collection of complexes 
$C_\bullet^\Omega(M)$ in $\cMatid^+(E)$, depending on an $\Omega$ of $\cN$.  
In contrast, we will define a single complex $C_\bullet^\wedge(\cN)$ in $\cMatid^\wedge(E)$
that does not depend on any choices.  The objects will be the same; that is, we put $$C_k^\wedge(\cN) := \bigoplus_{N\in\cN_k} N.$$
For each maximal face $N\in\cN_{d(M)}$, $L(M,N)$ is canonically isomorphic to $\Q$, and we define
$C_{d(M)+1}^\wedge(\cN) \to C_{d(M)}^\wedge(\cN)$ to be the diagonal map.  Given $k\leq d(M)$, $N\in\cN_k$, and $N'\in \cN_{k-1}$
with $P(N')\subset P(N)$, we define the $(N,N')$ component of the differential $C_{k}^\wedge(\cN)\to C_{k-1}^\wedge(\cN)$
to be the class of the outward unit normal vector to $P(N')$ inside of $P(N)$ in
$$\revise{V}(N)/\revise{V}(N') = Q(N,N') = \Hom_{\cMatid^\wedge(E)}(N,N').$$
It is straightforward to check that the differential squares to zero.

\subsection{The determinant functor}
In this section, we define an additive functor $\Det$ from $\cMatid^\wedge(E)$ to the category $\Vec_\Q$.
For any matroid $M$ on $E$, we put
$$\Det(M) := \wedge^{d(M)} \revise{V(M)}^*.$$
If $P(M')\subset P(M)$ and $$\sigma\in \wedge^{d(M)-d(M')}(\revise{V(M)}/\revise{V(M')}) = L(M,M') = \Hom_{\cMatid^\wedge(E)}(M,M'),$$ then wedge product with $\sigma$
gives an isomorphism from $\wedge^{d(M')} \revise{V(M')}$ to $\wedge^{d(M)}\revise{V(M)}$.  Dualizing, we obtain our map
$$\Det(\sigma):\Det(M)\to\Det(M').$$

Let $\Phi:\cMatid(E)\to\cA$ be a categorical matroid invariant valued in a $\Q$-linear category $\cA$. There is a monoidal action of
$\Vec_\Q$ on $\cA$, which we denote with $\otimes$.
We define a functor $\Phi^\wedge:\cMatid^\wedge(E)\to\cA$ by putting $\Phi^\wedge := \Phi \otimes \Det$.
More precisely, we define $\Phi^\wedge(M) := \Phi(M)\otimes\Det(M)$ for all $M$, and 
given an element $$\sigma\in L(M,M') = \Hom_{\cMatid^\wedge(E)}(M,M'),$$
we put $$\Phi^\wedge(\sigma) := \Phi(\iota_{M,M'})\otimes\Det(\sigma):\Phi^\wedge(M)\to \Phi^\wedge(M').$$

\begin{proposition}\label{val exact}
If $\Phi:\cMatid(E)\to\cA$ is valuative and $\cN$ is a decomposition of a matroid $M$ on the ground set $E$, then $\Phi^\wedge(C_\bullet^\wedge(\cN))$ is contractible.
\end{proposition}

\begin{proof}
Let $\Omega$ be an orientation of $\cN$.  This induces an isomorphism $\Det(M)\cong\Q$, along with isomorphisms $\Det(N)\cong\Q$ for each $N\in\cN$.
These isomorphisms fit together to form an isomorphism of complexes 
$\Phi^\wedge(C_\bullet^\wedge(\cN))\cong \Phi(C_\bullet^\Omega(\cN))$, and the latter is contractible by definition of valuativity.
\end{proof}

Proposition \ref{val exact} has a corollary that allows us to generalize Equation \eqref{characters} to situations where it is not
possible to choose an orientation $\Omega$ in a way that is fixed by symmetries.  Fix an abelian group, and let $\cA$ be the category of finite dimensional vector spaces over $\Q$ that are graded by that group.

\begin{corollary}\label{the point}
Let $\cN$ be a decomposition of a matroid $M$ on the ground set $E$.  
Let $\Gamma$ be a finite group acting on $E$ preserving $\cN$.  
Let $\Phi:\cMat\to\cA$ be a valuative categorical invariant. 
Then we have an equation \begin{eqnarray*}\notag 0 = \chi^\Gamma\!\Big(\Phi(C_\bullet^\wedge(\cN))\Big) &:=&  \sum_k (-1)^k \,\Phi(C_k^\wedge(\cN))\\ &=& 
\sum_k(-1)^k \bigoplus_{N\in\cN_k}\Phi(N)\otimes\Det(N)\\ \notag &=& \sum_k (-1)^k \!\bigoplus_{N\in\cN_k/\Gamma} \Ind_{\revise{\Gamma_N}}^\Gamma\Phi(N)\otimes\Det(N)\end{eqnarray*}
of virtual graded representations of $\Gamma$.
\end{corollary}

\begin{remark}
If there exists an orientation $\Omega$ of $\cN$ that is fixed by the action of $\Gamma$, then for all $N\in\cN$, $\Det(N)$
is isomorphic to the 1-dimensional trivial representation of the stabilizer of $N$.  In particular, Equation \eqref{characters}
is a special case of Corollary \ref{the point}.
\end{remark}

\begin{example}\label{finally}
Let $\cN$ be the decomposition of $M$ in Examples \ref{octahedron 1} and \ref{octahedron 2}, 
and let $\Gamma = \fS_2$ act by swapping 1 with 3 and 2 with 4.  Applying Corollary \ref{the point} to the trivial categorical invariant $\tau$, we obtain the equation
$$0 = \Det(M) -\Det(N) - \Det(N') + \Det(N'')$$
of virtual representations of $\fS_2$.  Since $\fS_2$ acts on $P(M)$ in an orientation preserving way, $\Det(M)$ is the trivial representation.  Since $\fS_2$ acts on $P(N'')$ in an orientation reversing way, $\Det(N'')$ is the sign representation.
Finally, since $\fS_2$ swaps $N$ with $N'$, $\Det(N) \oplus \Det(N')$ is the regular representation.
\end{example}

\subsection{Equivariant relaxation}\label{sec:eq relax}
Let $M$ be a matroid of rank $k$ on the ground set $E$, equipped with an action of the group $\Gamma$.
Let $F$ be a stressed flat of rank $r$, and let $\cF := \{\gamma F\mid \gamma\in\Gamma\}$.  Let $\tM$ be the matroid obtained by relaxing $M$ with respect to every $G\in\cF$, 
as described in Section \ref{sec:relaxation}.  
Let $\GF\subset\Gamma$ denote the stabilizer of $F$.  Note that $\GF$ acts on the sets $F$ and $E\setminus F$,
and therefore on the matroids $\Pi_{r,k,\revise{F,E}}$ and $\Lambda_{r,k,\revise{F,E}}$ from Section \ref{sec:relaxation}.  

Fix an abelian group, and let $\cA$ be the category of finite dimensional vector spaces over $\Q$ that are graded by that group.
Let $\Phi:\cMat\to\cA$ be a valuative categorical invariant.

\begin{proposition}\label{virtual}
We have the following equality of virtual graded $\Gamma$-representations:
$$\Phi(\tM) = \Phi(M) + \Ind_{\GF}^{\Gamma}\Phi(\Lambda_{r,k,\revise{F,E}})
- \Ind_{\GF}^{\Gamma}\Phi(\Pi_{r,k,\revise{F,E}}).$$
\end{proposition}

\begin{proof}
If $M = \Pi_{r,k,\revise{F,E}}$, then $\tM = \Lambda_{r,k,\revise{F,E}}$ and the statement is trivial.  Assume now that this is not the case.
Let $\cN$ be the decomposition of $\tM$ described in Theorem \ref{relax symmetrically}.
By Corollary \ref{the point}, 
we have 
\begin{eqnarray*}(-1)^{d(\tM)}\Phi(\tM)\otimes\Det(\tM) &=& (-1)^{d(M)}\Phi(M)\otimes\Det(M)\\
&& +\;\; (-1)^{d(\Lambda_{r,k,\revise{F,E}})}\Ind_{\GF}^{\Gamma}\Phi(\Lambda_{r,k,\revise{F,E}})\otimes\Det(\Lambda_{r,k,\revise{F,E}})\\
&& +\;\;(-1)^{d(\Pi_{r,k,\revise{F,E}})} \Ind_{\GF}^{\Gamma}\Phi(\Pi_{r,k,\revise{F,E}})\otimes\Det(\Pi_{r,k,\revise{F,E}}).\end{eqnarray*}
To simplify this, we first note that $$d(\tM) = d(M) = d(\Lambda_{r,k,\revise{F,E}}) = d(\Pi_{r,k,\revise{F,E}})+1.$$
Next, we observe that
$$\revise{V(\tM)} = \revise{V(M)} = \revise{V}(\Lambda_{r,k,\revise{F,E}}) = \revise{V}(\Pi_{r,k,\revise{F,E}}) \oplus \Q\cdot x_F,$$
which implies that $\Det(\tM) = \Det(M)$ as representations of $\Gamma$.  Since $\GF$ fixes $x_F$, it also implies that
$\Det(\Lambda_{r,k,\revise{F,E}}) = \Det(\Pi_{r,k,\revise{F,E}}) = \Res^\Gamma_{\GF}\Det(M)$ as representations of $\GF$.
Thus
\begin{eqnarray*}(-1)^{d(M)}\Phi(\tM)\otimes\Det(M) &=& (-1)^{d(M)}\Phi(M)\otimes\Det(M)\\
&& +\;\; (-1)^{d(M)}\Ind_{\GF}^{\Gamma}\Big(\Phi(\Lambda_{r,k,\revise{F,E}})\otimes\Res^{\Gamma}_{\GF}\Det(M)\Big)\\
&& -\;\;(-1)^{d(M)}\Ind_{\GF}^{\Gamma}\Big(\Phi(\Pi_{r,k,\revise{F,E}})\otimes\Res^{\Gamma}_{\GF}\Det(M)\Big)\\
&=& (-1)^{d(M)}\Phi(M)\otimes\Det(M)\\
&& +\;\; (-1)^{d(M)}\Big(\Ind_{\GF}^{\Gamma}\Phi(\Lambda_{r,k,\revise{F,E}})\Big)\otimes\Det(M)\\
&& -\;\;(-1)^{d(M)}\Big(\Ind_{\GF}^{\Gamma}\Phi(\Pi_{r,k,\revise{F,E}})\Big)\otimes\Det(M).
\end{eqnarray*}
Dividing both sides by $(-1)^{d(M)}\Det(M)$, which is an invertible element of the ring of virtual graded representations of $\Gamma$, 
yields the statement of the proposition.
\end{proof}

\subsection{Relaxing a stressed hyperplane}
We conclude by applying Proposition \ref{virtual} to the functors $\OS$ and $\uKL$ in the special case where the stressed flat $F=H$
is a hyperplane, meaning that $r=k-1$.  In this case, we have
$$\Pi_{k-1,k,\revise{H,E}} = U_{1,E\setminus H} \sqcup U_{k-1,H}.$$
The matroid $\Lambda_{k-1,k,\revise{H,E}}$ obtained by relaxing $\Pi_{k-1,k,\revise{H,E}}$ with respect to the stressed hyperplane $H$ has as its
bases those subsets $B\subset E$ of cardinality $k$ with $|B\cap H| \geq k-1$.  It has no loops, and its simplification is the uniform matroid of rank $k$ on
the ground set $\bar{H} := H\sqcup\{*\}$ obtained from $E$ by identifying all of the elements of $E\setminus H$.

The group $\Aut(E\setminus H)\times\Aut(H)$ acts on the matroids $\Pi_{k-1,k,\revise{H,E}}$ and $\Lambda_{k-1,k,\revise{H,E}}$, and therefore on their Orlik--Solomon algebras.
Since all of the elements of $E\setminus H$ are parallel in both $\Pi_{k-1,k,\revise{H,E}}$ and $\Lambda_{k-1,k,\revise{H,E}}$, the actions on the Orlik--Solomon algebra factor through the projection to $\Aut(H)$.

For the rest of this section, we will write $h=|H|$ and identify $H$ with $\{1,\ldots,h\}$.  If $\lambda$ is a partition of $h$, we write $V_\lambda$ to denote the
corresponding Specht module for $\mathfrak{S}_h = \Aut(H)$. Recall that $V_{[h]}$ is the trivial representation.

\begin{lemma}\label{OS correction}
We have
$$\OS(\Lambda_{k-1,k,\revise{H,E}})
- \OS(\Pi_{k-1,k,\revise{H,E}}) = 
\left(\wedge^{k-1}\,V_{[h-1,1]}\right)\otimes\Big(\Q(1-k) \oplus  \Q(-k)\Big)$$
as graded virtual representations of $\fS_h$.
\end{lemma}

\begin{proof}
The matroids $\Lambda_{k-1,k,\revise{H,E}}$ and $\Pi_{k-1,k,\revise{H,E}}$ have the same independent sets of cardinality less than $k$, so their Orlik--Solomon algebras
differ only in degrees $k-1$ and $k$.  Let us consider the degree $k$ part first.  

The Orlik--Solomon algebra of a direct sum is isomorphic
to the tensor product of the Orlik--Solomon algebras, and the degree $k-1$ part of the Orlik--Solomon algebra of $U_{k-1,H}$ is isomorphic to $\wedge^{k-2} V_{[h-1,1]}$,
thus
$$\begin{aligned}\OS^k(\Pi_{k-1,k,\revise{H,E}}) &\cong \OS^1(U_{1,E\setminus H}) \otimes \OS^{k-1}(U_{k-1,H})\\
&\cong V_{[h]} \otimes \wedge^{k-2} \,V_{[h-1,1]}\\
&\cong \wedge^{k-2}\, V_{[h-1,1]}.\end{aligned}$$
The Orlik--Solomon algebra is unaffected by simplification, thus
$$\begin{aligned}\OS^k(\Lambda_{k-1,k,\revise{H,E}}) &\cong \Res^{\fS_{h+1}}_{\fS_h}\OS^k(U_{k,\bar H})\\
&\cong \Res^{\fS_{h+1}}_{\fS_h} \wedge^{k-1} \, V_{[h,1]}\\
&\cong \wedge^{k-1} \Res^{\fS_{h+1}}_{\fS_h} V_{[h,1]}\\
&\cong \wedge^{k-1} \left(V_{[h-1,1]}\oplus V_{[h]}\right)\\
&\cong \wedge^{k-1}\,V_{[h-1,1]} \oplus \wedge^{k-2}\, V_{[h-1,1]}.
\end{aligned}$$
This proves that the lemma is correct in degree $k$.
The statement in degree $k-1$ follows from the fact that, for any matroid $M$ on a nonempty ground set, 
$\sum (-1)^i \OS^i(M) = 0$ as a virtual representation of the automorphism group of $M$.
\end{proof}

Let $\Gamma_H := \Gamma \cap \Aut(E\setminus H)\times\Aut(H)$ be the stabilizer of $H$, which maps via projection to $\fS_h=\Aut(H)$. 
We write $\Res_{\GH}^{\fS_h}$ for the functor that pulls back a representation along the homomorphism from $\GH$ to $\fS_h$.
Proposition \ref{virtual} and Lemma \ref{OS correction} have the following immediate corollary.

\begin{corollary}\label{OS relax}
Let $M$ be a matroid equipped with an action of the group $\Gamma$.  
Let $H$ be a stressed hyperplane of $M$, and let $\tM$ be the matroid obtained from $M$ by relaxing $\gamma H$ for all $\gamma\in\Gamma$.
Then 
$$\OS(\tM) = \OS(M)
+
\left(\Ind_{\GH}^{\Gamma}\Res_{\GH}^{\fS_h}\wedge^{k}\,V_{[h-1,1]}\right)\otimes\Big(\Q(1-k) \oplus  \Q(-k)\Big)
$$
as virtual graded representations of $\Gamma$.
\end{corollary}

We now turn our attention to the functor $\uKL$.
Recall that, if $M$ is a matroid equipped with an action of a group $\Gamma$, then the 
coefficient of $t^i$ in the $\Gamma$-equivariant Kazhdan--Lusztig polynomial $P_M^\Gamma(t)$ is equal
to the virtual $\Gamma$-representation $\sum_j (-1)^j \KL^{i,j}(M)$.
Given a skew partition $\lambda/\mu$ of $h$, we write $V_{\lambda/\mu}$ to
denote the corresponding skew Specht module for $\mathfrak{S}_h$.

\begin{lemma}\label{KL correction}
If $k>1$, then we have
$$P^{\fS_h}_{\Lambda_{k-1,k,\revise{H,E}}}(t)
- P^{\fS_h}_{\Pi_{k-1,k,\revise{H,E}}}(t)
= \sum_{0<i<k/2} V_{[h-2i+1,(k-2i+1)^i]/[k-2i,(k-2i-1)^{i-1}]}\, t^i.
$$
\end{lemma}

\begin{proof}
Equivariant Kazhdan--Lusztig polynomials of matroids are multiplicative under direct sums and 
the equivariant Kazhdan--Lusztig polynomial of a rank 1 loopless matroid is the trivial representation in degree zero, hence
$$P^{\fS_h}_{\Pi_{k-1,k,\revise{H,E}}}(t) = P^{\fS_h}_{U_{k-1,H}}(t)
= \sum_{0\leq i<(k-1)/2} V_{[h-2i,(k-2i)^i]/[(k-2i-2)^i]}\, t^i,$$
where the second equality is proved in \cite[Theorem 3.7]{eq-inverse-kl}.  On the other hand, equivariant Kazhdan--Lusztig polynomials of loopless matroids are
unchanged by simplification, hence 
$$P^{\fS_h}_{\Lambda_{k-1,k,\revise{H,E}}}(t) = P^{\fS_h}_{U_{k,\bar{H}}}(t) = 
\sum_{0\leq i<k/2} \Res_{\fS_h}^{\fS_{h+1}}V_{[h-2i+1,(k-2i+1)^i]/[(k-2i-1)^i]}\, t^i,$$
where the last equality again follows from \cite[Theorem 3.7]{eq-inverse-kl}.
By \cite[Lemma 4.1]{KNPV} and \cite[Proposition 2.3.5, Lemma 2.3.12]{kleshchev}, we may rewrite this as
$$P^{\fS_h}_{\Lambda_{k-1,k,\revise{H,E}}}(t) =\sum_{0\leq i<(k-1)/2}V_{[h-2i,(k-2i+1)^i]/[(k-2i-1)^i]}\, t^i + \sum_{0<i<k/2} V_{[h-2i+1,(k-2i+1)^i]/[k-2i,(k-2i-1)^{i-1}]}\, t^i.$$
The lemma is obtained by taking the difference
between the our expressions for $P^{\fS_h}_{\Pi_{k-1,k,\revise{H,E}}}(t)$ and $P^{\fS_h}_{\Lambda_{k-1,k,\revise{H,E}}}(t)$.
\end{proof}

Proposition \ref{virtual} and Lemma \ref{KL correction} have the following immediate corollary, which was first proved in \cite[Theorems 1.3 and 1.4]{KNPV}.

\begin{corollary}\label{KL relax}
Let $M$ be a matroid of rank $k>1$ equipped with an action of the group $\Gamma$.  
Let $H$ be a stressed hyperplane of $M$, and let $\tM$ be the matroid obtained from $M$ by relaxing $\gamma H$ for all $\gamma\in\Gamma$.
Then
$$P_{\tM}^\Gamma(t) = P_M^\Gamma(t)  + \bigoplus_{0<i<k/2}
\Ind_{\GH}^{\Gamma}\Res_{\GH}^{\fS_h}V_{[h-2i+1,(k-2i+1)^i]/[k-2i,(k-2i-1)^{i-1}]}\, t^i.$$
\end{corollary}

\bibliography{./symplectic}
\bibliographystyle{amsalpha}

\end{document}